\documentclass[a4paper,12pt]{article}
\usepackage[utf8]{inputenc}
\usepackage[english]{babel}
\usepackage{graphicx}
\usepackage{amscd}
\usepackage{amsfonts}
\usepackage{latexsym}
\usepackage{amssymb,amsmath}
\usepackage[usenames]{color}
\usepackage{caption}
\usepackage{subcaption}
\usepackage{amscd}
\usepackage{psfrag}
\usepackage{hyperref}

\usepackage{amsthm}

\theoremstyle{plain}
\newtheorem{teor}{Theorem}[section]

\newtheorem{cor}{Corollary}[section]
\newtheorem{prop}[teor]{Proposition}

\newtheorem{lemma}{Lemma}[section]

\newtheorem{remark}{Remark}[section]

  \newcommand*{\Log}{{\operatorname{Log}}}
\theoremstyle{definition}
\newtheorem{defi}[teor]{Definition}

  \newcommand*{\e}{\ensuremath{{\operatorname{e}}}}
\def\BB{{\mathcal{B}}}
\def\cz{{c_0}}
\def\C{{\mathbb{C}}}
\newcommand{\Chat}{{\widehat\C}}
\def\CC{\mathcal{C}}

\def\S{\mathbb{S}}
\def\D{\mathbb{D}}
\def\Dbar{{\overline{\mathbb{D}}}}
\def\HH{\mathcal{H}}
\def\R{\mathbb{R}}
\def\S{{\mathbb{S}}}
\def\Sen{{\S^1}}

\def\Z{\mathbb{Z}}

\def\oo{{o}}
\def\OO{\mathcal{O}}

\def\H{\mathbb{H}}
\def\ZZ{\mathcal{Z}}
\def\chHH{{\check{H}}}
\def\chZZ{{\check{\mathcal{Z}}}}
\def\whHH{{\widehat{H}}}

\def\M{{\mathcal{M}}}
\def\chM{{\check{\M}}}
\def\Mpq{{\M_{p/q}}}
\def\Mpqr{{\M^*_{p/q}}}
\def\Mppq{{\M_{p'\!/q}}}
\def\Mppqr{{\M^*_{p'\!/q}}}

\newcommand{\QG}{{\mathcal{QG}}}
\newcommand{\QM}{{\mathcal{Q\!M}}}
\def\iln{{\kappa}}
\def\ilnz{{\kappa_0}}

\def\l0{_{\lambda_0}}
\def\i0{_{\iota_0}}
\def\la{_{\lambda}}
\newcommand*{\mapfromto}[3]{\hbox{\ensuremath{#1 : #2 \longrightarrow #3}}}

\def\al{\alpha}
\def\ga{\gamma}

\def\nf{{\widehat f}}
\def\ng{{\widehat g}}

\def\nn{_{\nu}}
\def\n0{_{\nu_0}}
\def\ov{\overline}

\def\chsi{{\check{\sigma}}}

\def\pot{\mathcal{N}}
\def\potz{{\pot_0}}
\def\poto{{\pot_1}}

\def\sm{\setminus}

\def\whh{{\widehat h}}
\def\whPhi{{\widehat \Phi}}
\def\whR{{\widehat R}}
\def\whResit{{N\!RI}} 
\def\whsi{{\widehat\sigma}}

\def\whV{{\widehat V}}

\def\wtal{{\widetilde \alpha}}
\def\wtA{{\widetilde A}}
\def\wtbe{{\widetilde \beta}}
\def\wtf{{\widetilde f}}
\def\wtl{{\widetilde l}}

\def\wtR{{\widetilde R}}
\def\wtw{{\widetilde w}}
\def\wtPsi{{\widetilde \Psi}}
\def\wtx{{\widetilde x}}
\def\wty{{\widetilde y}}
\def\wtX{{\widetilde X}}
\def\wtY{{\widetilde Y}}
\def\wtz{{\widetilde z}}
\def\whE{{\widehat E}}
\def\whBB{{\widehat\BB}}

\newcommand{\Arg}{{\operatorname{Arg}}}
\newcommand{\thmref}[1]{Theorem~\ref{#1}}
\newcommand{\corref}[1]{Corollary~\ref{#1}}
\newcommand{\subsecref}[1]{Subsection~\ref{#1}}
\newcommand{\subsubsecref}[1]{Subsubsection~\ref{#1}}
\usepackage{color} 
\title{On qc compatibility of satellite copies of the Mandelbrot set: II}
\author{Luna Lomonaco, Carsten Lunde Petersen}

\begin{document}
\maketitle
\begin{abstract}

The Mandelbrot set $\M$ contains infinitely many small copies of itself, each canonically homeomorphic to $\M$ via the Douady--Hubbard theory of polynomial-like maps. These copies come in two kinds: \emph{primitive} copies, whose principal hyperbolic component carries a cusp at its root, and \emph{satellite} copies, whose boundary is smooth at the root. Douady and Hubbard conjectured that the straightening maps of analytic families of polynomial-like maps are quasiregular \cite[p.~332]{DH}, predicting that primitive copies are quasiconformally homeomorphic to $\M$ and de-rooted satellite copies to $\M\setminus\{1/4\}$ --- hence that satellite copies are mutually quasiconformally homeomorphic away from their roots. Lyubich proved the primitive case, and showed that satellite copies are quasiconformally homeomorphic to $\M$ outside every neighbourhood of the root \cite{Ly}. Whether the homeomorphisms between satellite copies are quasiconformal at the roots remained open. In \cite{LP} we gave a negative answer: satellite copies $\M_{p/q}$ and $\M_{p'/q'}$ with $q \neq q'$ are not quasiconformally homeomorphic, disproving the conjecture; and we conjectured that copies whose rotation numbers share the same denominator are quasiconformally homeomorphic. In the present paper we prove that they are. Together, these results yield a complete geometric classification: two satellite copies of $\M$ are quasiconformally equivalent if and only if their rotation numbers have the same denominator. This settles the quasiconformal geometry of the small copies of the Mandelbrot set.
\end{abstract}

\section{Introduction}
For a polynomial on the Riemann sphere, infinity is a (super) attracting fixed point, and the filled Julia set 
consists of points with bounded orbit. 
In particular, for each member of the quadratic family $Q_c(z) = z^2 + c,\,\,c \in \C$, the point at infinity is fixed and (super) attracting, with basin of attraction $A_c(\infty)$, and we can define the filled Julia set $K_c$ of $Q_c$ as $K_c := \widehat \C \setminus A_c(\infty)$. 
The Mandelbrot set $\M$  is the \textit{connectedness locus} for the family $Q_c(z) = z^2 + c$.
 That is, $\M$ consists of all parameters $c$ for which the corresponding filled Julia set $K_c$ is connected. 
 By a theorem of Fatou, the filled Julia set is connected if and only if it contains all the (finite) critical points of the polynomial.
 Since each map $Q_c$ has a unique finite critical point at $z=0$ we have 
$$\M:=\{c \in \C\,|\,Q_c(0)^n \nrightarrow \infty \mbox{ as } n\rightarrow \infty \}.$$

Computer experiments quickly reveal the existence of small copies of $\M$ inside itself.
Each little copy $\M'$ is either \textit{primitive}, with a cusp at the root of its main hyperbolic component, or \textit{satellite}, with the root on the boundary of a lower period hyperbolic component and with the boundary of the main hyperbolic component smooth across the root. For a more accurate definition of primitive and satellite copies of the Mandelbrot set, see Section \ref{DH}. 
Douady and Hubbard defined a natural, dynamics-preserving map $\chi_{\M'}: \M' \rightarrow \M$, and proved that it is a homeomorphism whenever $\M'$ is a primitive copy. In the satellite case, they proved that $\chi_{\M'}$ is a homeomorphism between complements of matching neighbourhoods of the respective roots (\cite{DH}).
Haissinsky later extended the result for satellite copies and proved that the Douady--Hubbard map $\chi_{\M'}$ is a homeomorphism in the satellite case as well (\cite{Hi}). 

Primitive copies of the Mandelbrot set are visually similar to the full Mandelbrot set itself, while satellite copies resemble one half of the logistic Mandelbrot set. This naturally leads to the question of whether the homeomorphisms constructed by Douady and Hubbard possess stronger regularity properties.
Sullivan showed that the appropriate notion of smoothness is given by \textit{quasiconformality}. A conformal map preserves infinitesimal shapes: its differential exists everywhere and is complex linear, so it maps infinitesimal circles to circles. A quasiconformal map, by contrast, is an orientation-preserving homeomorphism that is ``almost conformal'' in the sense that its differential exists almost everywhere and is real linear, and maps infinitesimal circles to ellipses of bounded eccentricity. The distortion is quantified by a constant that bounds the ratio of the major to minor axes of the image ellipses
(for a precise definition, see Section \ref{qceqs}).
In fact, Douady and Hubbard conjectured a stronger statement: for any complex analytic family $\mathbf f=(f_\lambda)_{\lambda\in\Lambda}$ of polynomial-like mappings of degree $2$ with $\Lambda\approx\D$, the straightening map $\chi_{\mathbf f}$ factors on $\M_{\mathbf f}$ as $\chi_{\mathbf f}=\varphi_2\circ\varphi_1$, with $\varphi_1:\Lambda\rightarrow\D$ quasiconformal and $\varphi_2:\D\rightarrow\C$ holomorphic \cite[p.~332]{DH}; in other words, $\chi_{\mathbf f}$ is \textit{quasiregular}. Since an injective quasiregular map is quasiconformal, on a primitive copy $\M'$ of the Mandelbrot set, or on a de-rooted satellite copy $\M'\setminus\{\mathrm{root}\}$ --- the map at the root of a satellite copy admits no polynomial-like restriction, so the root belongs to no parameter space $\Lambda$ --- the conjecture predicts precisely that the straightening homeomorphism is quasiconformal --- onto $\M$ in the primitive case, onto $\M\setminus\{1/4\}$ in the satellite case; in particular, satellite copies are predicted to be mutually quasiconformally homeomorphic away from their roots. Nor could the root of a satellite copy be included: no quasiconformal map straightens a cusp, so no satellite copy, whose boundary is smooth at the root, is quasiconformally homeomorphic to $\M$, which carries a cusp at its root $c=1/4$. Between two satellite copies, however, whose boundaries are both smooth at the roots, no such obstruction is present, and whether the mutual homeomorphisms are quasiconformal at the roots becomes the genuine question.
Lyubich proved (\cite{Ly}) that primitive copies of $\M$ are indeed \textit{quasiconformally} homeomorphic to $\M$, and that satellite copies are quasiconformally homeomorphic to $\M$ outside any neighbourhood of their root; he moreover suggested that ``presumably'' every satellite copy is quasiconformally equivalent to one half of the logistic Mandelbrot set --- the satellite copy with $q=1$ --- which would make all satellite copies mutually quasiconformally equivalent, an expectation which, unlike the Douady--Hubbard conjecture, includes the roots (see the Remark on p.~366 of \cite{Ly}).
The situation around the roots of satellite copies turned out to be more subtle than these expectations anticipated: in~\cite{LP} we showed that satellite copies with rotation numbers having different denominators are not quasiconformally homeomorphic --- so that, in particular, the Douady--Hubbard conjecture fails, and no satellite copy $\Mpq$ with $q>1$ is quasiconformally homeomorphic to one half of the logistic Mandelbrot set ($q=1$). At the same time, we conjectured there that the denominator is the only obstruction \cite[Conjecture~1]{LP}: satellite copies whose rotation numbers share the same denominator should be quasiconformally homeomorphic.
In this paper we prove precisely this, thereby completing the quasiconformal classification of the satellite copies of $\M$: any two satellite copies $\Mpq$ and $\Mppq$ with rotation numbers of the same denominator $q$ are quasiconformally homeomorphic.

More precisely, recall that the Mandelbrot set $\M$ is the connectedness locus of the family $Q_c(z)=z^2+c$, parametrized by the critical value $c$. The \textit{logistic} Mandelbrot set, also denoted by $\M$, is the connectedness locus of the family $P\la(z)=\lambda z + z^2$, parametrized by the multiplier $\lambda$ of the $\alpha$-fixed point $z=0$. 
The map $c(\lambda) = \lambda/2-\lambda^2/4$ is a double branched cover from $\C$ to $\C$, branched over the point $1/4$, which is the root of the standard Mandelbrot set $\M$. It maps the logistic Mandelbrot set $2:1$ onto the (standard) Mandelbrot set, and it is conformal except above the branch point.
Let $\Mpq$ denote the satellite copy of the (logistic) Mandelbrot set bifurcating from the main hyperbolic component $\D$ at $\lambda_0=e^{2\pi i p/q}$. We call $\lambda_0=e^{2\pi i p/q}$ the \textit{root} of the satellite copy $\Mpq$.
Now, let $\Mpq$ and $\M_{P/Q}$ be satellite copies with respective roots at $\lambda= e^{2\pi i p/q}$ and $\nu=e^{2\pi i P/Q}$.
Then the map
$$\xi_{p/q,P/Q}:=\chi_{P/Q}^{-1}\circ \chi_{p/q}: \Mpq \rightarrow \M_{P/Q},$$
is a dynamical homeomorphism that is quasiconformal outside any neighbourhood of the root. 
In this paper, we prove that:
\begin{teor}\label{main}
For $p/q$ and $p'/q$ irreducible rationals, the induced Douady-Hubbard homeomorphism
$$\xi_{p/q,p'/q}:=\chi_{p'/q}^{-1}\circ \chi_{p/q}: \Mpq \rightarrow \M_{p'/q},$$
extends to a quasiconformal homeomorphism between neighbourhoods of the two sets.
\end{teor}
Theorem \ref{main} follows from Theorem \ref{Uniformly_qc_conjugacies} combined with Theorem \ref{Harvest} and Theorem \ref{nbh}.
\begin{cor}\label{cor:general}
    For any pair $H, H'$ of hyperbolic components in $\M$ of periods $Q, Q'\geq 1$ respectively and for any pair of relative satellite copies 
    $\M_{p/q}$ and $\M'_{p'/q}$ with periods $Qq$ and $Q'q$  
    (with roots on $\partial H$ and $\partial H'$) respectively 
    the Douady-Hubbard straightening homeomorphism $\chi_{\M'_{p'/q}}^{-1}\circ\chi_{\Mpq}$ extends to a quasiconformal homeomorphism of neighbourhoods.
\end{cor}
\begin{proof}
    The copies of $\M$ in $\M$ with main hyperbolic components $H$ and $H'$ are 
    quasiconformally homeomorphic to $\M$ except possibly for neighbourhoods of the roots. 
    Thus $\M_{p/q}$ and $\M'_{p'/q}$ are quasiconformally homeomorphic to the corresponding 
    $p/q$ and $p'/q$ satellites copies of the main hyperbolic components and so the result 
    follows from the main theorem.
\end{proof}

\begin{cor}\label{cor:classification}
Two satellite copies of $\M$ are quasiconformally homeomorphic if and only if their rotation numbers have the same denominator. In particular, Theorem \ref{main} proves Conjecture~1 of \cite{LP}.
\end{cor}
\begin{proof}
If the denominators agree, the copies are quasiconformally homeomorphic by Corollary \ref{cor:general}.
Conversely, suppose two satellite copies with rotation numbers of denominators $q\neq q'$ were quasiconformally homeomorphic. By Corollary \ref{cor:general} each is quasiconformally homeomorphic to the corresponding satellite copy of the main hyperbolic component, so the satellite copies of the main hyperbolic component with denominators $q$ and $q'$ would be quasiconformally homeomorphic, contradicting the main theorem of \cite{LP}.
\end{proof}
\begin{cor}
Any satellite copy $\Mpq$ of $\M$ admits an anti-quasiconformal involution. 
\end{cor}
\begin{proof}
    Note that $\lambda\mapsto\overline{\lambda}$ maps $\C$ to $\C$ and $\Mpq$ onto $\M_{(q-p)/q}$. 
    Then apply the main theorem to $\M_{(q-p)/q}$ and $\Mpq$.
\end{proof}

\subsubsection{Strategy of the proof}\label{sec:strategy}

The presence of homeomorphic copies of the Mandelbrot set within itself, and
in many other parameter spaces, is explained by the Douady--Hubbard theory of
polynomial-like maps, whose definitions and basic properties are recalled in
Section~\ref{pre}. 
Recall the dynamics-preserving Douady--Hubbard homeomorphisms
$\chi_{p/q}\colon \M_{p/q}\to\M$ and the induced homeomorphism
$\xi = \xi_{p/q,p'/q} := \chi_{p'/q}^{-1}\circ\chi_{p/q}$. Since $\chi$, and
hence $\xi$, is holomorphic on every hyperbolic component (see
\cite[Ch.~8.2]{CG}), it suffices to prove that the restriction
$\xi\colon \M^*_{p/q}\to\M^*_{p'/q}$ between the complements of the
respective main hyperbolic components is quasiconformal, i.e.\ locally the
restriction of a $K$-quasiconformal map with $K = K(q)$,
and then to extend $\xi$ quasiconformally off the
copy.

\medskip

For each fixed $\lambda\in\M^*_{p/q}$, the
definition of $\xi$ implies the existence of a hybrid equivalence between the quadratic-like
restrictions $f_\lambda = P^q_\lambda|$ and
$g_{\nu} = P^q_{\nu}|$, $\nu=\xi(\lambda)$. The content of the theorem is
therefore not existence but uniformity. All existing
proofs of uniformity, both globally and locally on little copies, are
based on lower bounds on moduli of the fundamental annuli. 
However as $\lambda$ approaches
the root $\lambda_0$, the $q$-cycle collides with the $\alpha$-fixed
point, so the fundamental annuli of the restrictions degenerate.

\medskip

The proof runs as follows. For $\lambda\in\M^*_{p/q}$
and $\nu = \xi(\lambda)$ we construct quasiconformal conjugacies between
the dynamical planes of the corresponding restrictions
$f_\lambda = P^q_\lambda|$ and $g_\nu = P^q_\nu|$ 
around the respective critical values of 
$P_\lambda$ and $P_\nu$, with dilatation bounded
uniformly in $\lambda$ (Section \ref{dyn}), and we then transfer this bound to
parameter space (Section \ref{L}). The transfer takes place in Lyubich's space
of quadratic-like germs, where $\xi$ appears as a holonomy; but this space
is infinite-dimensional, and dilatations cannot be computed there. The
notion of pre-exterior quasiconformal equivalence, which we introduce in
this paper (Definition \ref{pre-exterior}) --- a global quasiconformal map between the
exteriors of the domains of two polynomial-like maps, fixing $\infty$,
whose boundary extension conjugates the dynamics on the boundary curves
--- allow us to uniformize part of the germ
space to the complex plane, and hence to transfer the bound from the hybrid equivalence to the induced holonomy. Finally, $\xi$ is extended across the
root (Sections~\ref{wake}--\ref{dynpar}).

\medskip

The dynamical constructions occupy Section \ref{dyn}. We call two degree
$2$ polynomial-like maps \emph{corresponding} if they have connected filled
Julia sets and are hybrid equivalent (so that the map $\xi$ sends one to the
other). Note that the little Julia sets are separated 
from the $q-1$ other little Julia sets in their forward orbit by pairs of rays co-landing on 
the $\alpha$-fixed points of $P_\lambda$ and $P_\nu$ respectively. For $\lambda\in\M^*_{p/q}$ and $\nu=\xi(\lambda)$ we first construct a
pre-exterior quasiconformal equivalence $\varphi_\lambda$
(Definition \ref{pre-exterior}, introduced above) between the
corresponding restrictions $f_\lambda$ and $g_\nu$, with dilatation bounded
uniformly in $\lambda$ (Theorem \ref{Uniformly_qc_conjugacies}).
The construction has three parts. We first
construct uniformly quasiconformal maps between the exterior sectors (containing the $q-1$ other little Julia sets) of corresponding polynomials  so that they interpolate the equipotential maps between the bounding rays (Section \ref{between rays}). Then we construct uniformly quasiconformal interpolation in neighbourhoods of the $q-1$ points of the cycle, that is the points which do not play the role of $\alpha$-fixed point of the quadratic-like restriction (Section \ref{2}), and finally these interpolations are glued into the family $\varphi_\lambda$ of pre-exterior uniformly quasiconformal
equivalences between the polynomial-like restrictions (Lemma \ref{O}, Theorem
\ref{Uniformly_qc_conjugacies}). Let us give a brief overview of the first two steps. 

We start with the interpolation of the equipotential map on the sector 
between the bounding rays and not containing the respective polynomial-like little Julia sets (Section \ref{between rays}). For this, we first lift the region between the two rays to a bi-infinite strip, which a
further conformal map straightens to a horizontal strip of some height
$h_\lambda$, with boundary values induced by the B\"ottcher
conjugation (as, for example, in \cite{BF}). Then we prove that the interpolation, which is affine on vertical lines $\Phi_\lambda$ on the horizontal strips
is uniformly quasiconformal (see Theorem \ref{intrays}). The uniform bound on the quasiconformal 
dilatation of $\Phi_\lambda$ is obtained in three steps: the heights $h_\lambda$ are bounded above and below through Bers' inequality and the geometry near the root of the principal logarithm
$\Lambda = \mathrm{Log}\,\lambda^q$ (Lemma \ref{1sulambda}, Proposition \ref{2}); the
boundary derivatives are controlled by Koebe distortion applied to
univalent maps commuting with the translation (Lemma \ref{cpthatzeta}, Proposition
\ref{P1}); and the difference between the horizontal
displacements for corresponding maps is bounded using the inequality
$|1/\Lambda - 1/N| < D$, where $N = \mathrm{Log}\,\nu^q$ (Proposition \ref{3},
using \cite[Thm~1.4]{K1} and \cite[Prop.~5]{LP}).

Then we interpolate in neighbourhoods
of the $q-1$ points of the $q$-cycle other than the $\alpha$-fixed point of the
restriction (Section \ref{2}). This part is heavily based on parabolic implosion (and we include Subsection \ref{nutshell} called \textit{Parabolic implosion in a nutshell}
in order to facilitate the reader’s understanding). The maps $f\la$ and $g_{\nu(\lambda)}$ are compared in the rectifying coordinates of
the Buff vector field, which serve as pre- Douady--Fatou
coordinates. In these coordinates, each cycle point acquires a ``butterfly''
domain, and the interpolation is obtained by constructing explicit restrictions of these domains with suitable geometry (see Lemmas \ref{fundamental_segment_bound},
\ref{2.8} and \ref{estimates}), and hence a uniformly quasisymmetric map between these restrictions (Proposition \ref{buffext}).
Essential ingredients of this construction are the Buff vector field, a uniform bound on the difference between Douady-Fatou coordinates for the rectified Buff field (the rectified Buff field plays the role of 'pre-Fatou coordinates') and the identity (it is called the Quasi-identity theorem in \cite{PZ} and can also be found in \cite{Ch}), and the bound $|1/\Lambda - 1/N| < D$, where $N = \mathrm{Log}\,\nu^q$ (that we recall is based on \cite[Thm~1.4]{K1} and \cite[Prop.~5]{LP})

\medskip

Section \ref{L} converts Theorem \ref{Uniformly_qc_conjugacies} into parameter-space bounds.
The conversion takes place in the uniformization of a part of Lyubich's space of
quadratic-like germs. The space of quadratic-like germs ( \cite{Ly}) is an infinite-dimensional space, whose
complex-analytic structure is induced by Banach slices, and whose fibers
--- the sets of germs sharing a given external class --- meet the
connectedness locus in homeomorphic copies of the Mandelbrot set
(Corollary 3.1, after \cite{Ly}). The map $\xi$ is read there as the
holonomy along hybrid classes between two copies.
In this space, however, there is no {\`a}priori direct way to compute the dilatation of the
holonomy: in \cite[Lemma 4.19]{Ly} Lyubich bounds the quasi-conformal dilation of the hybrid holonomy between two transversals by an implicit function of lower bounds on the moduli of fundamental annuli.
But no such lower bound survives in our setting,
where the fundamental annuli degenerate as $\lambda$ approaches the
root.
So, to estimate the dilatation, we do something quite different.
 We introduce the notion of 
 pre-exterior quasiconformal equivalences, and using pre-exterior equivalences between hybrid equivalent degree $2$ polynomial-like maps with different external classes, as well as pre-exterior equivalences between degree $2$ polynomial-like maps with the same external class but different internal one, we lift neighbourhoods of points within satellite copies of the Mandelbrot set from the Lyubich space of quadratic-like germs to subsets of $\C$. 
We then construct a holomorphic motion between these lifted neighborhoods and bound its dilatation in terms of the dilatation of the pre-exterior equivalence between the hybrid equivalent degree $2$ polynomial-like maps (see Theorem \ref{Uniformly_qc_copies}). 
This enables us to bound the dilatation of the homeomorphism $\xi$ between neighbourhoods of points in satellite copies in terms of the dilatation of the quasiconformal conjugacy between these points in their dynamical planes (see Theorem \ref{Harvest}).

\medskip

Theorem~\ref{main} asserts more: $\xi$ extends to a
quasiconformal homeomorphism between neighbourhoods of the copies. Off the copy, inside the wake,
the extension is constructed in Section~\ref{wake}, via
B\"ottcher coordinates and a (classical) holomorphic motion argument. On the main hyperbolic component, the extension is 
the rotation
$\lambda\mapsto\lambda\nu_0/\lambda_0$; in the coordinate $\Lambda$ the
rotation becomes the identity. 
There remains a neighbourhood of the root outside the wake and outside
$\overline{\D}$. There we define $\xi$ by quasiconformal interpolation, 
which in the coordinate $\Lambda$ is
between the identity on the imaginary axis and the potential-preserving
boundary map along the boundary rays.
This interpolation is constructed in two steps:  first we define a quasiconformal map between the ``dynamical cusps'' of the
root maps $f_{\lambda_0}$ and $g_{\nu_0}$
(Section~\ref{dc}, Figure~\ref{Dyn});
and then a quasiconformal map between the dynamical planes and the parameter
planes (Section~\ref{dynpar}, Figure~\ref{All}), through which the
interpolation of step (1) is transported to parameter space.
The first step uses the parabolic setup of the root
maps, while the second, which occupies most of Section~\ref{dynpar}, rests on parabolic implosion.

\medskip

Theorem~\ref{main} then follows by combining
Theorems \ref{Uniformly_qc_conjugacies}, \ref{Harvest} and ~\ref{nbh}).

\medskip

\noindent\textbf{Remark.} The quantity
$|1/\Lambda - 1/N|$ may be regarded as the discriminant of the
classification. For rotation numbers sharing the denominator it is
uniformly bounded on $\M^*_{p/q}$ --- by Kapiamba's optimal near-parabolic
Yoccoz inequality (\cite[Thm~1.4]{K1}; see also \cite[Thm~5.7]{K})
combined with \cite[Prop.~5]{LP} --- and this bound enters the present
proof at Proposition \ref{3} and Lemma \ref{rhos}. For $q \neq q'$ it fails, and its
failure underlies the obstruction of \cite{LP}.

\begin{remark}
Let $\Mpq$ and $\M_{p'/q}$ be satellite copies of the Mandelbrot set with root points $\lambda= e^{2\pi i p/q}$ and $\nu=e^{2\pi i p'/q}$ respectively. Consider the dynamical homeomorphism 
$$\xi_{p/q,p'/q}:=\chi_{p'/q}^{-1}\circ \chi_{p/q}: \Mpq \rightarrow \M_{p'/q}.$$ 
Then $\xi_{p/q,p'/q}$ restricts to a biholomorphic map between the main hyperbolic components 
$H_{p/q}$ and $H_{p'/q}$ of the two copies, in fact this holds for any pair of coresponding connected components of the interiors (see \cite{CG}, chapter 8.2). 
Define $\Mpqr:= \Mpq\setminus (H_{p/q} \cup e^{2\pi i p/q})$ and $\Mppqr:= \Mppq\setminus (H_{p'/q}\cup e^{2\pi i p'/q})$. In the first three sections we prove that $$\xi_{p/q,p'/q|}:=\chi_{p'/q}^{-1}\circ \chi_{p/q|}: \Mpqr \rightarrow \Mppqr$$ is quasiconformal by showing that near any point in $\Mpqr$ it is the restriction of a quasiconformal homeomorphism with uniformly bounded dilatation. 
Finally in Section \ref{laststep} we show that it extends to a quasiconformal homeomorphism between neighbourhoods of the two sets.
\end{remark}

\subsection{Preliminaries}\label{pre}
\subsubsection{Polynomial-like maps and copies of the Mandelbrot set in the Mandelbrot set}\label{DH}
Tthe filled Julia set $K_c$ for $Q_c(z)=z^2+c$ is the complement of the basin of attraction of infinity.
Thus, the dynamics of a polynomial $Q_c$ in a neighborhood of $K_c$ is expanding, and the preimage under $Q_c$ of a (suitable) topological disk $U_P$ containing $K_c$ is a topological disk $U'_P$ compactly contained in $U_P$, and $Q_c$ is a proper and holomorphic map from $U'_P$ to $U_P$.
A polynomial-like map is a proper holomorphic map which topologically behaves like a polynomial. More formally, a polynomial-like map of degree $d$ is a triple $(f, U', U)$ where $U'$ and $U$ are topological disks, $U'$ is compactly contained in $U$, and $f: U' \to U$ is a proper and holomorphic map of degree $d$. The filled Julia set of the polynomial-like map $(f, U', U)$ is the set of points that do not escape $U$. By Douady and Hubbard's Straightening Theorem (\cite{DH}), each degree $d$ polynomial-like map $(f, U', U)$ is conjugate, via a \textit{hybrid equivalence}, i.e.~a quasiconformal homeomorphism $\phi$ with $\overline{\partial} \phi=0$ on $K_f$, to a polynomial of the degree $d$, a unique such member if the filled Julia set $K_f$ is connected.

Now, consider a family $\{f_l\}_{l \in L}$, $L \approx \D$, of polynomial-like maps of degree $2$, with $f_l = Q^k_c$ and $k > 1$ fixed in this family.
 Following Lyubich, we call the filled Julia set $K_l$ of the map $f_l$ a \textit{little} Julia set, to distinguish it from the filled Julia set $K_c$ of the original quadratic polynomial $Q_c$ (which contains $k$ little Julia sets: $K_l, Q_c(K_l), ..., Q_cQ^{k-1}_c(K_l)$).
For each parameter $l$, $f_l$ is hybrid equivalent to a polynomial, and this equivalence is unique if $K_l$ is connected. Thus, defining $M_f$ as the set of parameters $l$ for which the filled Julia set $K_l$ of $f_l$ is connected, we can define a map $$\chi: M_f \to \M$$ which associates to each parameter $l \in L$ the unique value $c$ such that $Q_c$ is  hybrid equivalent to $f_l$. Douady and Hubbard extended this map to the entire disk $L$ and proved that, if $M_f \subset L$ is compact (called \textit{full} by Lyubich) and the winding number of the critical value around the critical point is $1$ for a positively oriented closed path in $L$ surrounding $\M_f$, 
the map is a homeomorphism between the connectedness locus $M_f$ and the Mandelbrot set $\M$.
As a consequence, we have that little copies of the Mandelbrot set within the Mandelbrot set which are homeomorphic to the Mandelbrot set itself.
These little copies have an associated period $k$. 
The root of the copy is the unique parameter for which 
the map has a $k$-cycle of multiplier $1$. 
The copy is \textit{primitive} if this parabolic $k$-cycle is a collision of two $k$-cycles. 
Primitive copies are visually identifiable by their root cusp such as the cusp of 
the cardioid boundary of the main hyperbolic component of $\M$. 
Moreover for primitive copies the little Julia sets are disjoint. 
The copy is called a \textit{satellite} copy if the root is a collision of a $k$ cycle and a cycle of lower period $Q$ dividing $k$, i.e.~$k=qQ$ 
(for the root the $k$-cycle appears to have period $Q$).
A satellite copy is visually identifiable as a copy where the main hyperbolic component has smooth boundary and is sprouting of a hyperbolic component of lower period $Q$. 
Moreover in the satellite case the little Julia sets come in $Q$ clusters of $q$ copies 
sharing a $Q$-cycle acting as the relative $\beta$-fixed point (a quadratic polynomial and hence a quadratic polynomial-like map has counting multiplicity $2$ fixed points, $\alpha$ and $\beta$).
\subsubsection{Quasiconformal maps and quasisymmetric maps}\label{qceqs}
A map
$\phi:U \subseteq \C \rightarrow V \subseteq \C$ is called $K$-\textit{quasiconformal} if:
\begin{itemize}
    \item $\phi$ is an orientation preserving homeomorphism,
    \item $\phi \in W_{loc}^{1,2}(U)$, meaning that it has first-order distributional derivatives which are in $L_{loc}^2(U)$,
    \item $|\overline{\partial}\phi(z)| \leq k |\partial\phi(z)|$ a.e.~with $k= \frac{K-1}{K+1}<1.$
\end{itemize}
That is, defining the Beltrami coefficient $\mu_{\phi}(z):=\frac{\overline{\partial}\phi(z)}{\partial\phi(z)}$, which exists a.e.~we have $|\mu_{\phi}(z)|\leq k$ a.e. 
Then the (real) dilatation at $z$ is $K_\phi(z) = \frac{1+|\mu(z)|}{1-|\mu(z)|}$. 
Geometrically this implies that the differential 
$D_z\phi = \partial\phi(z)dz + \overline{\partial}\phi(z)d\overline{z}$ exists and is injective a.e.~and it and its inverse maps circles to ellipses with distortion $K_\phi(z) \leq K$. 
The set of $K$-quaisconformal maps on a domain $U$ is sequentially pre-compact and any limit map is either $K$ quasiconformal or constant.

A family of maps $\phi\la$ is \textit{uniformly quasiconformal} if for every $\lambda$ we have that $\phi\la$ is $K\la$ quasiconformal, and $K\la \leq \hat K <\infty$.

An orientation preserving homeomorphism
$h: \S^1 \rightarrow \S^1$ is \textit{quasisymmetric} if there exists a $k_h\geq 1$ such that, for all $x \in\R$ and $t>0$, we have 
$$\frac{1}{k_h}\leq \frac{H(x+t)-H(x)}{H(x)-H(x-t)}\leq k_h\,,$$
where $H:\R\to\R$ is any lift of $h$ to $e^{2 \pi i t}$. 

If $h: \S^1 \rightarrow \C$ is quasisymmetric, then it has a $K$-quasiconformal extension $\phi: \D \rightarrow \D$ with $K_\phi$ depending only $k_h$ (see \cite{BA}).

A \textit{quasicircle}  is the image of a circle under a quasiconformal map, and a \textit{quasidisc} is the image of a disc under a quasiconformal map.

\section{Plowing in dynamical space}\label{dyn}
We call two degree $2$ polynomial-like maps \textit{corresponding} if they have connected filled Julia set and are hybrid equivalent (so that the map $\xi$ would send one to the other).
In this section we will construct a family of pre-exterior uniformly quasiconformal equivalences $\varphi\la$ between corresponding degree $2$ polynomial-like maps, these belonging to the two families $\{f\la\}_{\lambda \in \Mpq},\,\,f\la:=P_{\lambda |}^q$ and $\{g\nn\}_{\nu \in \Mppq},\,\,\,g\nn:=P_{\nu |}^q$. We will start by constructing a family of uniformly quasiconformal interpolations $\{\widetilde\Phi\la\}_{\lambda \in \Mpqr}$, where $\Mpqr:= \Mpq\setminus (H_{p/q} \cup e^{2\pi i p/q})$, between external rays separating the little filled Julia set containing the critical value (see Section \ref{between rays}). 
In Section \ref{cycle}, we construct  a family of uniformly quasiconformal interpolations about the $q-1$ points of the $q$ cycle different from the $\alpha$-fixed point of the quadratic-like restriction $f\la$. In Section \ref{Extension}, we use these previous interpolations to construct a family of pre-exterior uniformly quasiconformal equivalences $\varphi\la$.

\subsection{Interpolation between external rays}\label{between rays}
Let $\phi\la: \C \rightarrow \C$ and $\phi\nn: \C \rightarrow \C$ denote the normalized linearizers for
$P\la$ and $P\nn$ at the repelling fixed point 
$\alpha\la$ and $\alpha\nn$ respectively, so we have  $P\la( \phi\la(z)) = \phi\la(\lambda z)$, 
$ P\nn (\phi\nn(z)) = \phi\nn (\nu z)$ and $\phi\la'(0) = \phi\nn'(0) =1$.
Let $D\la$ and $D\nn$ be the image of the maximal disk of univalency 
for $\phi\la$ and $\phi\nn$ respectively.
This is,  $r\la$ is the maximal radius such that 
$$\phi\la: \D(0,r\la) \rightarrow D\la$$ is univalent, $\lambda \in \Mpq$, and similarly
$r\nn$ is the maximal radius such that 
$$\phi\nn: \D(0,r\nn) \rightarrow D\nn$$ is univalent, $\nu \in \Mppq$.
For $\iln= \{\lambda, \nu\}$, let $K_\iln$ be the little Julia set containing the critical value, and let $R_\iln'$
and $R_\iln$ be the external rays separating $K_\iln$ from its $q-1$ copies in the filled Julia set $K_{P_\iln}$, in anticlockwise order. 
 Let $R_1'$ and $R_1$ be the images in B\"ottcher coordinates of $R\la'$ and $R\la$ respectively, and let
$R'_2$ and $R_2$ be the images in B\"ottcher coordinates of $R\nn'$ and $R\nn$ respectively. Then there exists and angle $\widehat \theta$ such that the rotation $Rot_{\widehat\theta}$ of angle $\widehat \theta$ maps $R_1$ to $R_2$ and $R_1'$ to $R_2'$.
\begin{figure}[tbh]
    \centering
    \includegraphics[scale = 0.6]{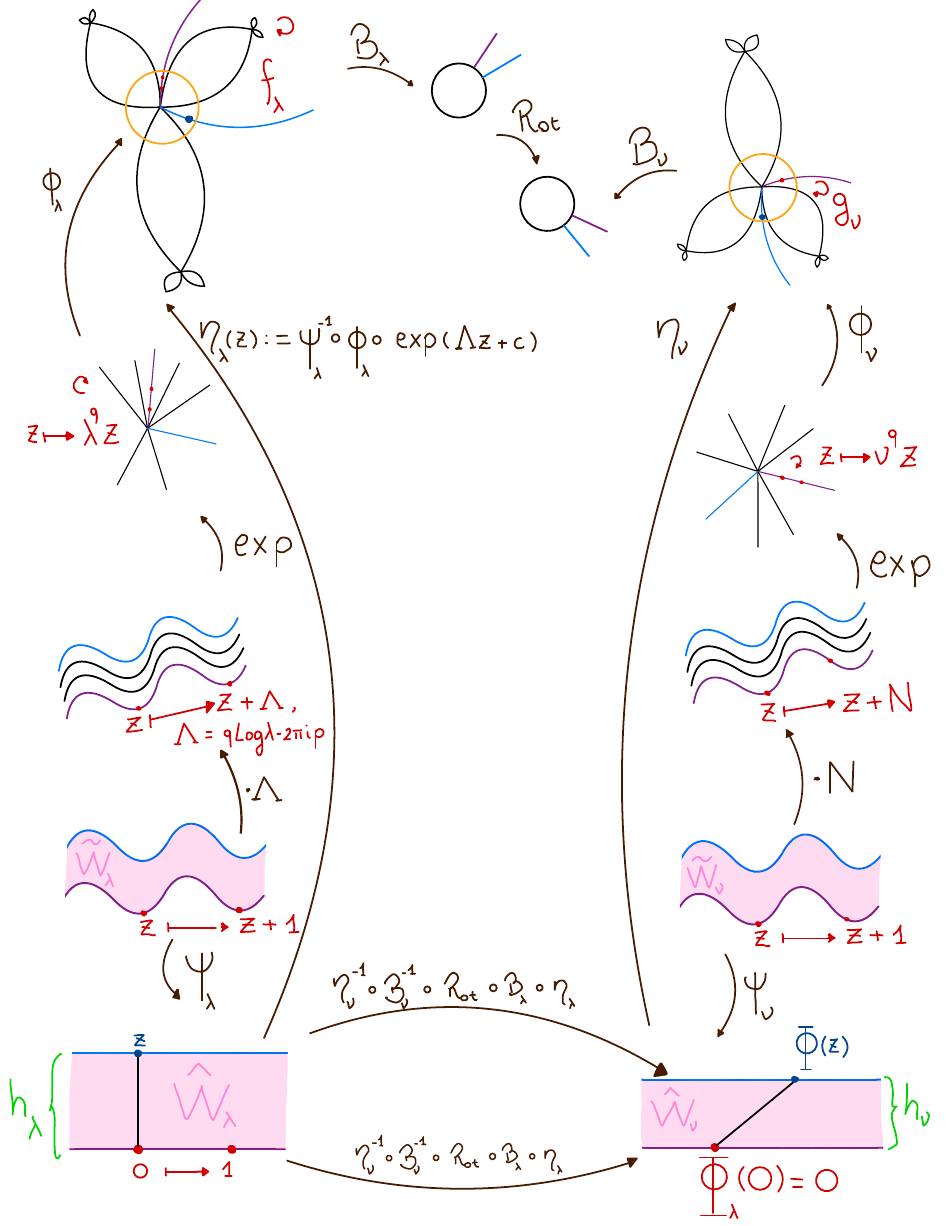}
    \caption{The change of coordinates}
    \label{fig:changeCoor}
\end{figure}
Recall that $\Mpqr:= \Mpq\setminus (H_{p/q} \cup e^{2\pi i p/q})$.
For every $\lambda \in \Mpqr$ choose an equipotential $E=E\la$ such that, defining $z_\iln:=R_\iln \cap E$, $x\la=f\la(z\la)\in R\la$, $x\nn=g\nn(z\nn)\in R\nn$, $z'_\iln:=R'_\iln \cap E$, $x'\la=f\la(z'\la)\in R'\la$, and $x'\nn=g\nn(z'\nn)\in R'\nn$,
the connected component of $R_\iln$ bounded by $\alpha_\iln$ and $x_\iln$, and the connected component of $R'_\iln$ bounded by $\alpha_\iln$ and $x'_\iln$, both belong to $D_\iln$ (see Figure \ref{fig:changeCoor}).
Let $\Lambda = \Lambda\la = \Lambda(\lambda) = \Log \lambda^q =q\Log\lambda-2\pi i p$. 
Then $\lambda\mapsto\Lambda$ is univalent on a neighbourhood of $\Mpq$. 
Let
$$
\phi\la\circ \exp(\Lambda z+c): \C \rightarrow \C, 
$$
be the conformal map conjugating $T_1(z)=z+1$ 
to the dynamics of $f\la$, and
normalized by sending $0$ to $z\la$ (see Figure \ref{fig:changeCoor}). Call $\tilde R\la$ the preimage of $R\la$ containing zero, $\tilde R'\la$ the preimage of $R'\la$ between 
$\tilde R\la$ and $\tilde R\la +2\pi i/\Lambda$, and $\tilde z\la'$ the preimage of $z\la'$ on $\tilde R'\la$.
Similarly, let $N= \Log \nu^q=q \Log \nu -2\pi i p'$ and
$$\phi\nn\circ \exp(N z+c'): \C\rightarrow \C$$
be the conformal conjugacy between $T_1$ and
$g\nn$, 
normalized by sending $0$ to $z\nn$, call $\tilde R\nn$ the preimage of $R\nn$ containing zero, $\tilde R'\nn$ the preimage of $R'\nn$ between 
$\tilde R\nn$ and $\tilde R\nn +2\pi i/N$, and $\tilde z\nn'$ the preimage of $z\nn'$ on $\tilde R'\nn$.

Call $\tilde W_\iln$, $i=\lambda, \nu$, the bi-infinite strip 
bounded by $\tilde R_\iln$ from below and $\tilde R'_\iln$ from above (see Figure \ref{fig:changeCoor}). 
Let $h_\iln$ be the modulus of the cylinder $\tilde W_\iln/\Z$, then there exists a unique conformal map $$\psi_\iln:\tilde W_\iln \rightarrow \hat W_\iln := \{x+iy\;| 0< y < h_\iln\}$$ fixing zero and commuting with $T_1(z)=z+1$. As the boundaries of $\tilde W_\iln$ and $\hat W_\iln$ are analytic, the map $\psi_\iln$ extends conformally to a neighbourhood of $\tilde W_\iln$, and we define  $$\hat z_\iln'=\psi_\iln(\tilde z_\iln'),\,\,\,\,\,\hat x_\iln'=\psi_\iln(\tilde x_\iln'),\,\,\,\,\,\hat R_\iln = \psi_\iln (\tilde R_\iln),\,\,\,\,\,\,\hat R'_\iln = \psi_\iln (\tilde R'_\iln).$$ Note that 
$\hat R_\iln\subset \R$, and $\hat R'_\iln\subset \R + i h_\iln$.
Define
$$\eta\la:= \phi\la\circ \exp(\Lambda z+c)\circ \psi\la^{-1}: \hat W\la \rightarrow \C,$$
and note that this map conjugates translation by one to the dynamics of $f\la$ (see Figure \ref{fig:changeCoor}). Similarly, the map
$$\eta\nn:= \phi\nn\circ \exp(N z+c')\circ \psi\nn^{-1}: \hat W\nn \rightarrow \C,$$
conjugates translation by one to the dynamics of $g\nn$.

Let $B\la$ and $B\nn$ be the B\"ottcher maps for $f\la$ and $g\nn$ respectively, and let 
$$\Phi\la: \partial \hat W\la \rightarrow \partial \hat W\nn $$
be a lift of $B\nn^{-1} \circ R_{\widehat\theta}\circ B\la \circ \eta\la$ to $\eta\nn$ fixing the real axis and zero, and mapping $z\la'$ to $z\nn'$, then $\Phi$
is a real analytic map between the boundaries of two straight infinite strips, with $\Phi(\hat z\la')=\hat z\nn'$ (see Figure \ref{fig:changeCoor}). 
We extend $\Phi$ to $\hat W\la$ by mapping the straight segment joining $x \in \hat R\la$ to $x+ih\la \in \hat R'\la$ to the straight segment joining $\Phi(x) \in \hat R\nn$ to $\Phi(x+ih\la) \in \hat R'\nn$, this is:  
$$\Phi\la: \overline{\hat W\la} \longrightarrow \overline{\hat W\nn}$$
$$x+iy \longrightarrow \Phi\la(x) + \frac{y}{h\la} (\Phi\la(x+ih\la) - \Phi\la(x)).$$
Recall that $\Mpqr:= \Mpq\setminus (H_{p/q} \cup e^{2\pi i p/q})$.
The aim of this Section is to show that the family  $\{\Phi\la\}_{\lambda \in \M^*_{p/q}}$ is uniformly quasiconformal:
\begin{teor}\label{intrays}
 Let $\mu\la= \frac{\bar \partial \Phi\la}{ \partial \Phi\la}$, 
 then there exists $k, 0\leq k < 1$ such that, for all $\lambda \in \M^*_{p/q}$,
 $$\sup\la ||\mu\la||_\infty < k.$$
\end{teor}  

\begin{cor}\label{gen}
 There exists a family $\widetilde \Phi^{i,j}\la$ of uniformly quasiconformal interpolations between any consecutive rays $R^i\la$ and $R^j\la$, in a domain of univalency of the linearizers maps. 
\end{cor}
\begin{proof}
The proof of Theorem \ref{intrays}, contained in the next section, applies verbatim to this Corollary.
\end{proof}

\subsubsection{Proof of Theorem \ref{intrays}}
We need to prove that, for all $\lambda \in \M^*_{p/q}$,
$$|\mu\la|^2 = 
\frac{\left(\frac{\partial u\la}{\partial x}-\frac{\partial v\la}{\partial y}\right)^2+
\left(\frac{\partial u\la}{\partial y}+\frac{\partial v\la}{\partial x}\right)^2}
{\left(\frac{\partial u\la}{\partial x}+\frac{\partial v\la}{\partial y}\right)^2
+\left(\frac{\partial v\la}{\partial x}-\frac{\partial u\la}{\partial y}\right)^2}  \leq k <1.
$$
As $\Phi\la|_{\hat R_1}: \hat R_1\subset \R \rightarrow \R$, and $\Phi\la|_{\hat R_1'}: \hat R_1'\subset \R+ih\la \rightarrow \R+ih\nn$,
we have $$\Phi\la(x+iy)=\Re(\Phi\la(x))+\frac{y}{h\la}(\Re(\Phi\la(x+ih\la))+ih\nn-\Re(\Phi\la(x))),$$
which gives:
$$u\la= \Re(\Phi\la(x))+\frac{y}{h\la}(\Re(\Phi\la(x+ih\la))-\Re(\Phi\la(x)))\,\,\,\,\,\,\,\,\mbox{ and }\,\,\,\,\,\,\,\,\,\,\,\,\,v\la=y\frac{h\nn}{h\la}$$
hence $\frac{\partial v\la}{\partial x} \equiv 0.$
Therefore, in order to prove the Theorem \ref{intrays} it is enough to show that there exist positive constants $C_1,C_2,$ $K_1,K_2$, and $M$ such that, for all $\lambda \in \Mpqr$,
$$C_1<\frac{\partial u\la}{\partial x}<C_2,\,\,\,\,\,\,\,\,\, K_1<\frac{\partial v\la}{\partial y}<K_2,\,\,\,\,\,\,\,\,\frac{\partial u\la}{\partial y}<M.$$
\subsubsection{Bounding $\frac{\partial v\la}{\partial y}$}

\begin{lemma}\label{1sulambda}~\\
$$
\forall\; \lambda \in \Mpqr :\Re(1/\Lambda)\leq \frac{1}{2 r}, 
\qquad
\forall\; \nu \in \Mppqr :\Re(1/N)\leq \frac{1}{2 \hat r}, 
$$
\end{lemma}
\begin{proof}
    Recall that $\Lambda = \Lambda(\lambda)= \Log \lambda^q =q\Log\lambda-2\pi i p$. 
Since $H_{p/q}$ has smooth boundary and is tangent to $\D$ at $\lambda_0$ 
the image $\widehat H_{p/q} = \Lambda(H_{p/q})$ contains a disk $\D(r,r)$ 
tangent to $i\R$ at $0$. 
Moreover the M{\"o}bius transformation $\Lambda\mapsto 1/\Lambda$ maps 
$\D(r,r)$ biholomorphically onto the right half plane $\{x+iy : x > 2r\}$. 
Let $\hat r>0$ be the similar radius for $H_{p'/q}$.
This proves the Lemma.
\end{proof}

\begin{prop}\label{2}
 There exist positive constants $K_1$ and $K_2$ such that, for all $\lambda \in \Mpqr$,
$$K_1<\frac{\partial v\la}{\partial y}<K_2.$$
\end{prop}
\begin{proof}
We have to prove the existence of positive constants $K_1$ and $K_2$ such that, for all $\lambda \in \Mpqr$, $K_1<\frac{h\la}{h\nn}<K_2.$
By construction, $h\la$ is the modulus of $\tilde W\la/\Z$, hence the modulus 
 of a cylinder embedded in the complex torus 
$$D\la \setminus \{\alpha\la\}/\lambda^q \approx \C/ (\Lambda \Z + 2\pi i \Z),$$ 
homotopically to $\Lambda$, so by Bers inequality we have:
$$h\la \leq \frac{2\pi \Re(\Lambda)}{|\Lambda|^2}= 2\pi \Re(1/\Lambda). $$
Also by construction, this cylinder contains $q-1$ of the $q$ annuli corresponding to connected components of the basin of attraction of $\infty$ in the complex torus $D\la \setminus \{\alpha\la\}/\lambda^q$ (to be precise, the $q-2$ plus $2$ $q/2$ connected components). Hence:
$$
\frac{(q-1)\pi}{q^2\Log 2}\leq h\la \leq \frac{2\pi \Re(\Lambda)}{|\Lambda|^2} 
= 2\pi \Re(1/\Lambda). 
$$

Combining with Lemma~\ref{1sulambda} we obtain 
$$
\frac{(q-1)\pi}{q^2\Log 2}\leq h\la \leq \pi/r.
$$ 
And similarly, 
$$
\frac{(q-1)\pi}{q^2\Log 2}\leq h\nn\leq \pi/\hat r.
$$ 
\end{proof}
\subsubsection{Bounding $\frac{\partial u\la}{\partial x}$}
 Let $U$ be the horizontal infinite strip of hight $\frac{\pi}{q\Log 2}$ and symmetric with respect to the real axis. There is a conformal map: $$p\la^u: U \rightarrow \C\setminus \overline\D$$ mapping the real line to the ray $R_1'$ and conjugating the map $T_1(z)=z+1$ acting on $U$ to the map $z \rightarrow z^{2^q}$ acting on $\C\setminus \overline\D$. Let $i= \lambda, \nu$, let $B_\iln$ be the B\"ottcher map, and 
$$\tilde\zeta\la^u:U \rightarrow \C,$$
be a lift of $B\la^{-1} \circ p\la^u:U \rightarrow D\la$ to  $\phi\la\circ \exp(\Lambda z+c)$, mapping zero to $\tilde z\la'$ (see Figure \ref{fig:decomposition}). Set $$\tilde U\la^u=\tilde \zeta\la^u(U).$$
\begin{figure}[tbh]
    \centering
    \includegraphics[scale = 0.6]{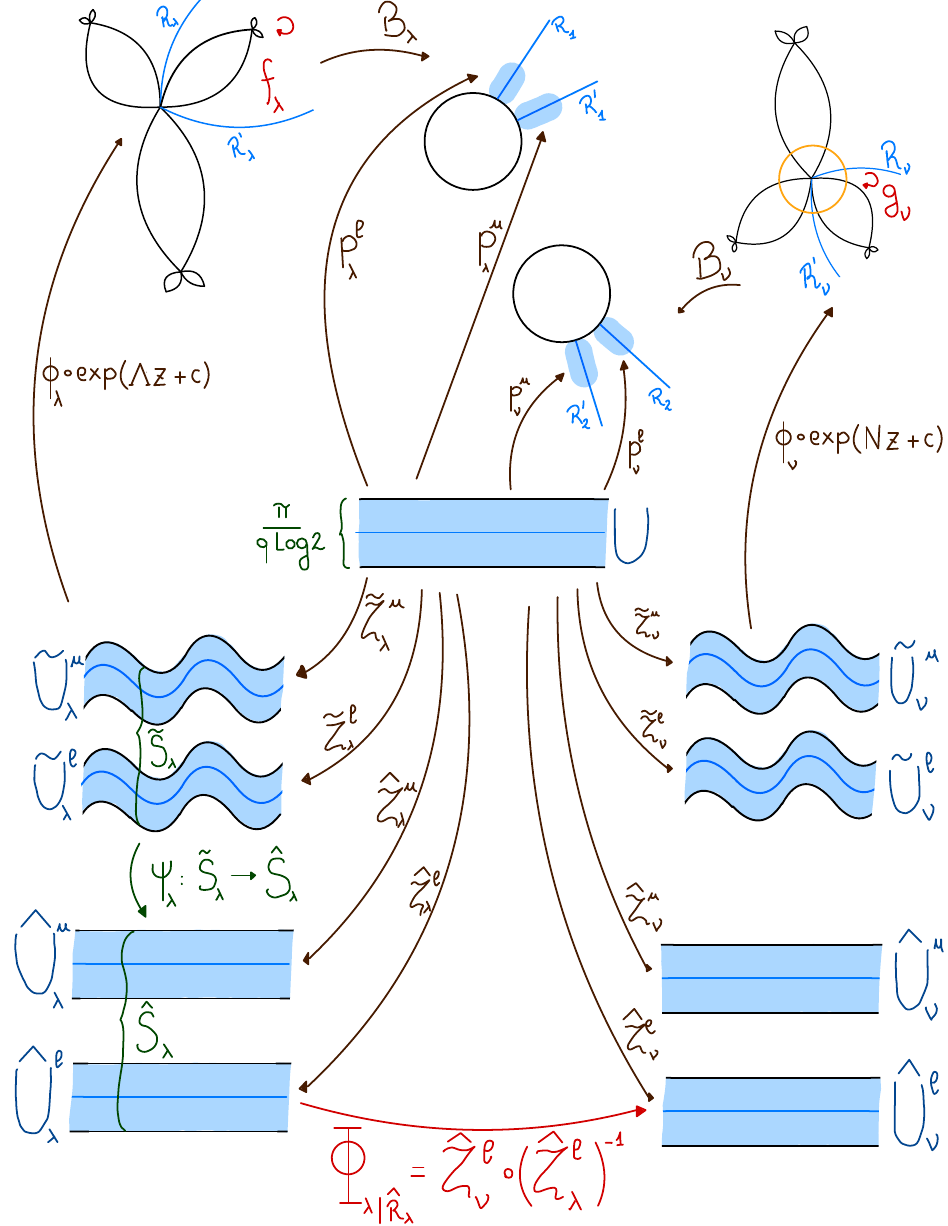} 
    \caption{The maps $\tilde\zeta\la^u$, $\tilde\zeta\la^l$,$\hat\zeta\la^u$, and $\hat\zeta\la^l$.}
    \label{fig:decomposition}
\end{figure}
Then the map $\tilde \zeta\la^u$ conjugates conformally the map $T_1$ acting on $U$ to the map $T_1$ acting on $\tilde U\la^u$. 
Define similarly the map $p\la^l: U \rightarrow \C\setminus \overline \D$, 
let $$\tilde\zeta\la^l:U \rightarrow \C,$$
be a lift of $B\la^{-1} \circ p\la^l:U \rightarrow D\la$ to  $\phi\la\circ \exp(\Lambda z+c)$ fixing zero, and set 
$\tilde U\la^l=\tilde \zeta\la^l(U)$ (see Figure \ref{fig:decomposition}). 
Set $$U_-=\{z= x+iy \in U \,\,|\,\, y \leq 0\},$$ then  $\tilde \zeta\la^u(\mathring U_-)=\tilde U\la \cap \tilde W$, and the map
$$\hat\zeta\la^{u,-}:= \psi\la \circ \tilde \zeta^u_{\lambda|U_-}:U_- \rightarrow \C$$
is a lift of $B\la^{-1} \circ p\la^u:U_- \rightarrow D\la$ to  $\eta\la=\phi\la\circ \exp(\Lambda z+c)\circ \psi\la^{-1}$ mapping $0$ to $\hat z\la'$ (see Figure \ref{fig:decomposition}). Note that $$\hat U\la^{u,-}:=\zeta\la^u(U_-)=\psi\la(\tilde U\la \cap \tilde W\la) \subset \hat W\la.$$  
Define $\hat U^{u,+}$ as the reflection of $\hat U\la^{u,-}$ with respect to the (straight) ray $\hat R\la'$, and set $\hat U^u:= \hat U^{u,+} \cup \hat R\la' \cup \hat U^{u,-}$. Then we can extend the map $\hat\zeta\la^{u,-}:U_- \rightarrow \C$ to a univalent map: $$\hat\zeta\la^{u}:U \rightarrow \hat U\la^u$$ by using reflection with respect to $\R$ and to $\hat R\la'$. 
Similarly, let $U_+=\{z= x+iy \in U \,\,|\,\, y \geq 0\}$, then  
$$\hat\zeta\la^{l,+}:= \psi\la \circ \tilde \zeta^l_{\lambda|U_+}:U_+ \rightarrow \C$$
is a lift of $B\la^{-1} \circ p\la^l:U_+ \rightarrow D\la$ to  $\eta\la=\phi\la\circ \exp(\Lambda z+c)\circ \psi\la^{-1}$ fixing zero, which we can extend to  a univalent map: $$\hat\zeta\la^{l}:U \rightarrow \hat U\la^l,$$
where $\hat U\la^l$ is the union of $\hat\zeta\la^{l,+}(U_+)$, its reflection with respect to the straight ray $\hat R\la$, and the ray itself, as above.
Set
$$\tilde S\la:= \tilde U\la^u \cup \tilde W\la \cup \tilde U\la^l,$$
then the map $\psi\la: \tilde W\la \rightarrow \hat W\la$ extends to $$\psi\la:\tilde S\la \rightarrow \hat S\la:= \hat U\la^u \cup \hat W\la \cup \hat U\la^l,$$ defining $\psi\la(\tilde U\la):= \hat \zeta\la \circ \tilde \zeta\la^{-1}(\tilde U\la)$ (see Figure \ref{fig:decomposition}).
As $\psi\la$ is holomorphic and $\tilde S\la$ is simply connected, this extension agrees with the extension about $\tilde R\la$ and $\tilde R'\la$ previously defined.
Note that, for every $\lambda\in \Mpqr$, the infinite  horizontal strip $\hat S\la$ has $mod(\hat S\la/\Z) \geq h\la+\frac{\pi}{q\Log 2}$.\\

Repeat for $\nu$ (see Figure \ref{fig:decomposition}). More precisely, let $\tilde\zeta\nn^u:U \rightarrow \C,$
be a lift of $B\nn^{-1} \circ p\nn^u:U \rightarrow D\nn$ to  $\phi\nn\circ \exp(N z+c)$ mapping $0$ to $\tilde z\nn'$, and set 
$\tilde U\nn^u=\tilde \zeta\nn^u(U)$. Similarly, let $\tilde\zeta\nn^l:U \rightarrow \C,$
be a lift of $B\nn^{-1} \circ p\nn^l:U \rightarrow D\nn$ to  $\phi\nn\circ \exp(N z+c)$ fixing zero, and set 
$\tilde U\nn^l=\tilde \zeta\nn^l(U)$.
Define $\tilde S\nn:= \tilde U\nn^u \cup \tilde W\nn \cup \tilde U\nn^l$.
Let $\hat\zeta\nn^{u,-}:= \psi\nn \circ \tilde \zeta^u_{\nu|U_-}:U_- \rightarrow \C$ be a lift of $B\nn^{-1} \circ p\nn^u:U \rightarrow D\nn$ to  $\eta\nn$ mapping zero to $\hat z\nn'$, and
construct $\hat\zeta\nn^u:U \rightarrow \C$  by reflecting the map $\hat\zeta\nn^{u,-}$. 
Then $\hat U\nn^u:=\tilde \zeta\nn^u(U)=\psi\nn (\tilde U\nn^u)$. Construct similarly $\hat\zeta\nn^l:U \rightarrow \C$, so $\hat U\nn^l:=\tilde \zeta\nn^l(U)=\psi\nn (\tilde U\nn^l)$.
Set $\hat S\nn:= \hat U\nn^u \cup \hat W\nn \cup \hat U\nn^l.$
Then the map $\psi\nn: \tilde W\nn \rightarrow \hat W\nn$ extends to $\psi\nn:\tilde S\nn \rightarrow \hat S\nn$ and, 
 for every $\nu \in \Mppqr$, the  infinite 
 horizontal strip  $\hat S\nn$ has $mod(\hat S\nn/\Z)=h\nn+\frac{\pi}{q\Log 2}$.

\begin{lemma}
Set $\varphi\la:=\psi\la^{-1}:\hat S\la\rightarrow \tilde S\la$. For for $\lambda \in \Mpqr$, the family of maps $\varphi\la$, is compact: for every sequence $\lambda_n$ there is a subsequence $\lambda_{n_k}$ such that $\varphi_{\lambda_{n_k}} \rightarrow \hat \varphi$, and $ \hat \varphi: \hat S \rightarrow \tilde S$ is a univalent map.
Similarly, setting $\varphi\nn:=\psi\nn^{-1}:\hat S\nn\rightarrow \tilde S\nn$, the family of maps $\varphi\nn$ is a compact family.
\end{lemma}

\begin{proof}
 The family $\varphi\la:\hat S\la\rightarrow \tilde S\la$ for $\lambda \in \Mpqr$ is a family of univalent maps defined on and 
 infinite straight horizontal strip  $\hat S\la$  with uniformly bounded hight, fixing the origin and commuting with translation by one. Hence the family $\varphi\la$ fixes infinite points: every integer number. So, it is a compact family. 
The same argument shows that the family of maps $\varphi\nn$ is a compact family.
\end{proof}
\begin{lemma}\label{b}
  There exists $K_q>0$ such that, for all $w \in \widehat W_\iln$:
 $$|\varphi_\iln(w)-w|<K_q$$
\end{lemma}
\begin{proof}
 Since the map $\varphi_\iln$ commutes with translation by one, and the families $\varphi_{\lambda}$ and $\varphi_{\nu}$ are compact by the previous Lemma, the result follows.
\end{proof}

\begin{lemma}\label{cpthatzeta}
There exists a constant $K=K(\frac{\pi}{q\Log 2})>0$ such that, for all  $\lambda \in \Mpqr$, and $\nu=\xi (\lambda)$, 
$$ \frac{1}{K}< (\hat \zeta\la^l)' < K,\,\, \frac{1}{K}< (\hat \zeta\la^u)' < K,\,\, \frac{1}{K}< (\hat \zeta\nn^l)' < K, \mbox{ and }  \frac{1}{K}< (\hat \zeta\nn^u)' < K.$$
\end{lemma}
\begin{proof}
For $h>0$, let $U_h =\{x+iy \,\,| \,\, |y|<h/2\}$, and consider the family of maps
$\Sigma(h) = \{\varphi : U_h \rightarrow \C \, | \,\,\,\, \varphi  \mbox{ is univalent, and } \varphi \circ T_1 = T_1 \circ \varphi \}$.
Then, by K\"obe Theorem, there exists $K=K(h)$ such that, for all $\varphi \in \Sigma$ and for all $z,\,w \in \R$,
$$ \frac{1}{K}<    \Big| \frac{\varphi'(z)}{\varphi'(w)}\Big| < K.$$
As $\varphi$ commutes with translation, we have
$$1= \int_{x-1/2}^{x+1/2} \varphi'(t) dt,$$
and as 
$$ \frac{\varphi'(x)}{K}<    | \varphi'(t)| < K \varphi'(x),$$
we obtain 
$$1\leq  \int_{x-1/2}^{x+1/2} \varphi'(x) K dt = K \varphi'(x),$$
so $\varphi'(x)\geq 1/K$. Similarly,
$$ \frac{\varphi'(x)}{K} =   \int_{x-1/2}^{x+1/2} \frac{\varphi'(x)}{K} dt \leq  \int_{x-1/2}^{x+1/2} \varphi'(t) dt =1 .$$

As the maps $\hat \zeta$ belong $\Sigma(\frac{\pi}{q\Log 2})$, because they are univalent, commuting with translation by one, and are defined on the bi-infinite strip of height $\frac{\pi}{q\Log 2}$ symmetric with respect to the real axis, the result follows.
\end{proof}
\begin{prop}\label{P1}
 There exist positive constants $C_1,C_2$ such that, for all $\lambda \in \Mpqr$,
$$C_1<\frac{\partial u\la}{\partial x}<C_2.$$
\end{prop}
\begin{proof}
Note that (see Figure \ref{fig:decomposition}) $$\Phi_{\lambda| \hat R\la}= \hat \zeta\la^u\circ (\hat \zeta\nn^u)^{-1}, \mbox{ and } \Phi_{\lambda|\hat R\la'}= \hat \zeta\la^l\circ (\hat \zeta\nn^l)^{-1}.$$
Hence the result follows from the previous Lemma \ref{cpthatzeta}.
\end{proof}
\subsubsection{Bounding $\frac{\partial u\la}{\partial y}$}
Recall that, at the beginning of this Section, for every $\lambda \in \Mpqr$ we chose an equipotential $E=E\la$ such that, defining $z_\iln:=R_\iln \cap E$, for $\iln ={\lambda,\nu}$, $x\la=f\la(z\la)\in R\la$, $x\nn=g\nn(z\nn)\in R\nn$, $z'_\iln:=R'_\iln \cap E$, $x'\la=f\la(z'\la)\in R'\la$, and $x'\nn=g\nn(z'\nn)\in R'\nn$,
the connected component of $R_\iln$ bounded by $\alpha_\iln$ and $x_\iln$, and the connected component of $R'_\iln$ bounded by $\alpha_\iln$ and $x'_\iln$, both belong to the maximal domain $D_\iln$ for the linearizer $\phi_\iln$ for $\alpha_\iln$.

Let $j \in (1, q-1)$ be given by $P^{j}\la(R\la)=R\la'$, and define $y\la= P\la^{j}(z\la) \in R\la'$. Similarly, let  $k \in (1,q-1)$ be given by $P\nn^{k}(R\nn)=R\nn'$, and define $y\nn= P\nn^{k}(z\nn) \in R\nn'$. Hence,
by construction we have (see Figure \ref{fig:Yp}):
$$z\la'< y\la < x\la' \mbox{ on } R\la',\,\ \mbox{ and }z\nn'< y\nn < x\nn' \mbox{ on } R\nn'.$$
\begin{figure}[tbh]
    \centering
    \includegraphics[scale = 0.5]{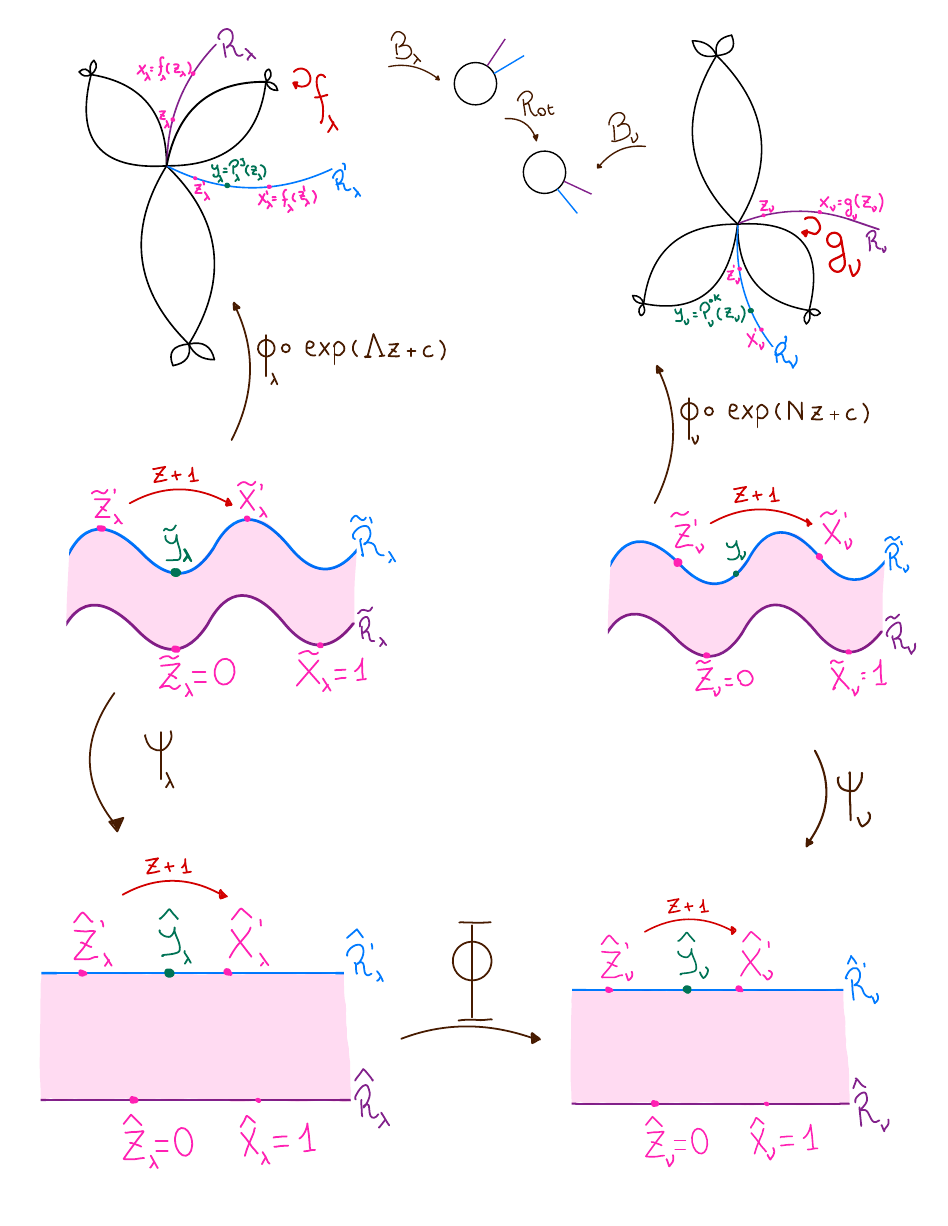}
    \caption{The points $y\la$ and $y\nn$. (All the points $z'_i,\,x'_i,\,z_i,\,x_i, i={\lambda,\nu},$ should be much closed to the $\alpha$-fixed point, to be contained in a domain oof univalency for the linearizer. We draw them so far from where they should be in order to make them visible).}
    \label{fig:Yp}
\end{figure}
Recall that the map $\phi\la\circ \exp(\Lambda z+c)$ maps $0 \in \tilde R\la$ to $z\la \in R\la$, and $\tilde z\la' \in \tilde R\la'$ to $z\la' \in R\la'$.
In this coordinates, the preimage $\tilde y\la$ of $y\la$ on $\tilde R'\la$ is 
$$\tilde y\la=\frac{j}{q}+\frac{q-1}{q \Lambda} 2\pi i,$$ as $R'\la$ and $R\la$ are consecutive rays in anticlockwise order.
Similarly, the preimage $\tilde y\nn$ of $y\nn$ under  $\phi\nn\circ \exp(N z+c')$, with $N= \Log \nu^q=q \Log \nu -2\pi i p'$, is
$$\tilde y\nn=\frac{k}{q}+\frac{q-1}{q N} 2\pi i.$$
We define $\hat y_\iln=\psi_\iln(\tilde y_\iln)$. 
\begin{prop}\label{3}
 There exists a constant $M>0$ such that, for all $\lambda \in \Mpqr$,
$$\Big|\frac{\partial u\la}{\partial y}\Big|<M.$$
\end{prop}
\begin{proof}
A computation shows that $$|\frac{\partial u\la}{\partial y}|= |\frac{\Re(\Phi\la(x+ih\la))-\Re(\Phi\la(x))}{h\la}|,$$ and as the denominator is uniformly bounded from below by Proposition \ref{2},
it is enough to show that there exists a $\hat M>0$ such that, for all $\lambda \in \Mpqr$ and for all $x \in \hat R$, we have 
$$|\Re(\Phi\la(x+ih\la))-\Re(\Phi\la(x))|<  \hat M.$$
As $\Re(x+ih\la)=\Re(x)$, we have $$|\Re(\Phi\la(x+ih\la))-\Re(\Phi\la(x))|\leq |\Re(\Phi\la(x+ih\la))- \Re(x+ih\la)|+| \Re(\Phi\la(x))-\Re(x)|,$$ and since $\Phi\la(0)=0$ and $\Phi\la\circ T_1=T_1\circ \Phi\la$, we obtain that  $$| \Re(\Phi\la(x))-\Re(x)| \leq 1,$$ so it is enough to bound $$|\Re(\Phi\la(x+ih\la))- \Re(x+ih\la)|.$$
As $\Phi\la\circ T_1=T_1\circ \Phi\la$, it is enough to bound the previous expression for a chosen $w\la=x+ih\la \in R\la'$. We choose to bound $|\Re(\Phi\la(\hat z\la'))- \Re(\hat z\la')|=|\Re(\hat z\nn')- \Re(\hat z\la')|$. 

As both 
$|\hat z\la' -\hat y\la|<1$ and $|z\nn' -\hat y\nn|<1$, this reduces to bound $|\hat y\la -\hat y\nn|$, and as $|\psi\la(x)-x|$ is uniformly bounded by Lemma \ref{b}, it is enough to bound 
$$|\tilde y\la -\tilde y\nn|= \Big|\frac{j}{q}+ \frac{2\pi i (q-1)}{q\Lambda}-\frac{k}{q}- \frac{2\pi i (q-1)}{qN}\Big|.$$
This follows from the fact that both
 $j/q$ and $k/q$ are less that $1$, and that, by Theorem 1.4 in \cite{K1} (see also Theorem 5.7 in \cite{K}) and the Proposition 5 in \cite{LP}, there exists a  $D>0$ such that, for all $\lambda \in \Mpqr$ and $\nu=\xi(\lambda)$,
 $$\Big|\frac{1}{\Lambda}-\frac{1}{N}\Big|<D.$$
\end{proof}

\subsection{Interpolation between neighbourhoods of the points of the $q$-cycle}\label{cycle}
Recall that $p/q$ and $p'\!/q$ are rational numbers with the same denominator $q\geq 3$ 
and that $\Mpq$ and $\Mppq$ are the satellite copies of the Mandelbrot set 
with root points $\lambda_0= e^{2\pi i p/q}$ and $\nu_0=e^{2\pi i p'\!/q}$ respectively. 
When $\lambda \in \Mpq$ and $\nu \in \Mppq$, the $q$-th iterates $f\la$ of $P\la(z)=\lambda z + z^2$ and 
$g\nn$ of $P\nn(z)=\nu z + z^ 2$ have quadratic like restrictions with $\beta$-fixed point 
the $\alpha$-fixed point $0$ of the respective quadratic polynomials.
Denote by $B\l0^j$ the components of the immediate parabolic basin 
for the parabolic fixed point $\alpha\l0 = 0$ and $q^j = q\la^j,\,j\in [0,q-1]$ 
the points of the $q$-cycle bifurcating from $\alpha\l0$ 
at the root of the satellite copy numbered in the anti-clockwise order, with $q^0$ the $\alpha$-fixed point of the quadratic-like restriction $f\la$ of $P\la$ and 
$q^{j+p} = P\la(q^j)$.
Recall also that $\rho\la:=(P\la^q)'(q^0)$ denotes the multiplier of the $q$-cycle. 
Define $B\n0$, $q^j\nn$ and $\rho\nn$ similarly.

Enumerate the rays $R_\iln^j$, where $\iln={\lambda,\nu}$ in the counter clockwise order with 
$R_\iln^0 = R_\iln'$ and thus $R_\iln = R_\iln^1$. 
Recall also that  $\Lambda = Log \lambda^q = q log \lambda -2\pi i p$ is the principal logarithm of $f'\la(\alpha\la)=\lambda^q$, and similarly $N$ is the principal logarithm of $g\nn'(\alpha\nn) = \nu^q$. 
Note that, if $\Im(\Lambda), \Im(N) <0$, the dynamics on the rays spiral out clockwise, 
hence the ray  $R_\iln^1$ is outer most relative to $R_\iln^0$ in the double spiral of the rays and 
so locally lies between $q^j_\iln$ and $R_\iln^0$ for $0<j<q$.
On the other hand, if $\Im(\Lambda), \Im(N)>0$, the ray $R_\iln^0$ is outer most in the double spiral of the rays and 
so locally lies between $q_\iln^j$ and $R_\iln^1$. 
To fix notation, we assume from now on that $\Im(\Lambda), \Im(N)<0$, the other case being similar. Note that by construction $|\lambda|, |\rho\la|, |\nu|, | \rho_\nu|\geq 1$ for $\lambda\in\Mpqr$ 
and $\nu\in\Mppqr$. Hence the logarithms $\Lambda, N, \Log(\rho\la), \Log(\rho\nn)$ 
all have non negative real parts on $\Mpqr$ and $\Mppqr$ respectively .

For $\whR>0$ let 
$$
\Delta(\whR) := \{\lambda \in  \D(\lambda_0,\whR) : |\Arg(-i\Lambda)| < \pi/6\}, 
\quad \Mpqr(\whR) := \Mpqr\cap\Delta(\whR)
$$
and 
$$
\Xi(\whR) := \{\nu \in \D(\nu_0,\widehat R): |\Arg(-iN)| < \pi/6\}, 
\quad
\Mppqr(\whR) := \Mppqr\cap\Xi(\whR).
$$

Write $f\l0(z) = z + a_1z^{q+1}+\OO(z^{q+2})$ and $g\n0(z) = z + b_1z^{q+1}+\OO(z^{q+2})$.  
In order to alleviate the following discussion we shall conjugate all maps $f\la$, $\lambda\in\Mpq$ 
by a $q$-th root $a'$ of $a_1$, and all maps $g\nn$, $\nu\in\Mppq$ by a $q$-th root $b'$ of $b_1$, 
and obtain new normalized families $\nf\la(z) = a' f\la(z/a')$ and $\ng\nn(z) = b' g\nn(z/b')$ for which 
$\nf\l0(z) = z + z^{q+1}+\OO(z^{q+2})$ and $\ng\n0(z) = z + z^{q+1}+\OO(z^{q+2})$ 
and the rays $a'R\l0^1$ and $b'R\n0^1$ converge to $0$ along the positive real axis. 
Possibly reducing $\widehat R>0$, there exists $r_0>0$ such that, 
for all $\lambda \in \D(\lambda_0,\widehat R)$ and $\nu \in \D(\nu_0,\widehat R)$, 
the disc $\D(0,r_0)$ contains $\alpha\la=0$, 
the $\nf\la$-$q$-cycle $q^j\la$ and no other fixed point of $\nf\la$ and furthermore 
$\Re(\nf'\la)>0$ on $\overline{\D}(0,r_0)$;  
and similarly the disc $\D(0,r_0)$ contains $\alpha\nn=0$, 
the $\ng\nn$-$q$-cycle $q^j\nn$ and no other fixed point of $\ng\nn$ 
and furthermor $\Re(\ng'\nn)>0$ on $\overline{\D}(0,r_0)$. 

\begin{remark}
To avoid heavy notation, we will set $f\la=\nf\la$,\,$g\nn=\ng\nn$ and so consider the families $f\la$ and $g\nn$ normalized as above, this is with $f\l0(z) = z + z^{q+1}+\OO(z^{q+2})$ and $g\n0(z) = z + z^{q+1}+\OO(z^{q+2})$. 
We shall not invent new names for entities such as rays, $q$-orbits and basins 
for the rescaled maps and shall omit the factors $a'$ and $b'$. \\
We will  furthermore use $w,\,x,\,y,\,z$ to denote points in a neighbourhood of the $q$ cycle, 
$E$ for denoting equipotentials, and $r$ and $R$ to denote radii, unrelated to points, 
equipotentials and radii having same notation as in the previous section. 
\end{remark}

In the following we shall make a parallel discussion for 
the families $f\la$ and $g\nn$. When only $f\la$ and quantities and maps related to $f\la$ 
are mentioned there are parallel quantities and maps related to $g\nn$. 

We shall make use of the Buff translation structure. To this end define a vector field $v\la(z) d/dz$ (the Buff-field) and the associated $1$-form $\omega\la(z) dz$ on $\overline\D(0,r_0)$, where
$$
v\la(z) = \frac{(f\la(z)-z)Log f'\la(z)}{f'\la(z)-1}, \qquad
\omega\la = \frac{1}{v\la}
$$
and define a holomorphic function $H$ on $\D(\lambda_0, \widehat R)$ by 
$$H(\lambda):=\frac{1}{2\pi i}\oint_{\S_{r_0}}{\omega\la(\zeta) d\zeta}
$$
Then we have (see \cite{LP}, Section 5.3)
$$
H(\lambda)= 
\left\{ 
  \begin{array}{ c l }
    \frac{1}{\Lambda\la}+ \frac{q}{Log(\rho\la)},
 & \mbox{ if  }\,\,\,\, \lambda \neq \lambda_0 \\
    Resit(f_{\lambda_0},0)  & \mbox{ if }  \,\,\,\, \lambda=\lambda_0.
  \end{array}
\right.
$$
In particular
\begin{equation}\label{H_assymptotics}
\frac{1}{\Lambda\la}+ \frac{q}{Log(\rho\la)} \underset{\lambda\to\lambda_0}\longrightarrow 
Resit(f_{\lambda_0},0)
\end{equation}
and we may also use the multiplier $\rho$ of the $q$-cycle as parameter provided 
$\whR>0$ is small enough, i.e.~we can invert and find $\lambda$ as a function of $\rho$.

Define similarly $v\nn$, $w\nn$ and $H(\nu)$. 
We obtain the following Lemma 
as an immediate consequence of equation~(\ref{H_assymptotics}).
\begin{lemma}\label{difmult}
There exist a constant $C_1>0$ such that, for every
 $\lambda \in \D^*(\lambda_0, \widehat R)$,  
 and $\nu \in \D^*(\nu_0, \widehat R)$, 
 we have 
 $$
 \Big| \frac{2\pi i }{q\Lambda} +\frac{2\pi i }{Log \rho\la}\Big|<C_1,\qquad\Big|\frac{2\pi i }{qN} + \frac{2\pi i }{Log \rho\nn}\Big|<C_1.
 $$
Moreover for $\lambda\in\Mpqr(\whR)$ and $\nu\in\Mppqr(\whR)$ 
$$
0\leq \Im\left(\frac{2\pi i }{q\Lambda}\right), \Im\left(\frac{2\pi i }{Log \rho\la}\right)< C_1,\qquad 
0\leq \Im\left(\frac{2\pi i }{qN}\right), \Im\left(\frac{2\pi i }{Log \rho\nn}\right)< C_1. 
 $$
\end{lemma}
\begin{proof}
On $\D(\lambda_0,\widehat R)$ the function $H(\lambda)$ is bounded and equals 
$\frac{2\pi i }{q\Lambda} +\frac{2\pi i }{Log \rho\la}$ on $\D^*(\lambda_0, \widehat R)$. 
And similarly for $\nu, N$.
By taking the larger of the bounds we can assume the constants are the same. 

Moreover when also $\lambda\in\Mpqr$ and $\nu\in\Mppqr$ 
then in particular $\lambda\notin\D\cup H_{p/q}$
and $\nu\notin\D\cup H_{p'/q}$ so that 
$$
0\leq\Re(\Lambda), \Re(N), \Re(\Log\rho\la), \Re(\Log\rho\nn)
$$
from which the second line of inequalities follows.
\end{proof}

\begin{lemma}\label{rhos}
  There exists $\hat C>0$ such that, 
  $\forall \lambda \in \Mpqr(\whR)$ and $ \nu=\xi(\lambda) $ we have
  $$
  \left|\frac{2\pi i}{Log(\rho\la)}-\frac{2\pi i}{Log(\rho\nn)}\right|<\hat C.
  $$
\end{lemma}
\begin{proof}
It follows from Proposition 5 in \cite{LP} and Theorem 1.4 in \cite{K1} (see also Theorem 5.7 in \cite{K}) that there exists a  $D>0$ such that, for all $\lambda \in \Mpqr$ and $\nu=\xi(\lambda)$,
 $$\Big|\frac{1}{\Lambda}-\frac{1}{N}\Big|<D.$$
Then the result follows from the previous Lemma \ref{difmult}.
 \end{proof}

The Buff-translation structures on 
$W_{\lambda, r_0}:=\Dbar(0, r_0)\setminus\{\textrm{$f\la$-fixed points}\}$ and $W_{\nu, r_0}:=\Dbar(0, r_0)\setminus\{\textrm{$g\nn$-fixed points}\}$ respectively are given by the integrating coordinates for the vector fields $v\la(z) d/dz$ and $v\nn(z) d/dz$
$$
\varphi\la(z) = \varphi_{\lambda,z_0}(z) = \int^z_{z_0} \omega\la(\zeta) d\zeta, \qquad
\varphi\nn(z) = \varphi_{\nu,z_0}(z) = \int_{z_0}^z \omega\nn(\zeta) d\zeta.$$
A priori these maps depend on a choice of base point $z_0$ for the integration 
and they are only defined on simply connected subsets 
of $W_{\lambda, r_0}$ and $W_{\nu, r_0}$ respectively. 
But they admit globally defined lifts to universal coverings of the respective domains and 
multivalued holomorphic extension by analytic continuation along curves with periods, which can be calculated from the residues of the respective $1$-forms and the homotopy class of closed curve connecting a point to itself within the domains. 
In particular the change of charts is a translation in the co-domain. Thus defining a translation structure, 
the Buff-translation structure on the respective domains $W_{\lambda, r_0}$ and $W_{\nu, r_0}$. 

An elementary computation shows that for a holomorphic map $f$ with $|f'(z) -1| < 1$ 
on an open subset $U\subset\C$
\begin{align*}
\omega_f(z) &= \frac{f'(z) -1}{(f(z) - z)\Log(1+(f'(z)-1))}\\ 
&= 
\frac{1 +\frac{1}{2}( f'(z) -1)+ \OO((f'(z) -1)^2)}{f(z) - z}\\
&=  \frac{1}{f(z) - z} + \frac{1}{2}\frac{f'(z) -1}{f(z) - z} + \OO\left(\frac{(f'(z) -1)^2}{f(z) - z}\right).
\end{align*}
From these fomulae it follows immediately that  $\omega_{\lambda_0}$ and $\omega_{\nu_0}$ 
can be expressed as Laurent series with dominant polar term $1/qz^{q+1}$. 
As a consequence the Buff forms 
$\omega_{\lambda_0}dz$ and $\omega_{\nu_0}dz$ have anti-derivatives: 
\begin{align*}
\varphi\l0(z) &= Resit(f_{\lambda_0},0)\log z - \frac{1}{qz^q} + \OO(\frac{1}{z^{q-1}}), 
\qquad 0<]z| \leq r_0\\
\varphi\n0(z) &= Resit(g_{\nu_0},0)\log z - \frac{1}{qz^q} + \OO(\frac{1}{z^{q-1}}),
\qquad 0<]z| \leq r_0
\end{align*}
Possibly reducing $r_0>0$ we can assume 
that 
$$
\max\{C_1, 10\} < \frac{1}{2}\frac{1}{qr_0^q},
$$
and for some constant $C_2>0$
$$
|qz^q\e^{iqt}\varphi_\ilnz(z\e^{it}) + 1|, 
\leq C_2|z|
$$
for all $z, |z|\leq r_0$ $t, |t|\leq\pi$ and $\ilnz = \lambda_0, \nu_0$.  
Thus possibly reducing $r_0$ we can suppose 
\begin{equation}\label{basic_coordinate_estimates}
|\varphi_\ilnz(z\e^{it}) + \frac{1}{qz^q\e^{iqt}}| \leq \frac{1}{20}\frac{1}{q|z|^q}
\end{equation}
for all $z, |z|\leq r_0$ and $t, |t|\leq\pi$ and $\ilnz = \lambda_0, \nu_0$.

Possibly reducing $r_0>0$ further we can assume
that for each $0<r\leq r_0$ the image curves $\varphi_\ilnz(r\e^{it})$ 
have negative signed curvature, and thus that the argument 
of the tangent vector is decreasing, and it decreases by $q2\pi$, 
when $r\e^{it}$ makes one round, i.e. when $t$ increases by $2\pi$. 
In particular, for each $\ilnz = \lambda_0, \nu_0$, 
there exist precisely $q$ points $y^j_\ilnz = y^j_\ilnz(r_0) = r_0\e^{it}$, $t = t_\ilnz^j$, 
of norm $r_0$ where the argument function of $d/dt(\varphi_\ilnz(r_0\e^{it})) =  ir_0\e^{it}/v_\ilnz(r_0\e^{it})$ 
is zero modulo $2\pi$. 
This means that there are precisely $q$ points 
$y^j_\ilnz$ of norm $r_0$, where the tangent vector
$iy^j_\ilnz/v_\ilnz(y^j_\ilnz)$ at $\wty^j_\ilnz = \varphi_\ilnz(y^j_\ilnz)$ is real and positive. 
The points $y^j_\ilnz$ are interlaced with
$q$ points $w^j_\ilnz = w^j_\ilnz(r_0)$ of norm $r_0$, 
where the tangent vector $iw^j_\ilnz/v_\ilnz(w^j_\ilnz)$ at $\wtw^j_\ilnz = \varphi_\ilnz(w^j_\ilnz)$ 
is real and negative, 
$q$ points $x^{j+}_\ilnz = x^{j+}_\ilnz(r_0)$ of norm $r_0$, where the tangent vector 
$ix^{j+}_\ilnz/v_\ilnz(x^{j+}_\ilnz)$ at 
$\wtx^{j+}_\ilnz = \varphi_\ilnz(x^{j+}_\ilnz)$
is purely imaginary and points upwards, 
and $q$ points $x^{j-}_\ilnz = x^{j-}_\ilnz(r_0)$ of norm $r_0$, 
where the tangent vector 
$ix^{j-}_\ilnz/v_\ilnz(x^{j-}_\ilnz)$ at 
$\wtx_\ilnz^{j-} = \varphi_\ilnz(x^{j-}_\ilnz)$is purely imaginary and points downwards (see Figure \ref{fig:BF}).
\begin{figure}[tbh]
    \centering
    \includegraphics[scale = 0.5]{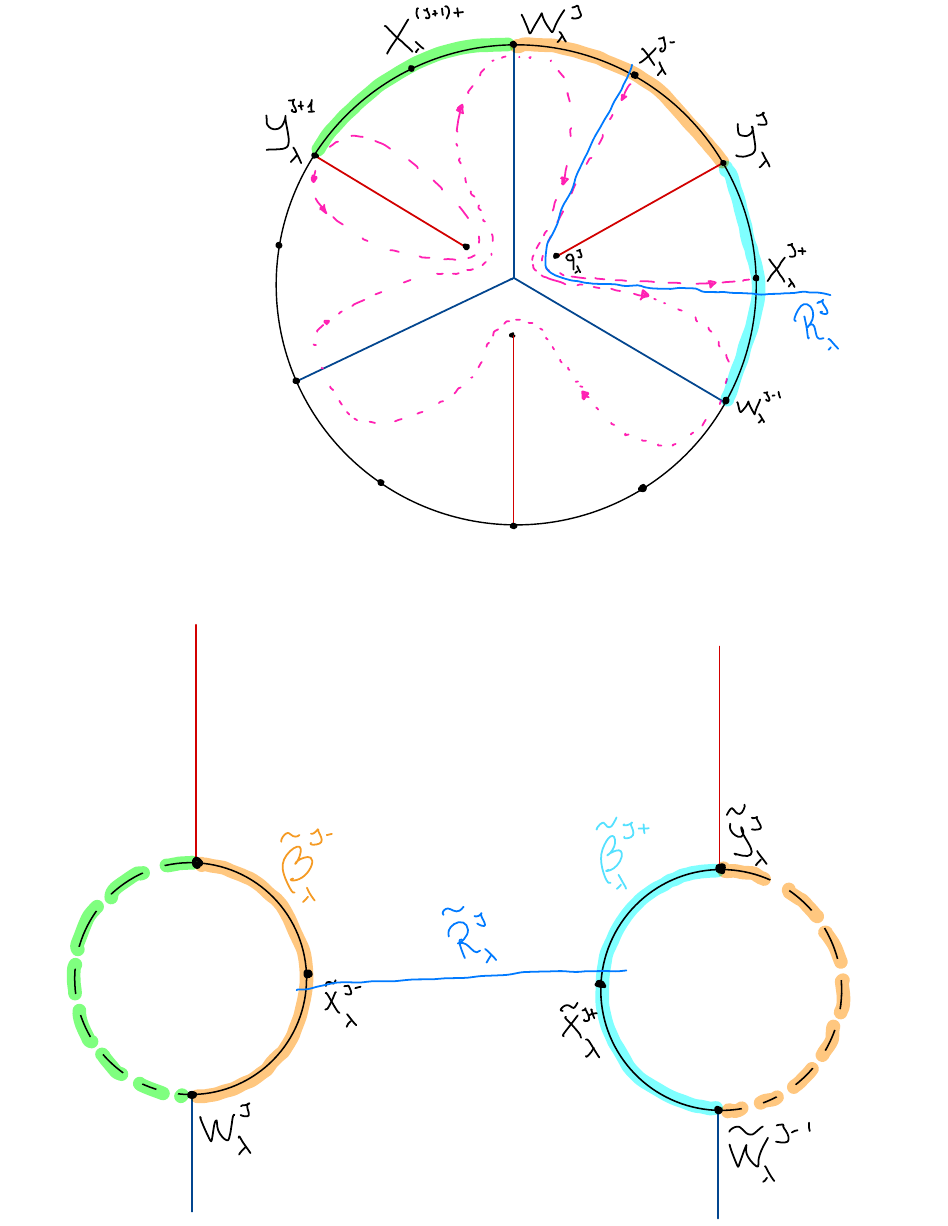}
    \caption{The points $y\la^j,\,\,w\la^j$ and $x\la^j$.}
    \label{fig:BF}
\end{figure}
The indexing being in the counter clockwise direction with $y^0_\ilnz$ 
in the sector bounded by $\overline{R^0_\ilnz\cup R^1_\ilnz}$, 
with $w^0_\ilnz$ between $y^0_\ilnz$ and $y^1_\ilnz$ and 
with $w^{j-1}_\ilnz<x^{j+}_\ilnz<y^j_\ilnz<x^{j-}_\ilnz<w^j_\ilnz$ for $0<j<q$. 
Then $x^{j+}_\ilnz$, $x^{j-}_\ilnz$ are the points on the circle, where the vector field is orthogonal to 
the circle (see Figure \ref{fig:BF}). These points lie asymptotically in the repelling and attracting directions respectively, 
so that, possibly reducing $r_0>0$, we can assume 
$$
\left|\Arg\left(\frac{x^{j+}_\ilnz}{\e^{(j-1)2\pi/q}}\right)\right|, 
\left|\Arg\left(\frac{x^{j-}_\ilnz}{\e^{(2j-1)\pi/q}}\right)\right|\leq \frac{2\pi}{24q}
\qquad\textrm{for all } 0<j<q.
$$
The functions $\varphi_\ilnz$, $\ilnz = \lambda_0, \nu_0$, 
define Riemann surfaces such that 
the inverse of $\varphi_\ilnz$ are universal coverings of $\D^*(0,r_0) = W_{\ilnz,r_0}$.\\

For $\lambda\not=\lambda_0$ and $\nu\not=\nu_0$ the situation is more complicated, 
but similar for the two families $f\la$ and $g\nn$. 
Around the fixed point $0$ there are topological punctured disks $D_\iln^*(0)$, 
corresponding to topological lower half planes $H_-$, and universal covering maps 
$\psi_\iln : H_- \to D_\iln^*(0)$ with deck-transformations $Z\mapsto Z + 2\pi i/\Lambda$ 
and $Z\mapsto Z + 2\pi i/N$ respectively, 
and such that $\varphi_\iln\circ\psi_\iln(Z) = Z$ for the appropriate choice of 
analytic continuation of $\varphi_\iln$. 
Similarly around any point $q_\iln^j$ in the $q$-cycle there are 
topological punctured disks $D^*(q_\iln^j)$, corresponding to topological upper half planes $H_+$, 
and universal covering maps $\psi_\iln^j : H_+ \to D^*(q_\iln^j)$ 
with deck-transformations $Z\mapsto Z + 2\pi i/Log(\rho_\iln)$,
and such that $\varphi_\iln\circ\psi_\iln(Z) = Z$ for the appropriate choice of 
analytic extension of $\varphi_\iln$. 
When $\lambda\in\Delta(\whR)$, $\nu\in\Xi(\whR)$
such upper and lower half-planes are connected by vertical strips, and 
there are common determinations of $\psi_\iln$ on such half planes and connecting strips.
To explain the picture, we will take an excursion into parabolic implosion.

\subsubsection{Parabolic implosion in a nut shell}\label{nutshell}
The connection between the parabolic limits $f_{\lambda_0}$ and $g_{\nu_0}$ and the 
nearby maps $f\la$ and $g\nn$ respectively can be understood via the theory of parabolic implosion. 
For $f\l0(z)=z+z^{q+1}+O(z^{q+2})$ and $0\leq j < q$
let $\psi_{\lambda_0,rep}^j$ be a repelling Fatou parameter for the petal containing the lower end 
of the ray $R\l0^j$ and let $\phi_{\lambda_0,att}^j$ be an attracting Fatou coordinate on $B\l0^j$. 
Thus the associated petals are neighbours and the attracting petal for $\phi_{\lambda_0,att}^j$, 
which is a subset of $B\l0^j$, 
follows the repelling petal for $\psi_{\lambda_0,rep}^j$ 
in the counter clockwise direction around $\alpha_0$.
To simplify statements we shall suppose these normalized at least partially, 
so that $\phi_{\lambda_0,att}^{(j+p)\mod q}\circ P\l0 = \frac{1}{q} + \phi_{\lambda_0,att}^j$ 
and $P\l0\circ\psi_{\lambda_0,rep}^j = \psi_{\lambda_0,rep}^{(j+p)\mod q}(z+\frac{1}{q})$. 

By Lavaurs theorem (\cite{La}), if $\rho_n\to 1$, $k_n \rightarrow \infty$ and 
$k_n -\frac{2\pi i}{Log \rho_n} \rightarrow \delta \in \C$, 
then $\lambda_n = \lambda(\rho_n) \rightarrow \lambda_0$ and 
$f_{\lambda_n}^{k_n} \rightarrow L^j_{\delta} = L^j_{\lambda_0,\delta}:= \psi_{\lambda_0,rep}^j\circ T_{\delta}\circ \phi_{\lambda_0,att}^j$ 
uniformly on compact subsets of the parabolic basin $B^j\l0$. 
Moreover for each $\delta$ it satisfies 
$P\l0\circ L_\delta = L_\delta\circ P\l0$. 
In this case, also the dynamical rays converge. 
That is, if $f_{\lambda_n}^{k_n} \rightarrow L^j_{\delta}$, then $R_{\lambda_n}^j$  
converges Hausdorff to a generalized ray 
$R_\delta^j = R^j_{\lambda_0,\delta} = f_{\lambda_0}(R_\delta^j)$, 
where $R_\delta^j$ consists of $R\l0^j$ and a collection of $q$ Hawaiian earrings, 
one in each $B^{j'}\l0$,  $0\leq j' < q$,
(see \cite{La}, Proposition 4.1.7, \cite{K1}, Proposition 3.3 and \cite{PZ1}, Thm A \& B). 
Note that for $(\lambda_n)_n$ as above $f_{\lambda_n}^{k_n+1} \to L^j_{\delta+1} = f\l0\circ L^j_\delta$.
Hence possibly adjusting $k_n$ by a uniformly bounded sequence we shall assume $0\leq\Re(\delta)<1$.

In more detail, the connected components of $R_\delta^j\setminus\{\al\}$ are 
$R^j_{\lambda_0}$ and 'homoclinic' arcs in the parabolic basin of $\al=0$. 
Any two distinct homoclinic arcs are disjoint and 
each homoclinic arc $l$ is horo-cycle-like with base point $\al$, so $l\cup\{\al\}$ is a loop, bounding a disk in one of the $q$ connected components of the immediate parabolic basin. 
Loops in any fixed component of the immediate parabolic basin are nested and 
form a 'Hawaiian earring' consisting of infinitely many loops. 
We will then call a homoclinic arc $l$ a 'single earring'.
On each single earring $l$ the map $f\la$ acts like a stabilizing parabolic and 
the map $L_\delta$ acts like a dilation shifting single earrings one level out. 

In further detail, there is a univalent restriction of $\phi^j_{\lambda_0,att}$ 
defined on a neighbourhood in $B^j_{\lambda_0}$ of the relatively closed disk bounded 
by the loop of the outermost single earring $R_{\delta}^{j,1} = R_{\lambda_0, \delta}^{j,1}$ 
of $R^j_\delta\cap B^j\l0$ with image a $T_1$-invariant upper half plane. 
It sends each individual single earring to a $T_1$-invariant curve 
and in particular it sends $R_{\delta}^{j,1}$ to the boundary of the upper half plane. 
Moreover the translation $T_\delta$ shifts such images curves one level down and sends $R_{\delta}^{j,1}$ 
to the $T_1$ invariant connected component of $\psi_{\lambda_0,rep}^{-1}(R^j_{\lambda_0})$. 
In particular $L^j_\delta =
\psi^j_{\lambda_0,rep}\circ T_{\delta}\circ \phi^j_{\lambda_0,att}$ 
maps the single earring $R_{\delta}^{j,1}$ diffeomorphically onto the ray $R^j\l0$ and the parametrization of the ray by potential 
$t\mapsto R^j\l0(t)$ induces a parametrization 
$t\mapsto R_{\delta}^{j,1}(t)$ by 
$L^j_\delta(R_{\delta}^{j,1}(t)) = R^j\l0(t)$.
Moreover, the functional relation 
$f\l0^q(R_{\delta}^{j,1}(t)) = R_{\delta}^{j,1}(2^qt)$ 
persists.

The single earring $R_{\delta}^{j,1}$ leaves (in the direction of increasing potential) 
the fixed point $0$ tangentially to the repelling direction 
$2j\pi/q$ and arrives at $0$ again tangentially to the attracting direction $(2j-1)\pi/q$. 
In particular, possibly decreasing $r_0>0$ we can suppose that 
each ray $R^j\l0$ enters $\D(0, r_0)$ (in the direction of decreasing potential)  with an argument deviating from $2(j-1)\pi/q$ by at most 
$2\pi/(24q)$ and for any $\delta$ with $0<\Im(\delta) < 2C_1$ (where $C_1$ is the constant from Lemma~\ref{difmult}) the single earring $R_{\delta}^{j,1}$ exits $\D(0, r_0)$ (in the direction of decreasing potential) with an argument deviating from $(2j-1)\pi/q$ by at most $2\pi/(24q)$. 
Each of the rays $R^j_{\lambda_n}$ in the $q$-cycle of rays gives rise to similar limits 
$R_\delta^j$. And in each immediate component of the parabolic basin we have $q$ 
interlaced Hawaiian earrings. 

In what follows we shall only be interested in the sub-segments of the rays $R_{\lambda_n}^j$ 
which approach (part of) the outermost single earring $R_{\delta}^{j,1}$
in the Hawaiian earring of $R_\delta^j$. 
More precisely, fix some $t_0>0$ and any compact interval 
$I\subset ]0, \infty[$ containing $t_0$. 
Then $t\mapsto R_{\lambda_n}^j(t/2^{qk_n})$ 
converges uniformly to $t\mapsto R_{\delta}^{j,1}(t)$ on $I$.

All of the above are found similarly for $g\nn$ with Lavaurs map $L^j_{\nu_0, \delta}$ 
and limit rays $R^j_{\nu_0, \delta}$ and $R^{j,1}_{\nu_0, \delta}$.\\

After this small excursion into Lavaurs theory of parabolic implosion, we return to the main vein.
In what follows we shall work mostly in the near-Douady-Fatou coordinates given by 
the rectifying coordinates for the Buff vector fields instead of the actual Douady-Fatou coordinates (see \cite{Ch} and \cite{PZ}). 
We start by defining useful domains for their inverses, the near-Douady-Fatou parameters.

Let $r_1$ be arbitrary with $0< r_1 \leq r_0$. 
Since $v\la d/dz$ depends holomorphically on $\lambda\in\D(\lambda_0,\whR)$ and 
$v\l0$ is zero free in the punctured disk $\D(r_0)\setminus\{0\}$
we can suppose, possibly reducing $\whR>0$, that $v\la$ is zero free on the annulus 
$\Dbar(0,r_0)\setminus\D(0,r_1)$,  
that $||\omega_\lambda-\omega_{\lambda_0}|| < 1/(2\pi r_0)$ on this annulus, and that for $r\in[r_1, r_0]$ the curvature of $\varphi\la(r\e^{it})$ is negative, 
and that this also holds for $v\nn, \omega\nn$ and $\varphi\nn$.
Thus for $\lambda\in\D(\lambda_0, \whR)$, $\nu\in\D(\nu_0, \whR)$ and $\iln = \lambda, \nu$, there are unique continuous determinations 
$y^j_\iln(r), w^j_\iln(r), x^{j+}_\iln(r), x^{j+}_\iln(r)\in S_{r}$, $r\in[r_1, r_0]$, $1\leq j \leq q$, 
with $iy^j_\iln(r)/(v_\iln(y^j_\iln(r))$ real and positive, 
$iw^j_\iln(r)/v_\iln(w^j_\iln(r))$ real and negative, 
$ix^{j+}_\iln(r)/v_\iln(x^{j+}_\iln(r))$ purely imaginary and pointing upwards and  
$ix^{j-}_\iln(r)/v_\iln(x^{j-}_\iln(r))$ purely imaginary and pointing downwards. 
By the proximity of $\omega_\iln$ to $\omega_\ilnz$ we can assume, possibly reducing $\whR>0$, that 
$$
\left|\Arg\left(\frac{x^{j+}_\iln}{x^{j+}_\ilnz}\right)\right|, 
\left|\Arg\left(\frac{x^{j-}_\iln}{x^{j-}_\ilnz}\right)\right|\leq \frac{2\pi}{24q}
\qquad\textrm{for all } 0<j<q.
$$

Choosing $r_1>0$ sufficiently small, $\whR$ accordingly, and normalizing $\varphi_\iln$ 
($\iln = \lambda, \nu$, with
$\lambda\in\D(\lambda_0,\whR)$ and $\nu\in\D(\nu_0,\whR)$)
by $\varphi_\iln(y^j_\ilnz) = \varphi_\ilnz(y^j_\ilnz)$, 
we obtain from the proximity of $\omega_\iln$ to $\omega_\ilnz$ and the bounds \eqref{basic_coordinate_estimates} that, possibly reducing $\whR>0$, 
\begin{equation}\label{coordinate_estimates}
|\varphi_\iln(z\e^{it}) + \frac{1}{qz^q\e^{iqt}}| \leq \frac{1}{10}\frac{1}{q|z|^q}
\end{equation}
for all $z, r_1 \leq |z| \leq r_0$ and $t, |t|\leq\pi$. 

For $\lambda\in\Delta(\whR)$, $\nu\in\Xi(\whR)$, and 
$\iln = \lambda, \nu$, the integral curves of the vector fields $iv_\iln(z) d/dz$, $\e^{i\pi/4}v_\iln(z) d/dz$ 
and $\e^{i3\pi/4}v_\iln(z) d/dz$ starting at the point $y^j_\iln = y^j_\iln(r_0)$ all converge to $q^j_\iln$ 
as time increases to $\infty$. 
The images of these curves under $\varphi_\iln$ 
are the upward pointing half-lines from $\wty^j_\iln = \varphi_\iln(y^j_\iln)$ 
with direction vectors $i$, $\e^{i\pi/4}$ and $\e^{i3\pi/4}$ respectively. 
We denote by $\wtl_\iln^j(r_0)$ the vertical half line $\wty^j_\iln+i\R_+$. 
On the other hand, at $\iln=\ilnz = \lambda_0, \nu_0$, these integral curves all converge to 
$0 = \alpha_\ilnz$ as time increases to $\infty$. 

Similarly, the integral curves for the vector fields $-iv_\iln(z) d/dz$, $\e^{-i\pi/4}v_\iln(z) d/dz$ 
and $\e^{-i3\pi/4}v_\iln(z) d/dz$ starting at $w^j_\iln = w^j_\iln(r_0)$ all
converge to $0$ as time increases to $\infty$, and their images under $\varphi_\iln$ 
are the downwards pointing half-lines from $\wtw^{j-1}_\iln = \phi_\iln(w^{j-1}_\iln)$ 
with direction vectors $-i$, $\e^{-i\pi/4}$ and $\e^{-i3\pi/4}$. 
Recall that $w^{j-1}_\iln$ and $w^j_\iln$ are neighbours of $y^j_\iln$ on the circle $S=S_r,\,\,r \in [r_1,r_0]$.
The analytic continuation of $\varphi_\iln$ 
along the arc of circle $[y^j_\iln, w^{j-1}_\iln]_S$ from $y^j_\iln$ to $w^{j-1}_\iln$ 
maps this arc to a curve $\wtbe^{j+}_\iln$, which connects $\wtw^{j-1}_\iln$ to $\wty^j_\iln$ 
and which is close to the left semi-circle of radius $1/(qr^q)$ and center $\wty^j_\iln-i/(qr^q)$. 
And the analytic continuation of $\varphi_\iln$ along the arc of circle $[y^j_\iln, w^j_\iln]_S$ 
from $y^j_\iln$ to $w^j_\iln$ maps this arc to a curve close to the right semi-circle 
of radius $1/(qr^q)$ and center $\wty^j_\iln-i/(qr^q)$. 
Denote by $\wtbe^{j-}_\iln(r_0)$ the curve $\varphi_\iln([y^j_\iln, w^j_\iln]_S)-\frac{-2\pi i}{\Log(\rho_\iln)}$. 

For a(ny) choice of base point $z_0$ we construct 
plane separating Jordan arcs $\Gamma_\iln^{j\pm}$ and $\Upsilon_\iln^{j\pm}$ by (see Figure \ref{fig:G})
\begin{align*}
\Gamma_\iln^{j+} &:= (\e^{i\pi/4}\R_+ + \wty^j_\iln)\cup \wtbe^{j+}_\iln \cup 
(\wtw^{j-1}_\iln+\e^{-i\pi/4}\R_+),\\
\Upsilon_\iln^{j+} &:= \wtl^j_\iln \cup \wtbe^{j+}_\iln \cup 
(\wtw^{j-1}_\iln-i\R_+),
\end{align*}
for $\iln\not= \iln_0$, i.e.~$\lambda\in\Mpqr(\whR)$ and $\nu\in\Mppqr(\whR)$ we further define
\begin{align*}
\Gamma_\iln^{j-} &:= T_{-\frac{i2\pi}{\Log\rho_\iln}}\left((\e^{i3\pi/4}\R_+ + \wty^j_\iln)
\cup 
(\wtw^{j-1}_\iln +\e^{-i3\pi/4}\R_+)\right)\cup \wtbe^{j-}_\iln ,\\
\Upsilon_\iln^{j-} &:= T_{-\frac{i2\pi}{\Log\rho_\iln}}
\left(\wtl^j_\iln \cup (\wtw^{j-1}_\iln-i\R_+)\right) \cup \wtbe^{j-}_\iln,
\end{align*}
where as for $\ilnz = \lambda_0, \nu_0$ we define $\wtbe^{j-}_\ilnz(r_0)$ to be the curve 
$\varphi_\ilnz([y^j_\ilnz, w^j_\ilnz]_S)$ and 
\begin{align*}
\Gamma_\ilnz^{j-} &:= (\e^{i3\pi/4}\R_+ + \wty^j_\ilnz) \cup \wtbe^{j-}_\ilnz \cup 
(\varphi_\ilnz(w^j_\ilnz)+\e^{-i3\pi/4}\R_+),\\
\Upsilon_\ilnz^{j-} &:= \wtl^j_\ilnz \cup \wtbe^{j-}_\ilnz \cup (\varphi_\ilnz(w^j_\ilnz)-i\R_+).
\end{align*}
\begin{figure}[tbh]
    \centering
    \includegraphics[scale = 0.5]{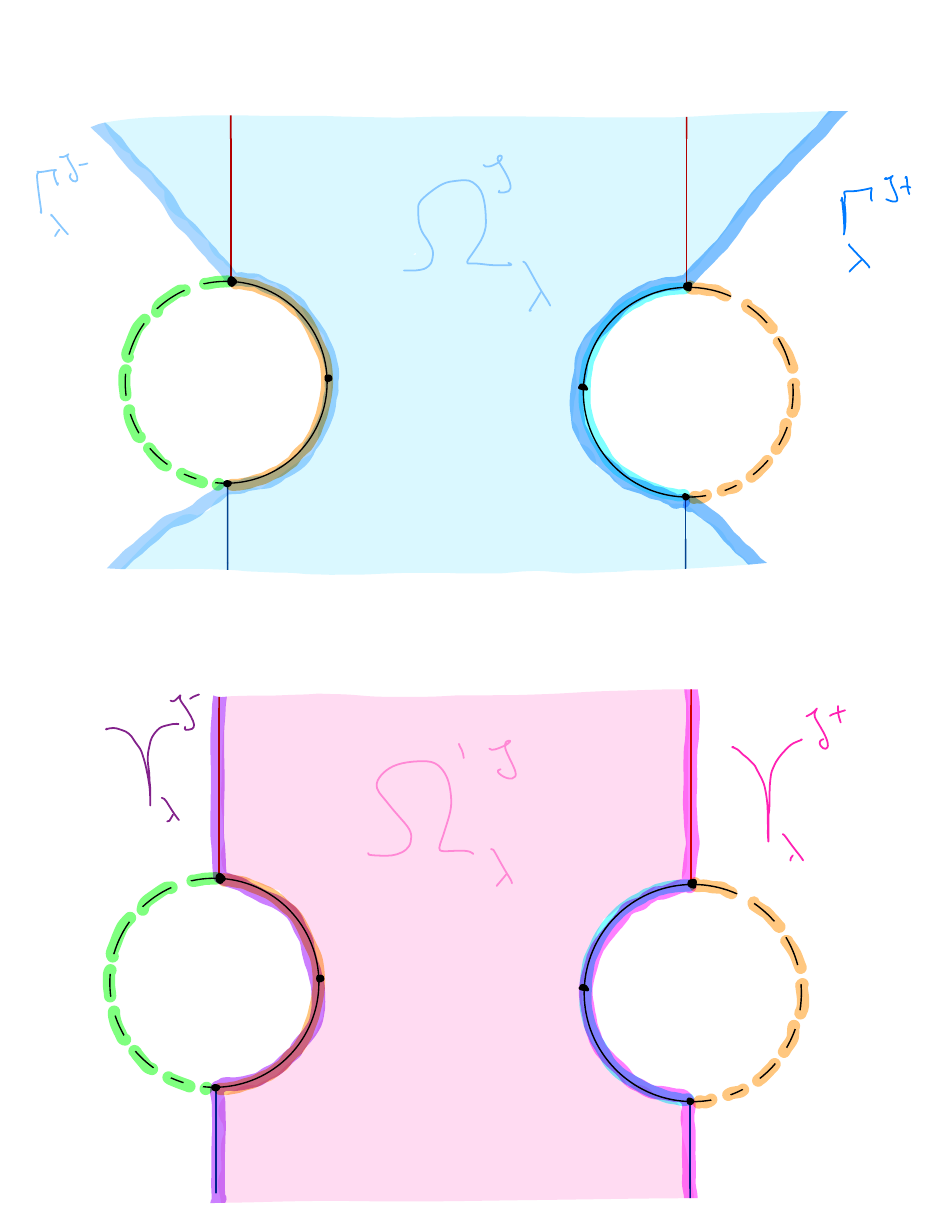}
    \caption{The domains $\Omega^j_\iln$ (above) and $\Omega^{'j}_\iln$ (below). The width of $\Omega^j_\iln$ and $\Omega^{'j}_\iln$ should be much bigger compared to the ray of the (approximate) discs bounding them, and the boundaries of $\Omega^j_\iln$ and $\Omega^{'j}_\iln$ should be more wiggling than what we managed to illustrate.}
    \label{fig:G}
\end{figure}
For $\lambda\in\Mpqr(\whR)$, $\nu\in\Mppqr(\whR)$ and $\iln = \lambda, \nu$ 
the pairs of arcs $\Gamma_\iln^{j+}$, and $\Gamma_\iln^{j-}$ bound a 'butterfly' domain $\Omega^j_\iln$, while the pairs of arcs
 $\Upsilon_\iln^{j+}$, and $\Upsilon_\iln^{j-}$ 
bound a 'vertical strip' subdomain $\Omega^{'j}_\iln\subset\Omega_\iln$ (see Figure \ref{fig:G}). 
On the other hand, for $\ilnz = \lambda_0, \nu_0$, the arcs $\Upsilon_\ilnz^{j+}$, $\Gamma_\ilnz^{j+}$ bound 
left half planes $\Omega^{'j+}_\ilnz\subset\Omega^{j+}_\ilnz$ and 
the arcs $\Upsilon_\ilnz^{j-}$, $\Gamma_\ilnz^{j-}$ bound 
right half planes $\Omega^{'j-}_\ilnz\subset\Omega^{j-}_\ilnz$. 
On any of the simply connected domains $\Omega^j_\iln$, $\Omega^{j+}_\ilnz$ and $\Omega^{j-}_\ilnz$
there exist a unique holomorphic local inverse $\psi_\iln$ of $\varphi_\iln$, and respectively $\psi_\ilnz$ of $\varphi_\ilnz$:
$$
\psi^j_\iln: \Omega^j_\iln \to \D(0,r_0), \qquad
\psi^{j+}_\ilnz: \Omega^{j+}_\ilnz \to \D(0,r_0), \qquad
\psi^{j-}_\ilnz: \Omega^{j-}_\ilnz \to \D(0,r_0), 
$$
with $\varphi_\iln\circ\psi^j_\iln(Z) = Z$ and $\varphi_\ilnz\circ\psi^{j\pm}_\ilnz(Z) = Z$ 
for the appropriate analytic extensions of $\varphi_\iln$ and $\varphi_\ilnz$ respectively. 
Note that $\psi^j_\iln$ is locally univalent, univalent on $\Omega^{'j}_\iln$, 
but not globally injective on $\Omega^j_\iln$. 
Note that a change in base point for $\varphi_\iln$, $\iln = \lambda, \lambda_0, \nu, \nu_0$ 
amounts to a translation of these domains. 
In particular the shape/geometry of the domains is independent of choice of base point. 
Also, since $x^{j+}\in[y^j_\iln, w^{j-1}_\iln]_S$ and $x^{j-}\in[y^j_\iln, w^j_\iln]_S$, 
the above also determines points $\wtx^{j\pm}_\iln\in\beta^{j\pm}_\iln$ with 
$\psi^j_\iln(\wtx^{j\pm}_\iln) = x^{j\pm}_\iln$. 

Possibly reducing $r_1>0$, we can suppose that $f\la(\Dbar(0,2r_1))\subset \D(0,r_0)$ for all 
$\lambda\in \D((\lambda_0, \whR)$ and that $g\nn(\Dbar(0,2r_1))\subset \D(0,r_0)$ for all 
$\nu\in \D(\nu_0, \whR)$. 
For $\lambda\in\Delta(\whR)\cup\{\lambda_0\}$, $\nu\in\Xi(\whR)\cup\{\nu_0\}$, $\iln = \lambda, \nu$ 
and for $r\in[r_1, r_0]$ we may thus construct 
$\Omega^j_\iln(r) \subset \Omega^j_\iln(r_0) := \Omega_\iln^j$ 
and $\Omega^{'j}_\iln(r) \subset \Omega^{'j}_\iln(r_0):=\Omega^{'j}_\iln$ 
analogously to $\Omega^{'j}_\iln\subset\Omega^j_\iln$, 
such that the restriction of $\psi^j_\iln: \Omega^j_\iln \to \D(0,r_0)$ to $\Omega^j_\iln(r)$ 
remains a local inverse of $\varphi_\iln$ (note that for $\iln = \lambda_0, \nu_0$ there are 
actually two domains, with super script $\pm$ for each $j$). 
Moreover applying \eqref{coordinate_estimates} 
we can assume that the distance between $\Omega^j_\iln(r_1)$ and the boundary 
of $\Omega^j_\iln(2r_1)$ is at least $4$.\\ 

Define holomorphic functions by
$$
1+u_\iln(z) = \int_{[z,f_\iln(z)]} \omega_\iln(\zeta) d\zeta, 
\qquad w_\iln(z) = u_\iln'(z)/\omega_\iln(z)
\qquad ]z| < 2r_1
$$
where $u_\iln(z)$ and $w_\iln(z)$ extend holomorphically to the fixed points with value $0$. 
Moreover $u_\iln$ and $w_\iln$ depend holomorphically on $\lambda\in\Delta(\whR)$ and 
continuously on $\lambda$ at $\lambda_0$, and respectively 
holomorphically on $\nu\in\Xi(\whR)$ and continuously on $\nu$ at $\nu_0$.
Then, possibly reducing $r_1>0$, we can suppose that the functions $u_\iln, w_\iln$ 
are uniformly bounded by, say, $1/4$.
Lift these functions to functions 
$U^j_\iln(Z) = u_\iln(\psi^j_\iln(Z))$, $W^j_\iln(Z) = w_\iln(\psi^j_\iln(Z))$ on 
$\Omega^j_\iln(2r_1)$, and $\Omega^{j\pm}_\iln(2r_1)$ respectively,
$\Omega^j\nn(C2r_1)$, 
and define the functions 
$$
F^j\la : \Omega^j\la(2r_1) \to \Omega^j\la(r_0)\qquad\textrm{and}\qquad
G^j\nn : \Omega^j\nn(2r_1) \to \Omega^j\nn(r_0)
$$ 
by 
$$
F^j\la(Z) = Z + 1 + U^j\la(Z)\qquad\textrm{and}\qquad
G^j\nn(Z) = Z + 1 + U^j\nn(Z)
$$ with derivatives 
$$
(F^j\la)'(Z) = 1 + W^j\la(Z)\qquad\textrm{and}\qquad
(G^j\nn)'(Z) = 1 + W^j\nn(Z)
$$ respectively.
Then $\psi^j\la\circ F^j\la = f\la\circ\psi^j\la$ and $\psi^j\nn\circ G^j\nn = g\nn\circ\psi^j\nn$, 
where both sides are defined. 
From the above bounds on the functions $u\la, w\la, u\nn$ and $w\nn$, and hence for 
$U^j\la, W^j\la, U^j\nn$ and $W^j\nn$, 
it follows that there exist globally defined Douady-Fatou coordinates 
$\Phi^j\la$ and $\Phi^j\nn$ for the functions $F^j\la$ and $G^j\nn$, 
where they are defined. 
Moreover these coordinates are uniformly close to translatations in the following sense: 

\begin{prop}\label{comparison}
There exists a constant $K = K(\whR,r_0, r_1)$ such that, 
for all $\lambda \in \Delta(\whR)\cup\{\lambda_0\}$ and $\nu\in\Xi(\whR)\cup\{\nu_0\}$, for all $j$,
 and for any $z\la^j\in\Omega^j\la(r_1)$ and $z\nn^j\in\Omega^j\nn(r_1),$ 
the Douady-Fatou coordinates $\Phi^j\la$ for $F^j\la$ and 
$\Phi^j\nn$ for $G^j\nn$, normalized by $\Phi^j\la(z^j\la) = z^j\la$ and $\Phi^j\nn(z^j\nn)=z^j\nn$, 
satisfy
$$
||\Phi^j\la(z) - z||, ||\Phi^j\nn(z) - z||  \leq K
\quad\textrm{and}\quad 
||(\Phi^j\la)'-1||, ||(\Phi^j\nn)'-1|| < \frac{1}{4}
$$ 
on $\Omega^j\la(r_1)$ and $\Omega^j\nn(r_1)$ respectively. 
\end{prop}
\begin{proof}
It follows from Proposition 30 and Corollary 31 in \cite{Ch}. See also \cite{PZ}.
\end{proof}

Since by choice of $r_0>0$ none of the outermost single earrings in the $q$ Hawaiian earrings of
$R_\delta^j$ are compactly contained in $\D(r_0)$, 
and since $0\leq \Im(\frac{i2\pi}{\Log\rho\la}) < C_1$ in $\Mpqr(\whR)$,
we can assume, possibly reducing $\whR>0$, that
for any $\lambda\in\Mpqr(\whR)$
the ray $R\la^j$ (for $0<j<q$) enters and exists the disc $\D(r_1)$ at least once, and so it
cuts out a domain $A\la^j\subset \D(r_1)\setminus R\la^j$ that contains $q\la^j$ and no other fixed point for $f\la$ (see Figure \ref{fig:bf}). 
Recall that for $\lambda\in\Delta(\whR)$, $\nu\in\Xi(\whR)$ and $\iln = \lambda, \nu$ 
the curve $\psi^j_\iln(\wtl^j_\iln)$ is the trajectory of the vector field $+iv_\iln(z) d/dz$ emanating from 
$y_\iln^j(r_0)$ and converging to $q_\iln^j$ as time tends to $\infty$. 
The set $A^j_\iln\setminus\psi_\iln(\wtl_\iln^j)$ has a univalent preimage
$\wtA^j_\iln\subset\Omega^{'j}_\iln(r_1)$ under $\psi^j_\iln$ (see Figure \ref{fig:bf}). 
The set $\wtA^j_\iln$ is bounded below by the connected component $\wtR^j_\iln$ 
of $(\psi^j_\iln)^{-1}(R^j_\iln)$ connecting $\wtbe^{j+}_\iln(r_1)\subset\Upsilon^{j+}_\iln(r_1)$ 
and $\wtbe^{j-}_\iln(r_1)\subset\Upsilon^{j-}(r_1)$, and on the sides by the parts of $\wtbe^{j+}_\iln(r_1)$ and
$\wtbe^{j-}_\iln(r_1)$ above the touching points of $\wtR_\iln$ 
together with $\wtl^j_\iln(r_1)$ and $T_{\frac{-i2\pi}{\Log(\rho_\iln)}}(\wtl^j\la(r_1))$ respectively (see Figure \ref{fig:bf}).

\begin{figure}[tbh]
    \centering
    \includegraphics[scale = 0.5]{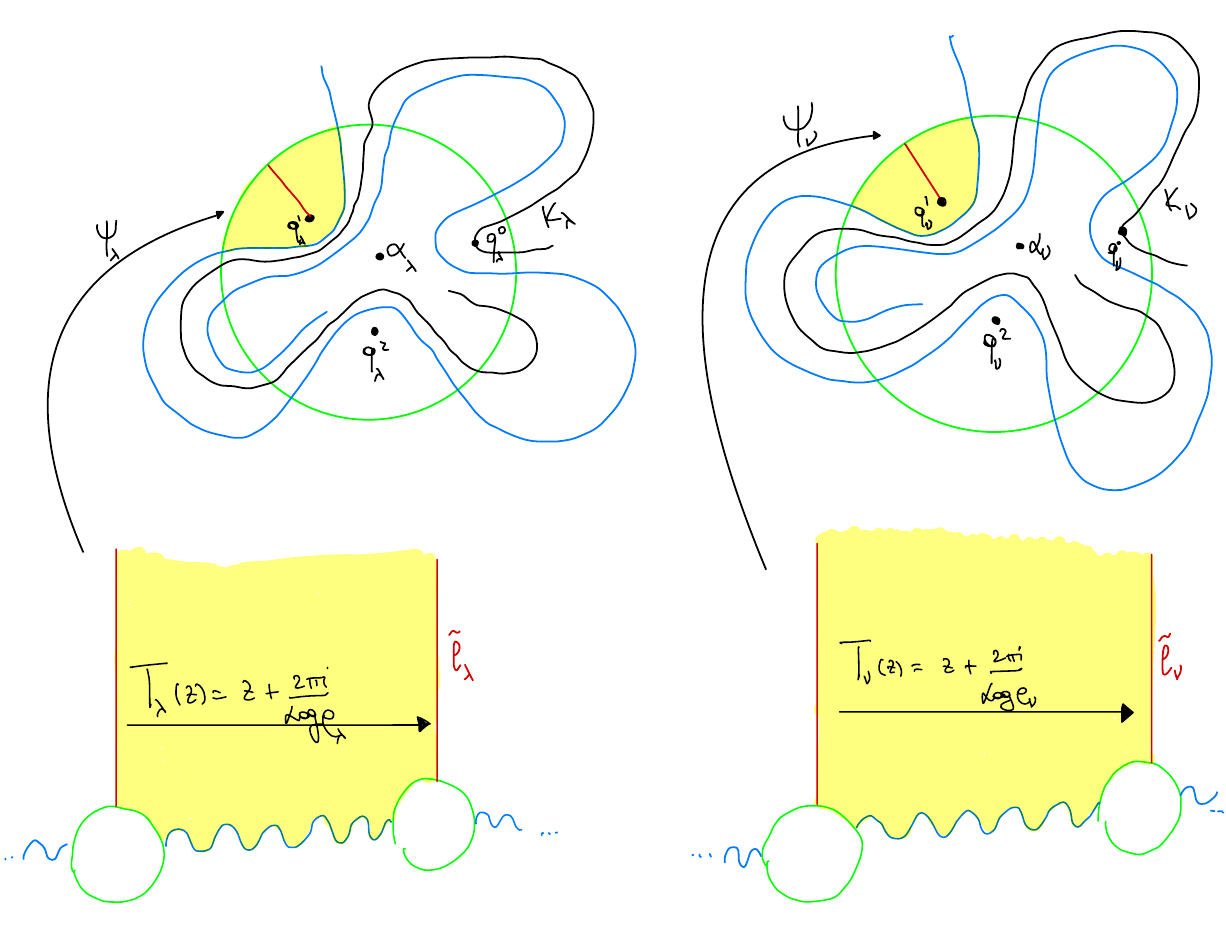}
    \caption{The domains $A^j_\iln$ (above) and $\wtA^j_\iln$ (below).}
    \label{fig:bf}
\end{figure}

For $\lambda\in\Mpqr(\whR)$ and $\nu\in\Mppqr(\whR)$ with $\xi(\lambda) = \nu$ 
we want to construct a quasiconformal homeomorphism between 
substantial parts of $A^j\la$ and $A^j\nn$ respecting potential on the ray part of the boundaries 
and with dilatation uniformly bounded.
We will do this by constructing uniformly quasiconfomal homeomorphisms between 
the domains $\wtA^j\la$ and $\wtA^j\nn$ minus corners pieces, these corner pieces having images under $\psi^j\la$ and $\psi^j\nn$ respectively uniformly close to the boundary circle $S_{r_1}$. For this we will need some technical lemmas.

\begin{lemma}\label{fundamental_segment_bound}
 There exists $\check C >0$ such that, for every $\iln = \lambda, \nu$, 
  $\lambda\in\Mpqr(\whR)$, and $\nu\in\Mppqr(\whR)$ 
 and for every $0<j<q$:
 \begin{enumerate}
 \item
 any fundamental segment of $\wtR^j_\iln$, 
 i.e.~any piece of $\wtR^j_\iln$ between potentials $E$ and $2^qE$, has diameter at most $\check C$;
 \item
$\wtR^j_\iln$ is contained in a horizontal strip of height at most  $\check C + K$.
\end{enumerate}
\end{lemma}
\begin{proof}
The curve $\Phi^j\la(\wtR^j\la)$ is invariant under 
the translation $T_1$ and the initial fundamental segments converge uniformly 
to (translations of) those of the corresponding $T_1$ invariant curve 
$\Phi^j_{\lambda_0}\circ\varphi_{\lambda_0}(R_{\lambda_0})$ for $f_{\lambda_0}$. 
Hence such a bound is true in Douady-Fatou coordinates. 
By Proposition \ref{comparison} it also holds for the curves $\wtR^j\la$ 
at the expense of a possibly larger but uniform bound. The proof for $\nu$ is similar.
\end{proof}

From here and onwards we are going to fix the value of $r_1$, and so we will omit the argument/index $r_1$ of $x^{j+}_\iln(r_1),y^j_\iln(r_1), x^{j-}_\iln(r_1), w^j_\iln(r_1)$ and of $S_{r_1}$, as well as their preimages by a choice of $\psi^j_\iln$, 
$\iln = \lambda, \nu$. 
We define 
$$
m_0 := \frac{2}{qr_1^q}. 
$$
For $\lambda\in\Mpqr(\whR)$, $\nu\in\Mppqr(\whR)$ and $\iln = \lambda, \nu$, 
taking $\varphi_\iln(y^j_\iln) = \wty^j_\iln$ as above 
we define $\wtx^{j+}_\iln := \varphi_\iln(x^{j+}_\iln) \in\wtbe^{j+}_\iln$
to be the leftmost point of $\Upsilon^{j+}_\iln(r_1)$, and 
$\wtx^{j-}_\iln$ as the rightmost point of $\Upsilon^{j-}_\iln(r_1)$. 
Then $\wtx^{j-}_\iln := T_{-\frac{i2\pi}{\Log\rho_\iln}}\circ\varphi_\iln(x^{j-}_\iln) \in\wtbe^{j-}\la$, 
and the horizontal width of $\Omega_\iln'^j(r_1)$ and of $\Omega_\iln^j(r_1)$ is the positive number
$$
\Re(\wtx^{j+}_\iln- \wtx^{j-}_\iln)
\Re\left(\frac{i2\pi}{\Log(\rho_\iln)} - \int_{[x^{j+}_\iln, x^{j-}_\iln]_S} \omega_\iln(z) dz\right).
$$
Here the path of integration $[x^{j+}_\iln, x^{j-}_\iln]_S$ for the integral
is the subarc of $S$ from $x^{j+}_\iln$ via $y^j_\iln$ to $x^{j-}_\iln$.
It follows from \eqref{coordinate_estimates} 
that 
$$
\frac{9}{10}m_0 \leq
\Re\left(\int_{[x^{j+}_\iln, x^{j-}_\iln]_S} \omega_\iln(z) dz\right) 
\leq \frac{11}{10}m_0
$$

Hence we have the following bounds on the widths
$$
D_\iln-\frac{11}{10}m_0 \leq \Re(\wtx^{j+}_\iln - \wtx^{j-}_\iln)  \leq D_\iln - \frac{9}{10}m_0
$$
where 
$$
D_\iln = \left|\Re\left(\frac{i2\pi}{\Log(\rho_\iln)}\right)\right| 
$$
Then, since by Lemma \ref{rhos} $|D\la - D\nn| < \hat C$ when $\nu = \xi(\lambda)$,
we have the following common bound on the widths:
\begin{lemma}\label{2.8}
Possibly reducing $\whR>0$ we can assume that, 
for all $\lambda\in\Mpqr(\whR)$ and $\nu = \xi(\lambda)\in\Mppqr(\whR)$, the horizontal widths of $\Omega^j\la(r_1)$ and $\Omega^j\nn(r_1)$  satisfy:
$$
D\la-\frac{11}{10}m_0 - \hat C \leq \Re(\wtx^{j+}\la - \wtx^{j-}\la),  \Re(\wtx^{j+}\nn - \wtx^{j-}\nn) 
\leq D\la-\frac{9}{10}m_0  + \hat C.
$$
\end{lemma}

In the following we shall use that $f\la$ and 
$g\nn$ multiply the Green's potential by $2^q$, 
that $\psi_\iln$ conjugates $F_\iln$ to $f_\iln$ and respectively $G_\iln$ to $g_\iln$, that 
$$
|F\la(Z)-Z-1| = |U\la(Z)| \leq\frac{1}{4}, 
\qquad
|G\nn(Z)-Z-1| = |U\nn(Z)| \leq\frac{1}{4}
$$
that 
$$
\Phi\la\circ F\la = T_1\circ \Phi\la,
\qquad
\Phi\nn\circ G\nn = T_1\circ \Phi\nn
$$
and that Proposition~\ref{comparison} and the above implies that, whenever 
$Z, F\la(Z), \ldots , F\la^n(Z)\in\Omega^j\la(r_1)$, 
then 
\begin{equation}
  |F\la^n(Z) - Z - n| \leq \min\{K, \frac{n}{4}\}.\label{Goodboundonlongiterates}  
\end{equation}

For each $0\leq j < q$ and $\ilnz = \lambda_0, \nu_0$ 
let $E^j_\ilnz$ be the maximum potential 
such that $R^j_\ilnz[0,E^j_\ilnz]\subset\overline{\D}(0, r_1)$. 
Moreover for $0\leq j < q$, $\ilnz = \lambda_0, \nu_0$ and $\iln = \lambda, \nu$ 
let $z^j_{\iln}\in R^j_\iln$ be the unique point at potential $\whE^j_\ilnz := E^j_\ilnz/2^{2q}$. 
Then $z^j_{\iln}$ converges to $z^j_{\ilnz}$ when $\iln\to\ilnz$.
Let $\whE_0$ be the minimum over the $\whE^j_\ilnz$ and 
$\whE^l$ the maximum value and define 
$$
m = \frac{\log(\whE^l - \whE_0)}{q\log 2} + 2.
$$ 
Possibly reducing $\whR>0$ we can suppose that 
for all $0< j < q$ and $\iln = \lambda, \nu$
$$
R^j_\iln([\whE_0/2^q, \whE_0]) \subset \D(0, r_1)
\qquad\textrm{and}\qquad
|z^j_\iln-z^j_\ilnz|<1.
$$

To fix the ideas let $\varphi^j_\iln$ be normalized so that $\wtx^{j+}_\iln = \varphi^j_\iln(x^{j+}_\iln) = 0$ 
for $\iln = \lambda, \nu$, where $\lambda\in\Mpqr(\whR)$, $\nu\in\Mppqr(\whR)$, 
and similarly 
$\wtx^{j\pm}_\ilnz = \varphi^{j\pm}_\ilnz(x^{j\pm}_\ilnz) = 0$ 
for $\ilnz = \lambda_0, \nu_0$.
Define $\wtz^j_{\iln} = \varphi^j_\iln(z^j_{\iln})$ and $\wtz^j_{\ilnz} = \varphi^{j+}_\ilnz(z^j_{\iln})$. 
Then $|\Arg(z^j_\ilnz/x^{j+}_\ilnz)| \leq \pi/6$ and thus 
$$
-5/2 = -\frac{2\cdot 5}{4} \leq 
\Re(\wtz^j_{\ilnz}-\wtx^{j+}_\ilnz) 
\leq -\frac{2\cdot 3}{4} + \left(1-\frac{1}{2}\sqrt{3})\right)\frac{m_0}{2} < 
\frac{m_0}{5} -\frac{3}{2}.
$$
and hence, since $\wtx^{j+}_\iln = \wtx^{j+}_\ilnz = 0$ and $|z^j_\iln-z^j_\ilnz|<1$
$$
-7/2 \leq 
\Re(\wtz^j_{\iln}-\wtx^{j+}_\iln) < \frac{m_0}{5}.
$$

From here we shall work with pairs $(\lambda, \nu)$, where 
$\lambda\in\Mpqr(\whR)$ and $\nu = \xi(\lambda)\in\Mppqr(\whR)$. 
In particular we shall use $\lambda$ as an index and assume $\nu = \xi(\lambda)$.
Define a constant
$$
m_1 = 2(m_0 + K + \check C + \hat C + 4 + m)
$$ 
and a function of $\lambda$
$$
m_{2,\lambda} = D\la - m_0 - m_1, 
$$
where $K = K(\widehat R, r_0, r_1)$ from Proposition~\ref{comparison}, 
$\hat C$ is from Lemma~\ref{rhos} and 
$\check C$ is from Lemma~\ref{fundamental_segment_bound}. 
For $\lambda\in\Mpqr(\whR)$, $\nu = \xi(\lambda)\in\Mppqr(\whR)$ and $\iln = \lambda, \nu$ 
let $z^j_{1,\iln}\in R^j_\iln$ be the unique point at potential 
$\check E_1 =\frac{\whE_0}{2^{qm_1}}$ and 
let $z^j_{2,\iln}\in R^j_\iln$ be the unique point at potential 
$\check E_{2,\lambda} =\frac{\whE_0}{2^{qm_{2,\lambda}}}$. 
Moreover let $\wtz^j_{1,\iln},\wtz^j_{2,\iln}\in \wtR^j_\iln$ 
be the corresponding images under analytic continuation of $\varphi^j_\iln$ along rays. 

\begin{lemma}\label{estimates}
For all $\lambda\in\Mpqr(\whR)$, $\nu = \xi(\lambda)\in\Mppqr(\whR)$ and $\iln = \lambda, \nu$
$$
\wtR^j_\iln([\check E^j_{2,\lambda}, \check E_1]) \subset \Omega^j_\iln(r_1).
$$
Moreover 
$$
\frac{m_1}{2} \leq \Re(\wtx^{j+}_\iln-\wtz^j_{1,\iln})
\leq \frac{3}{2} m_1.
$$
and
$$
\frac{m_1}{2} \leq \Re(\wtz^j_{2,\iln} - \wtx^{j-}_\iln)
\leq \frac{3}{2} m_1.
$$
\end{lemma}
\begin{proof}
We have from the above that 
$$
-\frac{m_0}{5} < \Re(\wtx^{j+}_\iln - \wtz^j_{\iln}) < 7/2
$$
And by \eqref{Goodboundonlongiterates} 
$$m_1 - K - \check C \leq \lfloor m_1\rfloor + 1 - K - \check C 
\leq\Re(\wtz^j_{\iln}-\wtz^j_{1,\iln}) 
\leq \lfloor m_1+m\rfloor  + K + \check C \leq m_1 + K + \check C + m
$$
and hence  
$$
\frac{m_1}{2} \leq m_1 - K - \check C - \frac{m_0}{2} \leq 
\Re(\wtx^{j+}_\iln-\wtz^j_{1,\iln}) 
\leq m_1 + K + \check C +7/2 + m
\leq \frac{3}{2} m_1.
$$
Similarly
$$
m_{2,\lambda} - K - \check C - \frac{m_0}{2} \leq\Re(\wtx^{j+}_\iln - \wtz^j_{2,\iln})  
\leq m_{2, \lambda}  + K + \check C +7/2 + m.
$$
and since $m_{2,\lambda} = D\la - m_0 - m_1$ we obtain
\begin{align*}
\frac{m_1}{2} &\leq
D\la - \frac{11}{10}m_0 - \hat C  - (D\la - m_0 - m_1 + K + \check C + 7/2 +m)\\
&\leq Re(\wtx^{j+}_\iln - \wtx^{j-}_\iln) -  \Re(\wtx^{j+}_\iln-\wtz^j_{2,\iln}) 
= \Re(\wtz^j_{2,\iln} - \wtx^{j-}_\iln)\\ 
&\leq D\la + \hat C  - \frac{9}{10}m_0 - (D\la - m_0 - m_1 - K - \check C - \frac{m_0}{2})
\leq \frac{3}{2} m_1.
\end{align*}
\end{proof}
\begin{figure}[tbh]
    \centering
    \includegraphics[scale = 0.5]{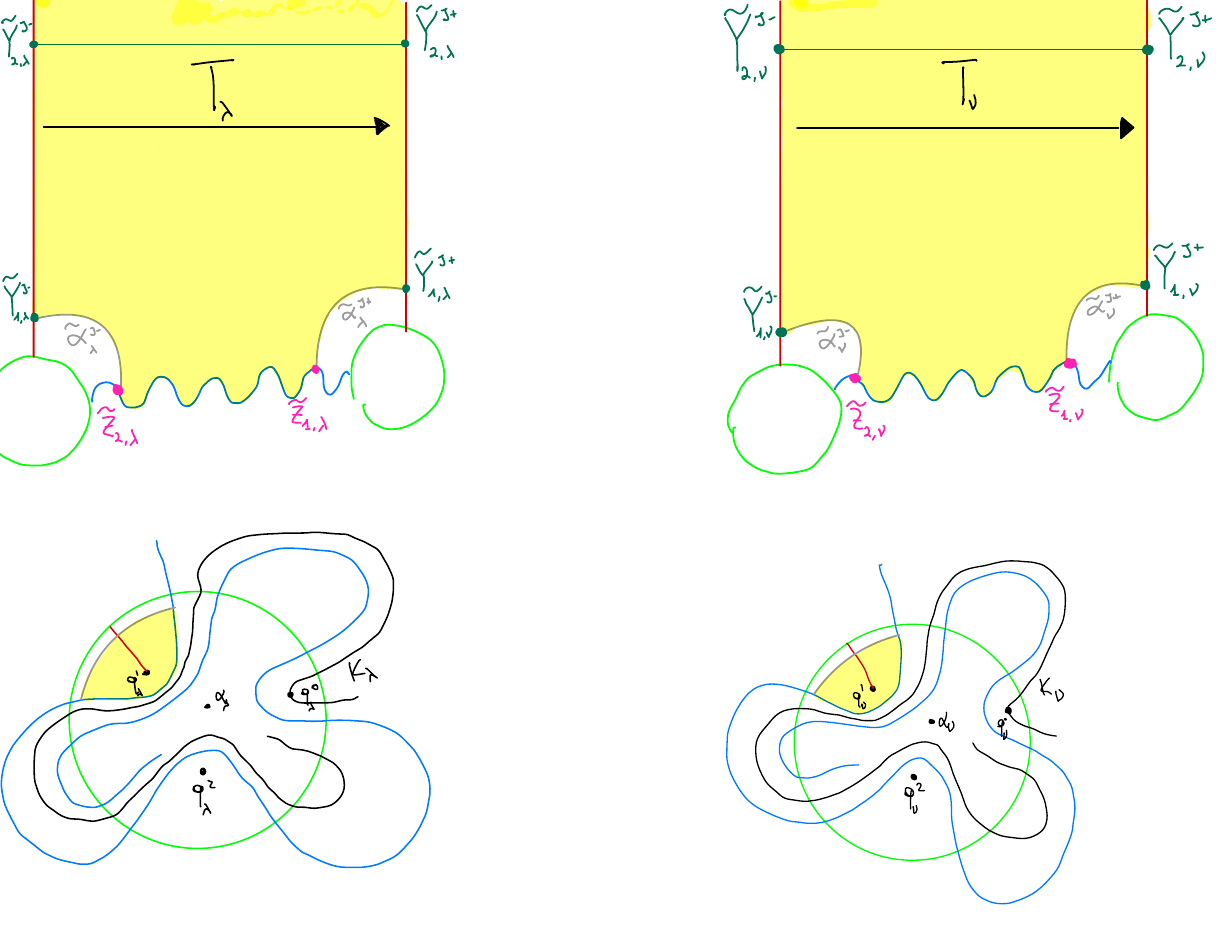}
    \caption{The domains $\wtA^{j,up}_\iln$ (above the straight lines connecting $\wtY^{j-}_{2,\iln}$ and $\wtY^{j+}_{2,\iln}$) and $\wtA^{j,down}_\iln$ (below the straight lines).}
    \label{fig:q}
\end{figure}
For $\lambda\in\Mpqr(\whR)$, $\nu = \xi(\lambda)\in\Mppqr(\whR)$ and $\iln = \lambda, \nu$
define
\begin{alignat*}{4}
&\wtY^{j+}_{1,\iln} &:= \wty^j_\iln+im_1,  \qquad
&\wtY^{j-}_{1,\iln} &:= T_{\frac{-i2\pi}{\Log(\rho_\iln)}}(\wtY^{j+}_{1,_\iln}),\\
&\wtY^{j+}_{2,\iln} &:= \wty^j_\iln+iD\la, \qquad
&\wtY^{j-}_{2,\iln} &:= T_{\frac{-i2\pi}{\Log(\rho_\iln)}}(\wtY^{j+}_{2,_\iln}),\\
&\wtX^{j+}_\iln &:= \wtY^{j+}_{1,_\iln}  - m_1,  \qquad
&\wtX^{j-}_\iln &:= \wtY^{j-}_{1,_\iln} + m_1.
\end{alignat*}

Let $\wtal^{j+}_\iln$ denote the hyperbolic geodesic in $\wtA^j_\iln$ 
connecting $z^j_{1, \iln}$ to $\wtY^{j+}_{1, \iln}$ and let 
$\wtal^{j-}_\iln$ denote the hyperbolic geodesic in $\wtA^j_\iln$ 
connecting $z^j_{2, \iln}$ to $\wtY^{j-}_{1, \iln}$ (see Figure \ref{fig:q}). 
Then these arcs have finite Euclidean lengths, 
because the boundary arcs of $\wtA^j_\iln$ at both end points are real analytic. 
We may henceforth parametrize these arcs by Euclidean arc-length 
increasing from zero at $z^j_{1, \iln}$ and $z^j_{2, \iln}$ respectively. 
Call $\wtA^{j,up}_\iln \subset \wtA^j_\iln$ the infinite vertical sub-strip bounded from below 
by the line joining $\wtY^{j-}_{2,\iln}$ and $\wtY^{j+}_{2,\iln}$ (see Figure \ref{fig:q}), and 
call $\wtA^{j,down}_\iln \subset \wtA^j_\iln$ the quadrilateral bounded above 
by the line joining $\wtY^{j-}_{2\iln}$ and $\wtY^{j+}_{2,\iln}$ and restricted further 
by the arcs $\wtal^{j-}_\iln$ and $\wtal^{j+}_\iln$. 
Then clearly the boundary homeomorphism, 
which maps the lower boundary line of $\wtA^{j,up}\la$ affinely 
to the lower boundary line of $\wtA^{j,up}\nn$ and maps the vertical boundary lines 
to the vertical boundary lines by translations has a quasiconformal extensions
to the interiors, which is uniformly q.c.- in fact with complex dilatation tending to zero:
indeed such extension affinely maps each line parallel to the bottom line to the corresponding 
line parallel to the bottom line. 

The pair of subdomains $\wtA^{j,down}\la$ $\wtA^{j,down}\nn$, 
are two topological squares each with two corners torn off. 
Define a piecewise real analytic boundary map as follows. 
Between the top lines, which are the bottom lines of $\wtA^{j,up}\la$ $\wtA^{j,up}\nn$, 
it is already defined to be affine. 
On the vertical sides it consists of translations. 
On the bottom lifted rays it preserves potential. 
Finally  the corner arcs $\wtal^{j\pm}_\iln$
have finite Euclidean lengths, 
since they meet real-analytic boundary arcs of $\wtA^j_\iln$ orthogonally. 
Parametrize these by Euclidean arc-length with $\wtal^{j+}_\iln(0) = \wtz^j_{1,\iln}$ 
and $\wtal^{j-}_\iln(0) = \wtz^j_{2,\iln}$.
And we define the map between the $\alpha$-arcs so that it is affine in the parameters.

\begin{prop}\label{buffext}
The map between the boundaries of $\wtA^{j,down}\la$ and $\wtA^{j,down}\nn$, with $\lambda\in\Mpqr(\whR)$ and $\nu = \xi(\lambda)\in\Mppqr(\whR)$,
is uniformly quasisymmetric, and hence it extends to a $K$-q.c. homeomorphism 
with $K$ uniformly bounded.
\end{prop} 
\begin{proof}
Note at first that the top boundary lines meet the vertical sides at nearly right angles and 
that all the other real-analytic pieces of the boundary meet at right angles. 
Secondly the coordinates $\varphi^j_\iln$ depends holomorphically on $\iln$. 
Let $\wtA^{j+}_{\ilnz}$ denote the upper connected component of 
$\Omega_\ilnz^{'j+}\setminus\wtR^j_\ilnz$. 
And let $\wtz_{1,\ilnz}\in\wtR^j_\ilnz$, $\wtY_{1,\ilnz}^j$ and $\wtal_{1,\ilnz}^{j+}$ 
be defined similarly to $\wtz_{1,\iln}\in\wtR^j_\iln$,
$\wtY_{1,\iln}^j$ and $\wtal_{1,\iln}^{j+}$ above.
Then the pointed domains $(\wtA^j_\iln, \wtX_\iln^{j+})$ 
converge Caratheodory to the pointed domains 
$(\wtA^{j+}_\ilnz, \wtX_\ilnz^{j+})$
and the arcs $\Upsilon^{j+}_\iln$, $\wtR^j_\iln$, 
converge uniformly in parameter 
to $\Upsilon^{j+}_\ilnz$, $\wtR^j_\ilnz$. 
Hence also $\wtal^{j+}_\iln$ 
converge uniformly in parameter to $\wtal^{j+}_\ilnz$. 
This shows that the boundary maps from the corner of $\partial\wtA^{down}\la$ around $\wtal\la^{j+}$ 
to the corner of $\partial\wtA^{down}\nn$ around $\wtal\nn^{j+}$
remains in a compact family of quasi-symmetric maps. \\

On the other hand, recall that, by the second item in Lemma \ref{estimates}, for every $k$, 
$$
\frac{m_1}{2} \leq \Re(\wtz^j_{2,\iln} - \wtx^{j-}_\iln)
\leq \frac{3}{2} m_1.
$$
Change normalization for $\varphi^j_\lambda$ so that $\tilde x^{j-}_\lambda=\varphi^j_\lambda(x^{j-}_\lambda)=0$, therefore 
$$
\frac{m_1}{2} \leq \Re(\wtz^j_{2,\iln})
\leq \frac{3}{2} m_1.
$$
Hence, there exists a $\tilde t>0$ such that  $\wtz^j_{2,\iln} \in R^j_k(\tilde t,\infty)$ for all $k$.
By the parabolic implosion theory (see the discussion at the beginning of this section), when $k_n \rightarrow \infty$, $\rho_n \rightarrow 1$, and $k_n - \frac{2\pi i }{\rho_n} \rightarrow \delta$, then $f_{\lambda_n}^{k_n} \rightarrow L_\delta$, and 
 the ray $R^j_{\lambda_n}(\frac{t}{2^{qk_n}})$ converges uniformly to $R^{j,1}_{\delta}(t)$ on compact subsets of $(\tilde t,\infty)$. In particular, there exists $z_\delta = \lim_{n \rightarrow \infty} \tilde z_{2,\lambda_n}$, with  $z_\delta \in R^{j,1}_{\delta}$.
Let $\wtA^{j-}_{\ilnz,\delta}$ denote the upper connected component of 
$\Omega_\ilnz^{'j-}\setminus\wtR^{j,1}_{\ilnz,\delta}$.
Define $\wtX_{\lambda_0}^{j-}:=\wty^j_{\lambda_0}+im_1+m_1$.
Then, also in this case, the pointed domain 
$(\wtA^j_{\lambda_n}, \wtX_{\lambda_n}^{j-})$ 
converge Caratheodory to the pointed domains 
$(\wtA^{j-}_{\lambda_0,\delta}, \wtX_{\lambda_0}^{j-})$.  
As also $\Upsilon^{j-}_{\lambda_n}$, 
converges uniformly in parameter 
to $\Upsilon^{j-}_{\lambda_0}$, we obtain that 
$\wtal^{j-}_{\lambda_n}$ 
converge uniformly in parameter to $\wtal^{j-}_{\lambda_0}$. Repeating the argument for $\nu$, we see that
the boundary maps from the corner of $\partial\wtA^{down}\la$ around $\wtal\la^{j-}$ 
to the corner of $\partial\wtA^{down}\nn$ around $\wtal\nn^{j-}$
remains in a compact family of quasi-symmetric maps.

Finally the domains $\wtA^{down}\la$ and $\wtA^{down}\nn$ are nearly square 
and the other parts of the boundary map are easily seen to remain in 
a compact family of quasi-symmetric maps. 
\end{proof}
 
 Define $\hat A^j\la:= \overline{\psi\la^j(\wtA^{j,down}\la \cup \wtA^{j,up}\la)}$, and similarly 
$\hat A^j\nn:= \overline{\psi\nn^j(\wtA^{j,down}\nn \cup \wtA^{j,up}\nn)}$.
Proposition \ref{buffext} together with the discussion before it gives directly the following:
\begin{cor}\label{1}
Let $\lambda\in\Mpqr(\whR)$ and $\nu = \xi(\lambda)\in\Mppqr(\whR)$.
There exists $K>1$ and a family of $K$-quasiconformal homeomorphisms 
$\widetilde \Phi^j\la:\hat A^j\la \rightarrow \hat A^j\nn$, which preserves rays potentials on the part of $\partial A^j\la$ given by external ray.
\end{cor}

\subsection{Pre-exterior uniformly qc equivalences}\label{Extension}
In this section we construct a family of pre-exterior uniformly quasiconformal equivalences $\varphi\la$ between corresponding polynomial-like maps $\{f\la\}_{\lambda\in\Mpqr(\whR)}$ and $\{g\nn\}_{\nu = \xi(\lambda)\in\Mppqr(\whR)}$.

Fix a potential $\mathcal{N}_0$ big, let $\D^{\mathcal{N}_0}\la$ the topological disc bounded by the equipotential $E^{\mathcal{N}_0}\la$ of potential $\mathcal{N}_0$, and note
that $R\la \cup \{\alpha\la\}\cup R'\la$ divides $\D^{\mathcal{N}_0}\la$ in two connected components: call $O\la$ the one not containing the critical value (equivalently not containing the filled Julia set $K\la$ of the polynomial-like restriction $f\la$ of $P\la$). Define similarly
 the topological disc $\D^{\mathcal{N}_0}\nn$ bounded by the equipotential $E^{\mathcal{N}_0}\nn$ of potential $\mathcal{N}_0$, and the set $O\nn$.
 We start proving the following:

\begin{lemma}\label{O}
There exists a family of uniformly quasiconformal maps $F\la: O\la \rightarrow O\nn$.
\end{lemma}
\begin{proof}
In Section \ref{cycle} (Corollary \ref{1}), we constructed families 
of uniformly quasiconformal maps 
$\widetilde \Phi^j\la:\hat A^{j}\la \rightarrow \hat A^{j}\nn$,   $j\in [1,q-1]$, 
between neighbourhoods $\hat A^j\la$ and $\hat A\nn^j$ of the $j$-th points of the q cycle (where $q_0$ is the $\alpha$-fixed point of the polynomial-like map).
In section \ref{between rays} (see Theorem \ref{intrays}) we constructed a family $\widetilde\Phi\la: \check W\la \rightarrow \check W\nn$ of uniformly quasiconformal interpolations between the external rays $R\la$ and $R'\la$ (considered in anticlockwise order, with $\check W\la=\eta\la(\hat W\la)$ and $\check W\nn=\eta\nn(\hat W\nn)$), in a domain of univalency of the linearizers maps. We can apply this construction to obtain families $\widetilde \Phi^{i,j}\la: \check W^{i,j}\la \rightarrow \check W^{i,j}\nn$ of uniformly quasiconformal interpolations between consecutive rays $R^i\la$ and $R^j\la$, in a domain of univalency of the linearizers maps (see Corollary \ref{gen}).
So, to prove the statement, we just need 
a uniform quasiconformal interpolation on 
$\hat O\la:=  O\la \setminus (\bigcup_{j \in [0,q-2]} \hat A^j\la\cup \check W\la \bigcup_{i,j \in [0,q-2],i\neq j} \check W^{i,j}\la)$.

Call $L\la$ the boundary of $\check W\la$ different from $R\la$ and $R'\la$, $L^{i,j}\la$ the boundary of $\check W^{i,j}\la$ different from the consecutive rays, and define similarly $L\nn$ and $L^{i,j}\nn$.
Using dynamics we can extend 
$\widetilde \Phi\la$ and $\widetilde \Phi^{i,j}\la$ to bigger domains, having boundary $L\la^n=f\la^n(L\la)$ and $L\la^{n,i,j}=f\la^n(L^{i,j}\la)$ respectively, with length bigger than some chosen constant. 
So choose $r>0$, and let $n=n(\lambda)$ be such that the length of $L\la^n$, $L\nn^n$, $L\la^{n,i,j}$ and $L\nn^{n,i,j}$ is at least $r$, the distance between the equipotential $E^{\mathcal{N}_0}\la$ of potential $\mathcal{N}_0$ and $L\la^n$ is at least $r$, as well as the distances between $E^{\mathcal{N}_0}\la$ and the  $L\la^{n,i,j}$, $E^{\mathcal{N}_0}\la$ and the $\hat  A^{ j}\la,\,\,j\in (1,q)$, between the equipotential $E^{\mathcal{N}_0}\nn$ of potential $\mathcal{N}_0$ and $L\nn^n$,  $E^{\mathcal{N}_0}\nn$ and the $L\nn^{n,i,j}$, and between $E^{\mathcal{N}_0}\nn$ and the $\hat  A^{j}\nn,\,\,j\in (1,q)$.
As $\widetilde\Phi\la: \check W\la \rightarrow \check W\nn$ is uniformly quasiconformal, and $f\la$ is holomorphic, we can lift $\widetilde\Phi\la$ to a uniformly quasiconformal map $\widetilde\Phi\la: f\la^n(\check W\la) \rightarrow g\nn^n(\check W\nn)$ and $\widetilde \Phi^{i,j}\la$ to $\widetilde \Phi^{i,j}\la: f\la^n(\check W^{i,j}\la) \rightarrow g\nn^n(\check W^{i,j}\nn)$.
As all the angles are uniformly bounded away from $0$ and $\pi$, the result follows.
\end{proof}
The previous extension gives this obvious
\begin{cor}\label{g}
There exists a family of uniformly quasiconformal maps $G\la: \check O\la:= O\la \setminus \check W\la\rightarrow \check O\nn:=O\nn\setminus \check W\nn$.
\end{cor}

For $\lambda \in \Mpqr$, let $x\la \in R\la \cap \overline{\check W\la}$ and $x'\la \in R'\la \cap \overline{\check W\la}$ be such that their images $\hat x\la$ and $\hat x'\la$ in $\hat W\la$ have the same real part. Call $\ell\la \in \overline{\hat W\la}$ the vertical line joining them, and $\check\ell\la=\eta\la(\ell\la)$.
Call $E_|^R$ the shortest curve in $E^{\mathcal{N}_0}\la$ between $R\la$ and the preimage  of $R\la'$  landing at the preimage  of $\alpha\la$.
Define $\ell\la^R$  to be a smooth convex curve defined in Log-Bottcher coordinates joining the image in Log-Bottcher coordinates of $x\la$  with the image in Log-Bottcher coordinates of a point in the interior of $E_|$, and such that $\ell\la^R \cap R\la=x\la$. 
Call $E_|^L$ the shortest curve in $E{\mathcal{N}_0}p\la$ between $R\la'$ and the preimage  of $R\la$  landing at the preimage  of $\alpha\la$.
Define $\ell\la^L$  to be a smooth convex curve defined in Log-Bottcher coordinates joining the image in Log-Bottcher coordinates of $x\la'$  with the image in Log-Bottcher coordinates of a point in the interior of $E_|^L$, and such that $\ell\la^L \cap R\la'=x\la'$.

Call $E^*\la$ the shortest curve in $E^{\mathcal{N}_0}\la$ bounded by $\ell\la^R$ and $\ell\la^L$. Define $\partial U\la$ as the piecewise smooth jordan curve defined by $\check\ell\la \cup \ell\la^R \cup E^*\la \cup \ell\la^L$, and $U\la$ as the bounded Jordan domain bounded by $\partial U\la$. Define $U\la':= P\la^{-q}$, then $(U\la,U'\la, f\la=P^q|_{\lambda|U'\la}) $ is the polynomial-like restriction we consider. 
Repeat the construction for $\nu=\xi(\lambda)$. In particular, note that,
 if $R_1'$ and $R_1$ are the images in B\"ottcher coordinates of $R\la'$ and $R\la$ respectively, and
$R'_2$ and $R_2$ are the images in B\"ottcher coordinates of $R\nn'$ and $R\nn$ respectively, then there exists and angle $\widehat \theta$ such that the rotation $R_{\widehat\theta}$ of angle $\widehat \theta$ maps $R_1$ to $R_2$ and $R_1'$ to $R_2'$.
Set $\check\ell\nn= \widetilde\Phi\la(\check\ell\la)$, $\ell\nn^R= \Log\circ B\nn^{-1}\circ R_{\hat\theta}\circ B\la\circ \exp(\ell\la^R)$, $\ell\nn^L= \Log\circ B\nn^{-1}\circ R_{\hat\theta}\circ B\la\circ \exp(\ell\la^L)$, and call $E^*\nn$ the shortest curve in the equipotential $E^{\mathcal{N}_0}\nn$ of potential $\mathcal{N}_0$ bounded by $\ell\nn^R$ and $\ell\nn^L$.
 Call $\partial U\nn$ the piecewise smooth Jordan curve defined by $\check\ell\nn \cup \ell\nn^R \cup E^{\mathcal{N}_0}\nn \cup \ell\nn^L $, and $U\nn$ the bounded Jordan domain bounded by $\partial U\nn$. Define $U\nn':= P\nn^{-q}$, then $(U\nn,U'\nn, g\nn=P^q_{|\nu|U'\nn}) $ is the polynomial-like restriction we consider. 

 Define $W\la:= \widehat \C \setminus U'\la$ and $W\nn:= \widehat \C \setminus U'\nn$.
 \begin{teor}\label{Uniformly_qc_conjugacies}
There exists a family of pre-exterior uniformly quasiconformal equivalences $\varphi\la: W\la \rightarrow W\nn$.
 \end{teor}
\begin{proof}
Let $S\la$ be the fundamental domain in $\check W\la$ between $\check \ell\la$ and its preimage in $\check W\la$. Define $S\nn=\widetilde\Phi\la(S\la)$, then obviously by Theorem \ref{intrays} there exists a family of uniformly quasiconformal maps 
$\widetilde\Phi_{\lambda|}:S\la \rightarrow S\nn$.
By Corollary \ref{g}, we gave a family of uniformly quasiconformal maps $G\la: O\la \setminus \check{ W\la}\rightarrow O\nn\setminus \check W\nn$ agreeing with the family $\widetilde\Phi_{\lambda|}$ on $\check \ell\la$.
Let $\check O\la'$ to be the preimage of 
$P\la^{-q}(\check O\la)$ in $U\la \setminus U\la'$, and let $\check O\nn'$ to be the preimage of 
$P\nn^{-q}(\check O\nn)$ in $U\nn \setminus U\nn'$. Then clearly we can lift the family of uniformly quasiconformal maps $G\la$ to a family of uniformly quasiconformal maps $\check G\la:\check O'\la  \rightarrow \check O'\nn$.
On $\hat O\la:=W\la \setminus (O\la \cup \check O'\la)$ we define $\varphi_{\lambda|\hat O\la}:= B\nn^{-1}\circ R_{\hat\theta}\circ B\la$. The result follows.
\end{proof}
\begin{remark}
Note that we could restrict the family of pre-exterior uniformly quasiconformal equivalences to a family of uniformly quasiconformal homeomorphisms between the fundamental domains of the corresponding polynomial-like maps
$\varphi_{\lambda|}: U\la \setminus U\la' \rightarrow U\nn \setminus U\nn'$.
\end{remark}

\section{Harvesting in parameter space}\label{L}
In this section we shall trade Theorem~\ref{Uniformly_qc_conjugacies} in for a proof of our main 
theorem that $\Mpq$ and $\Mppq$ are q-c homeomorphic. 
We shall use Lyubich notion and results on quadratic-like maps and quadratic-like germs (see \cite{Ly}) 
and complement it with some further definitions and results. 
We start with a brief presentation of the relevant definitions and results from \cite{Ly}. 
In line with \cite{Ly} we shall henceforth write 
$f:U\to U'$ for a polynomial-like map (note that the prime is on the co-domain instead of the domain). 

A quadratic-like map is a polynomial-like map $f: V \to V'$ of degree $2$, where $V\subset\subset V'\subset\subset\C$ are smoothly bounded Jordan disks 
and $f$ is normalized so that 
$f(z) = c + z^2+ \OO(z^3)$. 
In particular the critical point is $0\in U$ 
with critical value $c = c_f$. 
The space of quadratic like maps is denoted $\QM$. 
Two quadratic-like maps {\mapfromto{f_i}{V_i}{V_i'}}, $i=1,2$ are \emph{restriction equivalent} if they have a common quadratic-like restriction and \emph{equivalent} if they are the ends of a finite chain of neighbourwise restriction equivalent quadratic-like maps. 
In particular $f_1=f_2$ on a neighbourhood of $0$ and have the same filled-in Julia set.
For maps with connected filled-in Julia set equivalence and restriction equivalence is the same. 
Equivalence is obviously an equivalence relation on $\QM$ and the equivalence classes are called \emph{quadratic-like germs}. The space of quadratic-like germs is denoted $\QG$.\\
Following McMullen, Lyubich equips $\QM$ with the Caratheodory convergence structure in which
a sequence $(f_n:V_n\to V')\subset\QM$ converge to $f:V\to V'\in\QM$ if and only if the pointed disks $(V_n,0)$ converge Caratheodory to $(V,0)$ and $f_n\to f$ locally uniformly on $V$. 
He then equips the space $\QG$ with a topology and a complex analytic structure compatible with the Caratheodory convergence structure by 
covering $\QM$ with affine Banach balls that inject into $\QG$: 
for $V\ni 0$ a Jordan disk, denote by $\BB_V$ the Banach space of holomorphic maps $g:V\to\C$, with $g'(0)= 0$ and with a continuous extension to $\overline{V}$, 
equipped with the $\sup$-norm on $\overline{V}$. 
For $f:V\to V'\in\QM$ and $\epsilon=\epsilon_f>0$ sufficiently small,
define the affine Banach ball 
$\BB_V(f,\epsilon):=\{g \in\BB_V\,|\,||f-g||_{\infty}<\epsilon \}$ as the affine ball of center $f$ and radius $\epsilon$. 
Then there exists $\epsilon = \epsilon_f >0$ such that any map $g$ in the \emph{Banach slice}, the affine 
co-dimension $1$ sub-Banach ball, 
$\whBB_V(f, \epsilon) := \{g\in\BB_V(f,\epsilon)| g''(0) = 2\}$ of center $f$ and radius $\epsilon$ has a quadratic-like restriction $g:U\to U'$, where $U\subseteq V$ and $g(z)=c+z^2+\OO(z^3)$. 
The Banach slice $\whBB_V(f, \epsilon)\subset \QM$ 
injects into $\QG$ and the family of such Banach slices $\whBB_V(f, \epsilon_f)$, $f\in\QM$ provides $\QG$ with the desired topology and complex analytic structure.

Any quadratic-like map has an associated internal class and an associated external class. While the internal class is basically given by the little filled Julia set of the quadratic-like map, the external class can be defined in different ways. For Douady and Hubbard, the external class of a quadratic-like map is the conjugacy class under real-analytic conjugacy of an expanding degree $d$ real analytic circle covering map encoding the conformal dynamics of the polynomial-like map outside its filled Julia set $K_f$. It turns out that the external equivalence classes as defined by Douady and Hubbard are too large for giving the space of polynomial-like maps a complex analytical structure. McMullen made the initial restriction of considering only quadratic-like maps with co-domain a relatively compact disk in  complex plane, and equivalence classes to mean affine conjugacy classes. Lyubich defined the external class of a quadratic-like map to be the unique expanding degree $2$ circle map fixing $1$ defined via conjugation by the appropriate Riemann map from $\Chat\setminus K_f$ to $\Chat\setminus\Dbar$ fixing $\infty$ in the case $K_f$ is connected and some appropriate substitute in the disconnected case. 
Then two quadratic-like maps $f_i$ with connected filled-in Julia sets $K_i$ have the same Lyubich  external class if and only if they are externally equivalent in the sense of Douady and Hubbard and any external equivalence is a restriction of the biholomorphic map $\phi = \phi_2^{-1}\circ\phi_1:\Chat\setminus K_1\to\Chat\setminus K_2$ fixing $\infty$. 
In the disconnected case the construction of external map and class is slightly more involved, just as it is in the case with Douady and Hubbard's notion of external class. 
We shall omit it for the time being, but return to it after \corref{from_pre_exterior_to_exterior}. 
The important property however is that the notion of having same external class defines an equivalence relation on $\QM$, which descends to $\QG$. And if two quadratic-like maps are both internally equivalent and Lyubich externally equivalent they define the same germ. 
Lyubich calls the equivalence classes under his  equivalence vertical fibers and denotes them $\ZZ_f$, where $f$ is some representative of the equivalence class or the common external map:
$$
\ZZ_{f_0} := \{ [f]\in \QG : \textrm{$f$ has the same external class as $f_0$} \}.
$$

The connectedness locus $\CC\subset\QG$ is the set of germs of quadratic-like maps $f$ 
with connected filled-in Julia set. 
Lyubich proves that each vertical fiber 
intersects $\CC$ in a compactly contained homeomorphic copy of the Mandelbrot set
\begin{cor}[\cite{Ly}, Cor 4.24]
    The vertical fibers are full unfolded quadratic-like families.
\end{cor}

Douady and Hubbard \cite{DH} define 
two polynomial-like maps  $f_i:V_i\to V_i'\in\QM$, 
$i=1,2$ to be hybridly equivalent, if and only if there exists a q.c.-homeomorphism $\phi$ between neighbourhoods of the respective filled Julia sets $K_1$ and $K_2$ which conjugates dynamics and for which the complex dilatation $k_\phi(z) = 0$ a.e.~on $K_1$. 
Two germs in $\QG$ are hybridly equivalent, if they can be represented by hybridly equivalent quadratic-like maps.

The hybrid class $\HH_0\subset\QG$ 
consisting of germs of maps $f$ in $\QM$ with $0$ constant term, i.e.~$f(z) = z^2+\OO(z^3)$, 
is a co-dimension $1$ submanifold of $\QG$. 
Lyubich proves that the hybrid class $\HH_c$ of any $c\in \M$ is a co-dimension $1$ submanifold of $\QG$.

\begin{teor}[Lem~4.3 \& Thm~4.13, \cite{Ly}] \label{Lyubichconnectedmating}
The connectedness locus $\CC$ is homeomorphic to the product $\HH_0\times\M$. The homeomorphism is horizontally analytic and analytic in both variables for $c\in int(\M)$.
\end{teor}
Here the first coordinate $f_0\in\HH_0$ of a pair 
$(f_0,c)\in\HH_0\times\M$ represents the external class so that 
$\{f_0\}\times\M$ is mapped homeomorphically onto $\ZZ_{f_0}\cap\CC$. This locus (this is, $\HH_c$) is moreover sitting smoothly in $\QG$.
\begin{teor}[\cite{Ly}, Lem.~4.3\& Thm~4.23]\label{smooth_vertical_fibers}
   The vertical fibers $\ZZ_{f_0}$, $f_0\in\HH_0$ are complex analytic curves.
\end{teor}

Lyubich defined the Teichm{\"u}ller-Sullivan metric on 
the space $\QM$ and $\QG$ as:
$$dist_{TS}(f_1,f_2) = \inf \log K_h,$$
where $h$ ranges over all hybrid equivalences in the sense of Douady and Hubbard (see e.g.\cite[Section 3.5]{Ly}). 
The Teichm{\"u}ller-Sullivan metric is a metric on the space of Douady and Hubbard external classes. But it is not a metric on the space of Lyubich external classes, because any two real-analytically conjugate external maps would have distance zero.

In order to define a good Teichm{\"u}ller distance between hybridly equivalent quadratic-like maps 
we ask additionally here that a hybrid equivalence $\phi$ between two polynomial-like maps $f_i:V_i\to V_i'\in\QM$, 
$i=1,2$ is a global quasiconformal map $\phi:\Chat\to\Chat$ fixing $\infty$ (and $0$), 
which restricts to a hybrid equivalence in the sense of Douady and Hubbard, 
i.e. $\phi(V_1)=V_2$, 
$\phi\circ f_1= f_2\circ\phi$ on $V_1$ and 
$\overline{\partial}\phi=0$ on $K_{f_1}$. 
Then we define the Teichm{\"u}ller distance between
$f_1$ and $f_2$ as 
$$
dist_T(f_1, f_2) :=\inf\{\log K_\phi\;|\;\phi:\Chat\to\Chat, \phi(\infty)=\infty  
\textrm{ is a hybrid conjugacy }\}.
$$
Note that for $g_i:\S^1\to\S^1$ the Lyubich respective external classes of $f_i$, we have by construction 
$$
dist_T(f_1, f_2) =\inf\log K_\phi
$$
where $\phi$ ranges over all reflection symmetric in $\S^1$ quasiconformal homeomorphisms $\phi:\Chat\to\Chat$ fixing $\infty$ and conjugating $g_1$ to $g_2$ on neighbourhoods of $\S^1$. 
Thus this defines a distance between the connectedness loci in vertical fibers. 

External maps and external classes are very useful objects, 
because they are in some sense canonical. 
However a quick look at the proof of Douady and Hubbards straightening theorem \cite[Prop~5]{DH} reveals that the actual engine that makes the proof work is the notion of pre-exterior equivalence introduced in the introduction.

For the convenience of the reader we repeat it here. 
For $\ga\subset\C$ a Jordan curve we denote by $D(\ga)$ and $W(\ga)$ respectively the bounded and 
the unbounded (containing $\infty$) connected component of $\Chat\sm\ga$.
\begin{defi}\label{pre-exterior}
A degree $d \geq 2$ pre-exterior map is a holomorphic map 
defined in a neighbourhood of a quasi-circle $\ga$ 
such that the restriction 
{\mapfromto {h_|} \ga {\ga' := h(\ga)\subset \C}} 
is a degree $d$ orientation preserving local diffeomorphism and $D(\ga) \subset \ov{D(\ga)} \subset D(\ga')$. 
For now we denote this pre-external map $(h, \ga)$ 
even though the curve $\ga$ is subordinate to $h$.
\end{defi}

Two pre-exterior maps $(h_i, \ga_i)$, $i=1,2$ are called \textit{pre-exteriorly equivalent}, 
if and only if there exists a biholomorphic map $B$ between the disks
$W_i= W(\ga_i)$ with $B(\infty) = \infty$, 
whose homeomorphic extension to the boundary satisfies 
$B\circ h_1 = h_2\circ B$ on $\ga_1$ and 
thus $B(\ga'_1) = \ga'_2$. 
A \textit{pre-exterior quasiconformal equivalence} is a quasiconformal homeomorphism $\phi$ between 
the disks $W_i$ with $\phi(\infty) = \infty$ and whose homeomorphic extension to the boundary satisfies $\phi\circ h_1 = h_2\circ \phi$ on $\ga_1$ and thus $\phi(\ga'_1) = \ga'_2$.

Note that for any polynomial-like map {\mapfromto f U {U'\subset\subset\C}}, 
(i.e~including those with disconnected filled-in Julia set), 
for any quasi-circle $\ga'\subset U'\sm\ov{U}$ 
such that $D(\ga')$ contains all critical values of $f$ 
and $\ga := f^{-1}(\ga')$, 
the restriction {\mapfromto {h = f_|}\ga{\ga'}} is a pre-exterior map.
Note also that if $\partial U, \partial U'$ are quasi-circles 
and $f$ extends locally univalently 
to a neighbourhood of $\partial U$, 
then we may take $\ga = \partial U$ in this construction. 

Moreover any pre-exterior (qc-)equivalence 
{\mapfromto{(\phi) B}{W_1}{W_2}} extends to a (quasi) conformal conjugacy on neighbourhoods of the boundary 
arcs $\gamma_i$ by lifting under the dynanics 
and for the $B$ this extension is the unique analytic extension.
Note that in the q.c.~case the map $\phi$ and the lift 
of $\phi\circ h_1$ to $h_2$ will only coincide on $\ga_1$ 
and for the extension we keep only the part of the lift in $D_1$. 

The proof of the fundamental mating theorem between polynomials of degree $d$ and degree $d$ external classes also proves a similar mating theorem between polynomials and pre-exterior classes:
\begin{teor}[Prop~5 \& Prop~6,\cite{DH}]\label{connectedmating}
Let $(h,\ga)$ be a degree $d\geq 2$ pre-exterior map 
and let $P$ be a degree $d$ polynomial 
with connected filled-in Julia set $K_P$. 
Then there exists a polynomial-like map {\mapfromto f U {U'}} 
hybridly equivalent to $P$ such that $U, U'$ are quasi-disks, 
$f$ extends locally univalently to a neigbourhood 
of $\ga_U = \partial U$ and 
the pre-exterior map $(f, \ga_U)$ is equivalent to $(h, \ga)$. 
Moreover in degree $d = 2$, e.g. $P = Q_c$ and for $P(z) = z^d$ 
the map $f$ is unique up to affine conjugacy.
\end{teor}
\begin{proof}
The existence part is \cite[Proposition 5, p.~301]{DH}, 
which is formulated in terms of external maps {\mapfromto h \Sen \Sen}, 
but for which the proof uses only a(ny) induced pre-exterior map. 

The uniqueness part is \cite[Proposition 6, p.~302]{DH}, 
combined with the definition of equivalence of pre-exterior maps, 
in the notation of the proof of \cite[Proposition 6, p.~302]{DH} 
the map $\psi$ is global, 
i.e.~is an isomorphism {\mapfromto {\psi} {\Chat\sm K_f}{\Chat\sm K_g}}.
\end{proof}
\begin{cor}\label{from_pre_exterior_to_exterior}
For any degree $d\geq 2$ pre-exterior map $(h, \ga)$  
there exists a polynomial-like germ $f_0(z) = z^d + \OO(z^{d+1})$ 
unique up to conjugacy by $az$, where $a^{d-1} = 1$,  
which has a polynomial-like restriction, {\mapfromto {f_0}{U_0}{U'_0}}, 
such that $(f_0, \partial U_0)$ is a pre-exterior map equivalent to $(h, \ga)$.
\end{cor}

The notion of pre-exterior equivalence is sufficient for us to prove the main theorem of this section stated as Thm.~\ref{Uniformly_qc_copies} below. However the notion of pre-exterior equivalence does not project to QG, i.e. it is not well defined on germs, because two different representatives of the same quadratic-like germ are not in general pre-exteriorly equivalent. 
In order to remedy this and better connect with the notion of external class and provide the description of external class in the case of disconnected filled-in Julia set as promised above we shall introduce the notions of exterior map, exterior equivalence and exterior class. 

\begin{defi} 
A \emph{degree $d \geq 2$ exterior map}  
$h$ is a degree $d$ holomorphic covering map
{\mapfromto h V {V'}}, where 
\begin{enumerate}
\item
$V, V'\subset\subset\C$ are topological annuli with common inner topological boundary $J$ and quasi-circle outer boundaries
\item
the outer boundary of $V$ contained in $V'$ 
\item
$h(z)\to \partial J$ as $z\to\partial J$. 
\end{enumerate}
We call co-domain of $h$ the topological disk $W = W_h\subset\Chat$, 
with $(V'\cup\{\infty\})\subset W$ and boundary $J$. 
\end{defi}

    Let $f:V\to V'$ be a polynomial-like map of degree $d\geq 2$ with connected filled-in Julia set and $V\subset\subset V'\subset\subset\C$ quasi-disks. 
    Then the restriction $f:V\setminus K_f\to V'\setminus K_f$ is an exterior map with inner boundary $J_f$. And any external map of degree $d$ has a complex analytic extension, which is an exterior map with inner boundary $\S^1$.

Two exterior maps {\mapfromto {h_i} {V_i} {V'_i\subset W_i}}, $i= 1,2$ are called equivalent if and only if 
there exists a biholomophic map {\mapfromto B {W_1}{W_2}} with $B(\infty) = \infty$ and 
$B\circ h_1 = h_2\circ B$ on some annulus $\whV\subset V$ with inner boundary $J_1$. 

For {\mapfromto{f}{U}{U'\subset\subset\C}} 
a polynomial-like map with connected filled-in Julia set $K_f$ as above
the inner boundary $J_h$ of the exterior map 
$h= f_{|U\sm K_f}$ coincides with the Julia set 
$J_f = \partial K_f$ of $f$. 
We define the \emph{exterior class} of $f$ as the equivalence class of $h_f$. 
The complex analytic extension of the Lyubich external map for $f$ belongs to the exterior equivalence class of $f$.
Corollary~\ref{from_pre_exterior_to_exterior} shows that any pre-exterior map $(h, \ga)$ defines a unique equivalence class of exterior maps.

\begin{cor}\label{extension_of_qc-pre-equivalence}
Let {\mapfromto{\phi} {W(\ga_1)} {W(\ga_2)}} be a pre-exterior 
q.c.-equivalence between two quadratic pre-exterior maps 
$(h_i, \ga_i)$ 
and let {\mapfromto{f_i} {U_i}{U_i'}}, $[f_i]\in\HH_0$ 
be the quadratic-like maps, 
which are pre-exteriorly equivalent to $(h_i, \ga_i)$, $i=1,2$
as given by \corref{from_pre_exterior_to_exterior}. 
Let $B_i$ be the pre-exterior equivalences. 
Then $\phi$ induces a unique q.c.~homeomorphism 
{\mapfromto {\phi_0} \Chat \Chat}, 
such that the restriction 
{\mapfromto{\phi_0}{U_1'}{U_2'}} is 
a hybrid conjugacy of $f_1$ to $f_2$ and 
$\phi_0 = B_2\circ\phi\circ B_1^{-1}$, 
where the composition is defined. 
Moreover the dilatation of $\phi_0$ is bounded by that of $\phi$.
\end{cor}

\begin{proof}
The map {\mapfromto{\phi_0 =B_2\circ\phi\circ B_1^{-1}}{B_1(W_1)}{B_2(W_2)}} 
defines a pre-exterior q.c.~equivalence between $f_1$ and $f_2$. 
By iterated lifting of $\phi_0\circ f_1$ to $f_2$, 
the map $\phi_0$ extends to a q.c.
conjugacy $\phi_0$ defined on the complement of the filled Julia sets and fixing infinity.
Moreover the filled-in Julia sets $K_i$ of 
$f_i$ are quasi-disk and with interior dynamics uniquely conjugate 
by a biholomorphic map fixing $0$ with derivative $1$. 
Both this interior conjugacy and the conjugacy $\phi_0$
extend to the unique orientation preserving conjugacy of 
the restriction of $f_1$ to $J_1 = \partial K_1$ to 
the restriction of $f_2$ to $J_2 = \partial K_2$. 
Hence by the q.c.~gluing lemma they glue to form a global 
q.c.~homeomorphism {\mapfromto {\phi_0} \Chat \Chat} 
with the stated properties.
\end{proof}

In line with Lyubich's terminology we shall say that $(h, \ga)$ 
is a quadratic pre-exterior map, if $d=2$ and $0\in\D(\ga)$.
Recall that, by Corollary \ref{from_pre_exterior_to_exterior}, any quadratic pre-exterior map $(h,\ga)$ is equivalent to the restriction of a unique quadratic-like map 
$f_0(z) = z^2 + \OO(z^3)$.

\begin{teor}[Lyubich]\label{disconnectedmating}
Let {\mapfromto {f_0}{U_0}{U_0'}}, $f_0(z) = z^2 + \OO(z^3)$ be a quadratic like map 
with quasi-circle bounded domain and range and $f_0$ extending locally univalently 
to a neighbourhood of $\partial U_0$. Then for any $w\in U_0'\sm K_{f_0}$ there exists 
a unique quadratic-like map {\mapfromto {f_w}{U_w}{U_w'}}, $f_w(z) = c(w) + z^2 + \OO(z^3)$ 
and a pre-exterior equivalence of $(f_w,\partial U_w)$ to $(f_0, \partial U_0)$, 
{\mapfromto {B_w}{W_w}\Chat} with $c(w)\in W_w$, $B_w(c(w)) = w$ 
and $f_w(\ov{W}_w\cap U_w)\subset W_w$. 

Moreover $(w,z) \mapsto f_w(z)$ and 
$(w,z) \mapsto B_w(z)=B(w,z)$ are complex analytic 
as a function of two complex variables.
\end{teor}
\begin{proof}
This is proved in the proof of \cite[Lemma 4.14]{Ly}. The statement and proof is formulated in terms of the external class 
{\mapfromto {g_{f_0}}\Sen \Sen}, but the proof relies only on 
the fact that $g$ has a complex analytic extension, 
which is pre-exteriorly equivalent to the pre-exterior map of $f_0$. 
The statements here are slightly weaker than those 
of Lyubich (but the proof carries over verbatim modulo notation as described). 
\end{proof}

Here we consider the set
$$
\chZZ_{f_0} := \{ f : [f]\in\ZZ_{f_0}\text{ and $(f; \partial U_f)$ is pre-exteriorly equivalent to $(f_0, \partial U_0)$} \}.
$$
The proof in \cite{Ly} that $\ZZ_{f_0}$ is a Riemann 
surface is done via families of representatives, 
hence $\chZZ_{f_0}\subset\QM$ is also a Riemann surface 
isomorphic to $\{[f] : f\in \chZZ_{f_0}\}$. 
By \thmref{disconnectedmating} the latter
is a simply connected relatively open subset of $\ZZ_{f_0}$ 
compactly containing $\M_{f_0}$ and in particular a Riemann surface which compactly contains $\M_{f_0}$. 
Thus $\chZZ_{f_0}$ compactly contains the connectedness locus 
$\chM_{f_0} := \{ f\in\chZZ_{f_0} : [f]\in\M_{f_0}\}$
The following Theorem is the main technical theorem of this section.

\begin{teor}\label{Uniformly_qc_copies}
   Let {\mapfromto{f_i}{U_i}{U_i'}}, $i=0,1$ 
   be quadratic-like maps with $[f_i]\in\HH_0$, 
   quasi-circle bounded domains 
   and locally univalent extensions to neighbourhoods of 
   $\partial U_i$. 
   Suppose {\mapfromto{\phi} {\Chat\setminus U_0}{\Chat\setminus U_1}}
   is a pre-exterior q-c-equivalence.
   Then $\phi$ induces a quasiconformal homeomorphism 
   between the parameter Riemann-surfaces 
   $$
   \mapfromto\Phi{\chZZ_{f_0}}{\chZZ_{f_1}}
   \qquad\text{with}\quad ||\mu_\Phi||_\infty \leq 
   ||\mu_\phi||_\infty.
   $$
 That is the holonomy of the connectedness locus extends 
 to a quasiconformal homeomorphism $\Phi$, 
 whose dilatation is bounded by that of $\phi$.
\end{teor}
\begin{proof}
Towards a proof of \thmref{Uniformly_qc_copies}
we start by defining the map $\Phi$, 
which is a standard construction 
originally due to Douady and Hubbard, 
who calls it the miracle of continuity in degree $2$. 
The idea is also used in \cite[Lemma 4.22]{Ly}.

On $\chM_{f_0}$ the map is given by the hybrid holonomy. 
For $w\in U_0'\setminus K_{f_0}$ and $f_w\in\chZZ_{f_0}\setminus\chM_{f_0}$ given by 
\thmref{disconnectedmating} we proceed as follows. 
Let {\mapfromto{\phi} {W_0}{W_1}}, be a pre-exterior q-c-equivalence, where $W_i = \Chat\setminus \overline{U}_i$ for $i = 0, 1$. 
By \corref{extension_of_qc-pre-equivalence} the map 
{\mapfromto{\phi} {W_0}{W_1}} extends to a gobal 
q.c.~homeomorphism {\mapfromto{\phi_0} {\Chat}{\Chat}} fixing infinity
that restricts to a hybrid conjugacy 
{\mapfromto{\phi_0} {U_0'}{U_1'}} and to a q.c-conjugacy 
{\mapfromto{\phi_0} {\Chat\setminus K_{f_0}}{\Chat\setminus K_{f_1}}} 
between $f_0$ and $f_1$. 
We define $\Phi(f_w) = \wtf_{\phi(w)}$, 
where $\wtf_{\phi(w)}$ is the map in $\chZZ_{f_1}\setminus\chM_{f_1}$ given by \thmref{disconnectedmating}. 

In order to control the dilatation of $\Phi$ 
we firstly embed $\Phi$ in a holomorphic motion 
in the space of Quadratic-like maps of $\chZZ_{f_0}$ over $\D$. 
Secondly we rectify this motion to become biholomorphic to 
a holomorphic motion in $\Chat$.

We start with a holomorphic motion of the dynamical sphere for 
$f_0$. 
Let $\sigma_0$ be the Beltrami form given by $\phi_0$, i.e. 
the Beltrami form which in local charts is given by 
$\mu_0\frac{d\overline{z}}{dz}$, 
where $\mu_0 = \frac{d\phi_0/d\overline{z}}{d\phi_0/dz}$. 
Let $k := ||\sigma_0||_\infty < 1$ denote the supremum norm of 
$\sigma_0$. 

For $\lambda\in\D(1/k)$ define a complex analytic family of 
Beltrami forms $\sigma_0^\lambda := \lambda\sigma_0$, 
so that $\sigma_0^1=\sigma_0$. 
By the measureable Riemann mapping theorem with parameters there
exists a family {\mapfromto{\phi_0^\lambda}{\Chat}{\Chat}} 
of quasiconformal homeomorphisms 
fixing $\infty$ and fixing $0$ with derivative $1$ 
and depending holomorphically on $\lambda\in\D$. 
Let {\mapfromto{f_0^\lambda}{U_0^\lambda}{U_0^{'\lambda}}} 
be the family of quadratic-like maps 
$f_0^\lambda := \phi_0^\lambda\circ f_0\circ(\phi_0^\lambda)^{-1} 
\in \HH_0$, 
where $U_0^\lambda = \phi_0^\lambda(U_0)$ and 
$U_0^{'\lambda} = \phi_0^\lambda(U_0')$. 
Then $\phi_0^1 = \phi_0$ and $f_0^1 = f_1$ 
due to the normalizations. 
We call the family $(f_0^\lambda)_{\lambda\in\D(1/k)}$ the 
Beltrami-disk from $f_0$ through $f_1$.
To alleviate notation we define for $\lambda\in\D(1/k)$ 
$$
\chZZ\la := \chZZ_{f_0^\lambda}.
$$
More generally for $f\in\chM_{f_0}\subset\chZZ_{f_0}$, 
{\mapfromto f {U_f} {U_f'}} we define the Beltrami-disk from $f$ 
through $f^1 = \Phi(f)$ as follows. 

By iterated lifting we can extend the pre-exterior equivalence to a biholomorphic map
{\mapfromto{B = B_f}{\Chat\setminus K_f}{\Chat\setminus K_{f_0}}} 
fixing $\infty$ and such that 
$B(U_f') = U_0'$ and $B\circ f = f_0\circ B$ on 
$U_f$. 
Let $\sigma_f^\lambda$ denote the complex analytic (in $\lambda$) family of $f$ invariant Beltrami forms on $\Chat$ given by 
$\sigma_f^\lambda = B_f^*(\sigma_0^\lambda)$ 
on $\Chat\setminus K_{f_0}$ and by $0$ on $K_f$. 
Let {\mapfromto{\phi_f^\lambda}{\Chat}{\Chat}} be the family 
of quasiconformal homeomorphisms 
fixing both $\infty$ and $0$ and further normalized 
so that the family {\mapfromto{f^\lambda}{U_f^\lambda}{U_f^{'\lambda}}} of polynomial-like maps 
$f^\lambda := \phi_f^\lambda\circ f\circ(\phi_f^\lambda)^{-1}$, 
where $U_f^\lambda = \phi_f^\lambda(U_f)$ and 
$U_f^{'\lambda} = \phi_f^\lambda(U_f')$ is quadratic-like, 
i.e.~$f^\lambda(z) = c(f^\lambda) + z^2 + \OO(z^3)$. 
Then $\phi_f^\lambda$ and $f^\lambda$ depend holomorphically on $\lambda$ and $\phi_f(\lambda, z) = \phi_f^\lambda(z)$ is a holomorphic motion of the dynamical sphere for $f$. 

The map $f^\lambda\in\chZZ\la$, $\lambda\in\D(1/k)$ 
is the unique such map hybridly equivalent to $f$, 
in particular $f^1 = \Phi(f) \in\chZZ_{f_1}$. 
We note for later use that moreover 
{\mapfromto{B_{f^\lambda}}{\Chat\setminus K_{f^\lambda}}{\Chat\setminus K_{f_0^\lambda}}} satisfies 
$B_{f^\lambda} = \phi_0^\lambda\circ B_f\circ(\phi_f^\lambda)^{-1}$.

Finally for {\mapfromto f {U_f} {U_f'}},  $f\in\chZZ_{f_0}\setminus\chM_{f_0}=\chZZ_0\setminus\chM_0$ 
we define the Beltrami-disk from $f$ through $f^1 = \Phi(f)$ as follows. 
Recall that $f\in\chZZ_{f_0}$ if and only if 
the pre-exterior map $(f, \partial U_f)$ 
being pre-exteriorly equivalent to $(f_0, \partial U_0)$. 
Let {\mapfromto{B=B_f}{\Chat\setminus U_f}{\Chat\setminus U_0}} 
be the pre-exterior equivalence. 
Let $\sigma_f^\lambda$ be the unique complex analytic  
family of $f$-invariant Beltrami forms on $\Chat$, 
which equals $B_f^*(\sigma_0^\lambda)$ on $\Chat\setminus U_f$ 
and $0$ on the Cantor set $K_f$. 
As before let {\mapfromto{\phi_f^\lambda}{\Chat}{\Chat}} be the family 
of quasiconformal homeomorphisms fixing both $\infty$ and $0$ 
and further normalized so that the family 
{\mapfromto{f^\lambda}{U_f^\lambda}{U_f^{'\lambda}}} of polynomial-like maps 
$f^\lambda := \phi_f^\lambda\circ f\circ(\phi_f^\lambda)^{-1} 
= c(f,\lambda) + z^2 + \OO(z^3)$, 
where $U_f^\lambda = \phi_f^\lambda(U_f)$ and 
$U_f^{'\lambda} = \phi_f^\lambda(U_f')$ is quadratic-like. 
Then $\phi_f^\lambda$ and $f^\lambda$ depend holomorphically on $\lambda$ 
and $\phi_f(\lambda, z) = \phi_f^\lambda(z)$ is a holomorphic motion 
of the dynamical sphere for $f$. 

The pre-exterior equivalence $B_f$ univalently extends by iterated lifting 
under the dynamics to a domain $W_f\simeq\D$ 
containing the critical value $c$ for $f$ and such that 
{\mapfromto{f_|}{\Chat\setminus{\overline{W}_f}}{f(\Chat\setminus{\overline{W}_f})}} 
is restriction equivalent to $f$. 
The univalent map {\mapfromto{B_{f^\lambda}}{\phi_f^\lambda(W_f)}{\phi_0^\lambda(B_f(W_f))}} is a pre-exterior equivalence between 
$(f^\lambda,\partial U_f^\lambda)$ and $(f_0^\lambda,\partial U_0^\lambda)$ with 
$B_{f^\lambda}(c(f^\lambda)) = \phi_0^\lambda(B_f(c(f)))$. 
Thus $f^\lambda\in\chZZ\la$, $\lambda\in\D(1/k)$ is the unique such map 
corresponding to $\phi_0^\lambda(B_f(c(f)))\in\Chat\setminus K_{f_0^\lambda}$ 
as given by \thmref{disconnectedmating}.
In particular $f^1 = \Phi(f) \in\chZZ_{f_1}$. 

Then by construction the graphs of the holomorphic functions 
$\D(1/k)\ni\lambda\mapsto f^\lambda$ are disjoint and 
hence together constitute af holomorphic motion 
{\mapfromto{\chHH}{\D(1/k)\times\chZZ_{f_0}}{\chZZ}}
of $\chZZ_{f_0}$ in 
$\chZZ := \cup_{\lambda\in\D(1/k)}\chZZ_{f_0^\lambda}$ 
over $\D(1/k)$ with base point $0$. 
Moreover $\Phi = \chHH^1 := \chHH(1,\cdot)$ 
so that the bound $k$ on the complex dilation of $\Phi$ 
would follow if this was a holomorphic motion in $\Chat$.

To complete the proof of \thmref{Uniformly_qc_copies} 
let {\mapfromto{\varphi}{\chZZ_{f_0}}{\D}} 
be a uniformization with $\varphi(f_0) = 0$. 

The holomorphic motion $\chHH$ equips the parameter space 
$\chZZ_{f_0}$ with a complex analytic family of Beltrami forms $\chsi^\lambda$, $\lambda\in\D(1/k)$. 

Define a complex analytic family $\whsi^\lambda$, 
$\lambda\in\D(1/k)$ of Beltrami forms 
on $\Chat$ by
$$
\whsi^\lambda := 
\begin{cases}
    \varphi_*(\chsi^\lambda) 
:= (\varphi^{-1})^*(\chsi^\lambda) & \text{on }\D,\\
0 & \Chat\sm\Dbar.
\end{cases}
$$
and let 
{\mapfromto{\whHH}{\D(1/k)\times\Chat}{\Chat}} 
be the holomorphic motion of $\Chat$ in $\Chat$ over $\D(1/k)$ 
with base point $0$ obtained by integrating the family of 
Beltrami forms $\whsi^\lambda$ and normalizng so that $0$, $1$ and $\infty$ are fixed under the motion. 
Then $\whHH^1(z) := \whHH(1,z)$ is quasiconformal with complex dilatation bounded by $k$ and thus 
$$
{\mapfromto{\psi := \chHH^1\circ\varphi^{-1}\circ(\whHH^1)^{-1}(z))}{\whHH^1(\D)}{\chZZ_{f_1}}}
$$ 
is biholomorphc, 
since it preserves the standard almost complex structures. 
Hence also 
$$
\Phi = \psi\circ\whHH^1\circ\varphi
$$ has complex dilatation bounded by $k = ||\mu_\phi||_\infty$.
This completes the proof of \thmref{Uniformly_qc_copies}.
\end{proof}
\vspace{.1cm}
 The following Corollary is an elementerary 
consequence of the \thmref{Uniformly_qc_copies} and \corref{from_pre_exterior_to_exterior} and is left to the reader.
\begin{cor}\label{fin}
   Let {\mapfromto{f_i}{U_i}{U_i'}}, $i=0,1$ 
   be quadratic-like maps with
   quasi-circle bounded domains 
   and locally univalent extensions to neighbourhoods of 
   $\partial U_i$. 
   Suppose {\mapfromto{\phi} {\Chat\setminus U_0}{\Chat\setminus U_1}}
   is a pre-exterior q-c-equivalence. 
   Let {\mapfromto{f_{0,i}}{U_{0,i}}{U_{0,i}'}}, $i=0,1$ 
   be quadratic-like maps in $\HH_0$ with 
   $(f_i, \partial U_i)$ pre-exteriorly equivalent to 
   $(f_{0,i}, \partial U_{0,i})$, $i = 0, 1$.
   Then $\phi$ induces a quasiconformal homeomorphism 
   between the parameter Riemann-surfaces 
   $\chZZ_{f_{0,i}}\ni f_i$ :
   $$
   \mapfromto\Phi{\chZZ_{f_{0,0}}}{\chZZ_{f_{0,1}}}
   \qquad\text{with}\quad ||\mu_\Phi||_\infty \leq 
   ||\mu_\phi||_\infty.
   $$

\end{cor}
The following Theorem completes the proof of quasiconformal homeomorphicity between satellite copies of the Mandelbrot set with rotation numbers having the same denominator:
\begin{teor}\label{Harvest}
   Let {\mapfromto{f_{i}}{U_i}{U_i'}}, $i=\lambda,\nu$ 
   be families of quadratic-like maps, with connectedness loci $M_f=M_{p/q}$ and $M_g=M_{p'/q}$ respectively. Assume that, for $\lambda_0 \in M_f$, $\phi_{\lambda_0}$ is a hybrid equivalence between $f_{\lambda_0}$ and $g_{\nu_0}$ with $||\mu_\phi||_\infty \leq k <1$. Then, if 
   $$ \xi: M_f \rightarrow M_g$$ is unbranched near $\lambda_0$, we have, for $r>0$ 
   $$\lim_{r \rightarrow 0}\,\,\,\,\, \mbox{ess}\sup_{|\lambda-\lambda_0|<r}||\mu_{\xi}(\lambda)|| \leq k + \epsilon,$$ with $\epsilon$ arbitrarily small.
\end{teor}
\begin{proof}
This is a consequence of the $\lambda$-Lemma in \cite{Ly} (Appendix 2, page 411) together with Corollary \ref{fin}.
\end{proof}

\section{Extension of $\xi$ to a neighbourhood of $M_f$}\label{laststep}
Theorem \ref{Harvest} says that, for any $\lambda_0 \in M_f$ there exists a neighbourhood of $\lambda_0$ where $\xi$ has bounded dilatation. We will now construct a neighbourhood $U_f$ of $M_f$ where $\xi$ extends as a quasiconformal map.

\begin{teor}\label{nbh}
    There exists a a neighbourhood $U_f$ of $M_f$ where $\xi$ extends as a quasiconformal map.
\end{teor}
Let $M_f=M_{p/q}$, and let $\mathcal{N}_0$ be the (big) potential fixed at the beginning of section \ref{Extension}. Call $W_{p/q}^{\mathcal{N}_0}$ the wake $W_{p/q}$ up to potential $\mathcal{N}_0$. Fix $d>0$.
The extension of the map $\xi$, which gives proof of Theorem \ref{nbh}, will be done on $U_f:= W_{p/q}^{\mathcal{N}_0} \cup \D(\lambda_0, d) $ in three steps:
\begin{enumerate}
    \item Extend $\xi$ to $W_{p/q}^{\mathcal{N}_0}$. This will be done in Section \ref{wake}.
    \item Define $\xi$ between $\D\cap\D(\lambda_0, d)$ 
to $\D\cap\D(\nu_0, d')$ to be the rotation 
$\lambda\mapsto (\nu_0/\lambda_0)\lambda$.
  \item Extend $\xi$ to $\D(\lambda_0, d)\setminus (W_{p/q}^{\mathcal{N}_0} \cup \D)$.
   This will be done in Section \ref{cusp}.
\end{enumerate}

\subsection{Extension of $\xi$ to $W_{p/q}^{\mathcal{N}_0}$}\label{wake}
Recall that $M_f=M_{p/q}$ and that $\mathcal{N}_0$ is the potential fixed at the beginning of section \ref{Extension}.
Define the $\frac{1}{2}$ \textit{dynamical dyadic wake} $\Delta\la(\frac{1}{2})$ to be the topological disc bounded by the external rays landing at the $\frac{1}{2}$ dyadic tip, and the equipotential of potential $\frac{\mathcal{N}_0}{2^q}$.
If the critical value does not belong to $\Delta\la(\frac{1}{2})$, we define the $\frac{k}{2^2}$ \textit{dynamical dyadic wake} $\Delta\la(\frac{k}{2^2})$, $k=\{1,2\}$
to be a preimage of $\Delta\la(\frac{1}{2})$ under $f\la$; note that then
the map
$f\la=P_{\lambda|}^{q}: \Delta\la(\frac{k}{2^2}) \rightarrow \Delta\la(\frac{1}{2})$ is an isomorphism.
Repeating the process, we define the
 $\frac{m}{2^n}$ \textit{dynamical dyadic wake} $\Delta\la(\frac{m}{2^n})$ to be a preimage of $\Delta\la(\frac{1}{2})$ under $f\la^n$ with $f\la^n=P_{\lambda|}^{nq}: \Delta\la(\frac{m}{2^n}) \rightarrow \Delta\la(\frac{1}{2})$ isomorphism. 
 Note that we can define the $\frac{m}{2^n}$ dynamical dyadic wakes until (and including) the level containing the critical value $v\la$ of $f\la$.
 Note also that the $\frac{m}{2^n}$ dyadic tip belongs to the boundary of the $\frac{m}{2^n}$ dynamical dyadic wake $\Delta\la(\frac{m}{2^n})$.
Define the  $\Delta_{p/q}(\frac{m}{2^n})$ to be the corresponding $\frac{m}{2^n}$ \textit{parameter dyadic wake}: the topological disc in parameter plane bounded by the parameter rays of same angle as the external rays bounding the dynamical dyadic wake $\Delta\la(\frac{m}{2^n})$, and the parameter equipotential of same potential as the equipotential bounding the dynamical dyadic wake $\Delta\la(\frac{m}{2^n})$. Note that the $\frac{m}{2^n}$ dyadic tip in parameter plane, that we call $\lambda(\frac{m}{2^n})$, belongs to the boundary of the $\frac{m}{2^n}$ parameter dyadic wake $\Delta_{p/q}(\frac{m}{2^n})$.
Define accordingly the $\frac{m}{2^n}$ \textit{parameter dyadic wake} $\Delta_{p'/q}(\frac{m}{2^n})$.
\begin{prop}
    There exists a $K<\infty$ such that, for every dyadic $\frac{m}{2^n},\,m\in[1,2^{n-1]}$, there exists a $K$-quasiconformal map
$$\phi^{p/q}_{\frac{m}{2^n}}: \Delta_{p/q}(\frac{m}{2^n}) \rightarrow \Delta\la(\frac{m}{2^n})$$
respecting the B\"ottcher values at the boundaries (meaning that it extends on the boundary to the conformal map sending parameter rays to dynamical rays, respecting potentials, and parameter equipotentials to dynamical equipotentials, respecting the arguments); and similarly there exists $K$-quasiconformal maps
$$\phi^{p'/q}_{\frac{m}{2^n}}: \Delta_{p'/q}(\frac{m}{2^n}) \rightarrow \Delta\nn(\frac{m}{2^n}).$$
\end{prop}
\begin{proof}
The external rays $R\la$ and $R\la'$, their preimages and all the landing points move holomorphically on the $p/q$ parameter wake $W_{p/q} \supset M_{p/q}$. Hence, by Slodkowsky Theorem, there exists a holomorphic motion $$H_0: W_{p/q} \times \widehat \C \rightarrow \widehat \C$$
extending the motion of the rays and their landing points. 
Let $U(\Delta_{p/q}(\frac{1}{2}))$ be a neighbourhood of the $\frac{1}{2}$ parameter wake $\Delta_{p/q}(\frac{1}{2})$, consider the holomorphic motion
$$H_0:U(\Delta_{p/q}(\frac{1}{2}))\times \Delta\la(\frac{1}{2}) \rightarrow  \C$$
with base point $\lambda(\frac{1}{2})$.
For $\lambda$ in a neighbourhood of $\Delta\la(\frac{m}{2^n})$ we can extend the motion with base point 
$\lambda(\frac{1}{2})$ of the dynamical wake $\Delta\la(\frac{1}{2})$ to a holomorphic motion of $\Delta\la(\frac{m}{2^n})$ also with base point $\lambda(\frac{m}{2^n})$
using iterated lifting. 
We thus  obtain a holomorphic motion
$$H_{\frac{m}{2^n}}:U(\Delta_{p/q}(\frac{m}{2^n}))\times \Delta\la(\frac{m}{2^n}) \rightarrow  \C$$
with base point $\lambda(\frac{m}{2^n})$.
If the hyperbolic distance between the base point $\lambda(\frac{m}{2^n})$ of the holomorphic motion $H_(\frac{m}{2^n})$ and $\lambda \in \Delta\la(\frac{m}{2^n})$ is bounded by $r$, the dilatation of $H_0$ is bounded by $e^r$.
As it is enough to bound the hyperbolic diameter of the dyadic parameter wakes close to the root, and each such wake is contained in a cage as constructed by Kapiamba in \cite{K} (see \cite[Proposition 4.19 and Figure 6.2]{K}) take his sub-curve $\gamma^\sigma_{n,2}$ to be a sub-arc either of the oblique $\theta'$ ray of slope 1, or of the oblique $\theta$ ray of slope $-1$, depending on the side).
By Theorem 1.4 in \cite{K} each of these cages have uniformly bounded hyperbolic diameter. This gives $s>0$ such that for any $\frac{m}{2^n}$ the hyperbolic diameter of the $\frac{m}{2^n}$ parameter wake $\Delta_{p/q}(\frac{m}{2^n})$ is bounded by $s$, 
and so the dilatation of the holomorphic motion between any two parameters there inside is bounded by $e^s$. 
The quasiconformal map 
$$\phi^{p/q}_{\frac{m}{2^n}}(\lambda):=H\la(v\la))^{-1},\,\,\, \lambda \in \Delta_{p/q}(\frac{m}{2^n}) $$
has therefore dilatation bounded by $e^s$.
This ends the proof.
\end{proof}
Let $\psi_{\lambda(\frac{m}{2^n}), \nu(\frac{m}{2^n})}$ be the hybrid equivalence between $f_{\lambda(\frac{m}{2^n})}$ and $g_{\nu(\frac{m}{2^n})}$. Then, for every dyadic $\frac{m}{2^n}$ we obtain a $K$ quasiconformal map 
$$ (\phi^{p'/q}_{\frac{m}{2^n}})^{-1}\circ   \psi_{\lambda(\frac{m}{2^n}), \nu(\frac{m}{2^n})} \circ \phi^{p/q}_{\frac{m}{2^n}}: \Delta_{p/q}(\frac{m}{2^n}) \rightarrow \Delta_{p'/q}(\frac{m}{2^n}) $$
between the $\frac{m}{2^n}$ parameter dyadic wakes of $M_f=M_{p/q}$ and $M_g=M_{p'/q}$ respectively.
Therefore, we can extend the quasiconformal map $\xi: M_f \rightarrow M_g$ to the whole $p/q$-wake $W_{p/q}$ that contains $M_f$ by defining $\xi: W_{p/q} \rightarrow W_{p'/q}$ to be equal to $ (\phi^{p'/q}_{\frac{m}{2^n}})^{-1}\circ   \psi_{\lambda(\frac{m}{2^n}), \nu(\frac{m}{2^n})} \circ \phi^{p/q}_{\frac{m}{2^n}}: \Delta_{p/q}(\frac{m}{2^n}) \rightarrow \Delta_{p'/q}(\frac{m}{2^n})$ on the $\frac{m}{2^n}$ parameter dyadic wakes and equal to the conformal map given by B\"ottcher on $W_{p/q}\setminus (M_f \bigcup_{\frac{m}{2^n}}\Delta_{p/q}(\frac{m}{2^n})$.
Restrict to the potential $\potz$, we have the extension to $W_{p/q}^{\mathcal{N}_0}$.

\subsection{Extension of $\xi$ to $\D(\lambda_0, d)\setminus (W_{p/q}^{\mathcal{N}_0} \cup \D)$.}\label{cusp}
In this section we will define $\xi$ on $\D(\lambda_0, d)\setminus (W_{p/q}^{\mathcal{N}_0} \cup \D)$ by quasiconformal interpolation between the extension constructed in Section \ref{wake} and the map $\lambda\mapsto \lambda\nu_0/\lambda_0$, 
which sends points on the boundary ray landing at $\lambda_0$ to points of the same potential on the boundary ray landing at $\nu_0$, preserving the relative sides of the root.
Note that, setting $\lambda(\Lambda) = 
    \lambda_0\e^{(\Lambda/q)}$ and 
    $\nu(N) = \nu_0\e^{(N/q)}$, and abusing notation and writing $\Lambda$ and $N$ for the inverse maps we have $$N(\lambda(\Lambda)\nu_0/\lambda_0) = \Lambda.$$ 
    That is the rotation $\lambda\mapsto \lambda\nu_0/\lambda_0$ induces the identity 
    map from the $\Lambda$-plane to the $N$-plane.
    So, in the $\Lambda$-plane, the map $\xi$ is given by interpolating quasiconformally between the identity on the imaginary axis, and the lifted equipotential map on the boundary rays. 

We will construct $\xi$ in two steps:
\begin{enumerate}
    \item Construct a quasiconformal interpolation between 'dynamical cusps' of $f_{\lambda_0}$ and $g_{\nu_0}$, in Section \ref{dc} (see Figure \ref{Dyn}).
    \item Construct quasiconformal maps between the dynamical planes and the parameter planes, in Section \ref{dynpar} (see Figure \ref{All}).
\end{enumerate}
The composition will give the map $\xi$. As the boundary values of $\xi$ are determined, this composition must respect them.
\subsubsection{Qc interpolation between 'dynamical cusps' of $f_{\lambda_0}$ and $g_{\nu_0}$}\label{dc}
    For $\kappa_0 \in \{\lambda_0, \nu_0\}$ let 
    $0 < \theta'_\ilnz < \theta_\ilnz <1$. 
    We shall refer to $\theta'_\ilnz$ as the lower angle and to $\theta_\ilnz$ as the upper angle. 
    We shall discuss mainly the upper angle case and only point out the difference to the lower angle case. 
    Since the angles $\theta_\ilnz$ and $\theta'_\ilnz$ 
    are distinct for the two values of $\ilnz$, 
    we can also write $\ilnz(\theta)$ and $\ilnz(\theta')$ 
    respectively. We shall especially use this in indices of functions and sets.
    
    It will be convenient to work with the quadratic family $Q_c(z) = z^2 + c$, where the critical point is $0$ and the critical value is the parameter $c$. We shall still however write 
    $f\la$ and $g\nn$ for the map $Q_c^q$, $c =\kappa/2-\kappa^2/4$, $\iln\in\{\lambda, \nu\}$. 
    
    As the constructions for 
    $\lambda, \lambda_0 = \e^{i2\pi p/q}$ and $\nu, \nu_0 = \e^{i2\pi p'/q}$, $(p, q) = (p', q) = 1$ are similar, we shall detail only the ones for $\lambda, \lambda_0$, 
    leaving the similar $\nu, \nu_0$ case to the reader.\\
    In what follows we will often use a complex rescaling of $Resit(Q_\cz^q,\alpha_\cz)$, the R{\'e}sidu It{\'e}ratif of $Q_\cz^q$ at $\alpha_\cz$, so we define the Normalized R{\'e}sidu It{\'e}ratif (NRI) as:
    $$\whResit_\cz := \frac{2\pi i Resit(Q_\cz^q, \alpha_\cz)}{q} .
    $$
    Since $Q_\cz$ has a unique critical point, $\Re(Resit(Q_\cz^q, \alpha_\cz)) > 1/\pi^2 > 0$ by \cite[Theorem C]{BE}, 
    and thus $\Im(\whResit_\cz) >0$.\\
    Let $\Phi = \Phi_\theta$, $\cz = \cz(\lambda_0)$ denote the attracting Fatou-coordinate on the connected component $B_\cz(\alpha_\cz)$ of the immediate parabolic basin containing $\cz$, normalized by 
    $\Phi(\cz) = -\whResit$.\\
    Let $\Psi = \Psi_\theta$ denote the repelling Fatou-parameter associated with the repelling petal containing the end of the ray of (upper) argument $\theta$, or equivalently whose repelling axis follows the domain of $\Phi$ in the counter clockwise order around $\alpha_\cz$.
    We normalize $\Psi$ so that the lower horn map 
    \begin{equation}\label{parabolic_lower_horn_map}
        h_l(z) = h(z) := \Phi\circ\Psi(z) = z + \oo(1)
        \quad\textrm{as}\quad \Im(z) \to -\infty.
    \end{equation}
    The pre-image $A= A_\theta := \Psi^{-1}(\C\setminus K_\cz)$ is a $T_1$ invariant simply connected domain, a topological horizontal strip. 
    Moreover the restriction of $\Psi$ to $A$ is a universal 
    covering and the core geodesic $\whR =\whR_\theta = T_1(\whR)$ 
    is mapped diffeomorphically onto the external ray 
    $R_\cz(\theta)$ for $Q_\cz$.\\
    The strip $A$, and consequently also $\whR$, is contained in 
    the lower half plane, hence $\whR$ together with the 
    real line $\R$ bounds a $T_1$ invariant strip $S_\theta$. 
    (The image of $\R$ under $\Psi$ bounds a ``sepal'' in the component of the immediate parabolic basin following $R(\theta)$ in the counter clockwise order.) 
    Let $m=m_\theta>0$ denote the modulus of the quotient cylinder $S_\theta/T_1$. 
    Let $\eta_\theta: S_\theta\to SS_\theta :=\{x+iy| -m<y<0\}$ 
    be the isomorphism commuting with $T_1$ and normalized by 
    $\eta_\theta(0) = 0$ for $\theta = \theta(\ilnz)$ 
    and $\cz = \cz(\ilnz)$. 
    In order to distinguish the two cases $\ilnz = \lambda_0, \nu_0$ we shall temporarily in this paragraph replace the index 
    $\theta =\theta(\ilnz)$ with $\ilnz$ and similarly the index 
    $\cz(\ilnz)$ with $\ilnz$.
    Using the two isomorphisms $\eta\l0: S\l0\to SS\l0$ 
    and $\eta\n0: S\n0\to SS\n0$, which extend to isomorphisms of neighbourhoods by Schwarz reflections, we transport 
    the boundary maps: the identity $z\mapsto z$ on $\R$ to the identity 
    on $\R$, and the induced equipotential map between $\whR\l0$ and $\whR\n0$  to a real analytic map between the two lines $(-im\l0+\R)$ and 
    $(-im\n0+\R)$ forming the lower boundaries of $SS\l0$ and $SS\n0$ respectively. 
    Then both boundary maps commute with $T_1$. 
    (As a precaution we remark that the two heights $m\l0$ and $m\n0$ of the two strips $SS\l0$ and $SS\n0$ have no reason to be the same.)
    See also Figure \ref{Dyn}.
    \begin{figure}[tbh]
    \centering
    \includegraphics[scale = 0.5]{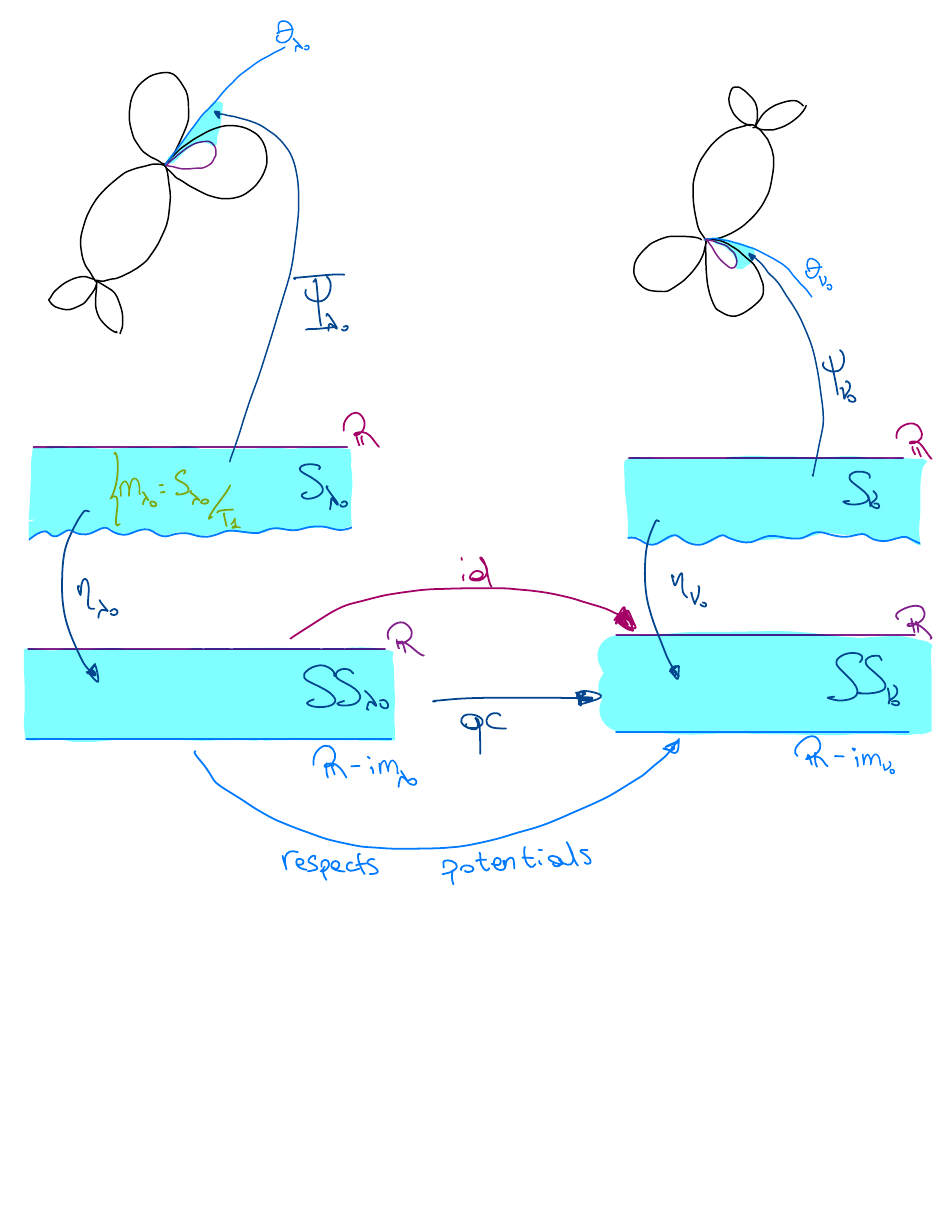}
    \caption{Qc interpolation between 'dynamical cusps' of $f_{\lambda_0}$ and $g_{\nu_0}$.}
    \label{Dyn}
\end{figure}
    From here it is a standard procedure to construct a $T_1$ equivariant q.c.-homeomorphism from $SS\l0$ to $SS\n0$ and hence 
    a $T_1$ equivariant q.c-homeomorphism {\mapfromto{\Gamma}{S\l0}{S\n0}}. The reason for choosing these boundaries and these boundary maps will 
    become come apparent during the course of the proof.\\
    The main differences to the lower angle side are that 
    we normalize $\Phi = \Phi_{\theta'}$ by 
    $\Phi(\cz) = \whResit$, $\cz = \cz(\theta')$ and let 
    $\Psi = \Psi_{\theta'}$ be the repelling Fatou parameter 
    associated with the repelling petal containing the end of the ray of the lower argument $\theta'$, (i.e.~the associated petals precedes the domain of $\Phi$ in the counter clockwise order around $\alpha_\cz$ 
    normalized so that the upper horn map
    $$
    h_u(z) = h(z) := \Phi_{p-1}\circ\Psi(z) = z + \oo(1)
    \quad\textrm{as}\quad \Im(z) \to \infty.
    $$
    Moreover the lifted rays are above the real axis in lower angle case.
    The heights of strips for the lower angle side in general also differ from their upper angle counterparts. We use primed 
    quantities to denote the corresponding map for the lower angle side, 
    i.e.~{\mapfromto{\Gamma'}{S\l0'}{S\n0'}}.

    \subsubsection{Qc map between dynamical and parameter planes}\label{dynpar}
    In this section we shall produce for each value of 
    $\theta_\ilnz, \theta'_\ilnz$, $\ilnz\in\{\lambda_0, \nu_0\}$ quasiconformal maps from the minus infinity end of $S_\theta$ (respectively $S_{\theta'}$) to the cusps in parameter space between the parameter ray $R(\theta)$ and the cardioid (respectively $R(\theta')$ and the cardioid), see \eqref{our_cusp_extension} below.
    Note that pre-composing these maps with $\Psi^{-1}$ gives quasiconformal homeomorphisms between cusps in dynamical plane and cusps is parameter plane.

    Let $L(\Lambda) := (-2\pi i)/q\Lambda$, which evidently is a M{\"o}bius transformation. It exchanges $0$ and $\infty$ and it sends the left half plane onto the upper half plane and the positive imaginary axis onto the negative real axis. Thus the identity map $N(\Lambda) = \Lambda$ on the closure of the left half plane becomes the identity on the closure of the upper half plane when conjugated by $L$.\\
    In the lower angle side, i.e. for $\theta'_\ilnz$ 
    we take $L(\Lambda) := (2\pi i)/q\Lambda$, which conjugates 
    the identity on the left half plane to the identity on the lower half plane.\\
    As above we shall detail only the construction for $\ilnz = \lambda_0$ and for the upper angle side $\theta_\ilnz$ pointing out the difference to the lower angle side.\\
    Let $\wtR_\theta$ and $\wtR_{\theta'}$, starting at $0$ be the lifts of the parameter rays $R^\M(\theta)$ and $R^\M(\theta')$ to $\Lambda$. 
    Let $R_\theta := L(\wtR_\theta)$ and 
    $R_{\theta'} := -L(\wtR_{\theta'})$. 
    Then $R_\theta$ is below the real axis, but meets the real axis at the origin tangentially to the negative imaginary axis and $R_{\theta'}$ is above the real axis and also meets the real axis at the origin, but tangentially to the positive imaginary axis. 
    Let $\Delta_\theta$ be the asymptotically (for $\Re(z)\to-\infty$) strip-like domain bounded by the reals axis and the arc $R_\theta$. 
    Similarly, let $\Delta_{\theta'}$ be the asymptotically (for $\Re(z)\to-\infty$) strip-like domain bounded by the real axis and the arc $R_{\theta'}$. We could have used also $L$ in place of $-L$ for the lower angle side, but since we want to compare with the corresponding dynamical picture our choice here makes visualization easier.\\
    

    Choose an isomorphism $\zeta_\theta: \Delta_\theta\to SS_\theta$, 
    where $SS_\theta$ is the corresponding  dynamical straight strip defined in Section \ref{dc}, such that $\Re(\zeta(z))\to-\infty$ when $\Re(z)\to-\infty$, and $\Re(\zeta(z))\to\infty$ when $z\to 0$. 
    Then $\zeta$ extends continuously to the boundary. 
     We may thus normalize horizontally by requiring 
    that $\zeta$ fixes some numerically large negative number. 
    To avoid confusion, we call here $\zeta_\theta(\Delta_\theta)=SS_\theta^p$.\\ 
    We define boundary maps  
    {\mapfromto{\xi=\xi_\theta}{\partial SS_\theta^p}{\partial SS_\theta}} 
    by 
    $$
    \xi_\theta(z) = \begin{cases}
        \eta_\theta\circ\zeta_\theta^{-1}(z),\qquad z \in \R\\
        \eta_\theta\circ\textrm{``equipotential map''}\circ\zeta_\theta^{-1}(x)), z \in -im_\theta + \R
    \end{cases}
    $$
    Note that $\xi(\R) = ]-\infty, 0[$, (because the upper boundary of $\Delta_\theta$ is the negative real axis) but $\xi$ 
    thus defined is a real analytic diffeomorphism onto the image. 
    In the next section we shall find the maximal complex analytic and in fact univalent extension of the ``equipotential map''.
    
    In the following sections we shall show that, we can interpolate quasiconformally to obtain a quasiconformal extension of $\xi_\theta$ between the $-\infty$ ends of $\overline{SS_\theta^p}$ and $\overline{SS_\theta}$ respectively. 
    From this we obtain a quasiconformal map $\Gamma_\theta$ between the minus infinity end of $S_\theta$ 
    and the cusp between the ray $R_{\theta_\ilnz}$ 
    and the cardioid of $\M$ given by
    \begin{equation}\label{our_cusp_extension}
    c\circ\lambda\circ L^{-1}\circ \zeta^{-1}\circ\xi^{-1}\circ\eta,        
    \end{equation}
    where $\lambda(\Lambda) = \lambda_0\exp(\Lambda/q)$ and 
    $\Lambda(L) = -2\pi i/qL$ (see also \eqref{Xi-definition}).\\
    
    The desired quasiconformal extensions on the $\theta$ and 
    $\theta'$-cusps of 
    $\D(\lambda_0, d)\setminus (W_{p/q}^{\mathcal{N}_0} \cup \D)$ 
    are the maps
    $$
    \Gamma_{\theta(\nu_0)}\circ\Gamma\circ\Gamma_{\theta(\lambda_0)}^{-1}
    \qquad\textrm{and}\qquad \Gamma'_{\theta(\nu_0)}\circ\Gamma'\circ(\Gamma'_{\theta(\lambda_0)})^{-1},  $$
    where we again have used primed quantities for the maps related to $\theta'$.
    By construction we have 
    \begin{align*}
      L^{-1}&\circ\zeta\n0^{-1}\circ\xi\n0^{-1}\circ\eta\n0\circ
      \Gamma\circ\eta\l0^{-1}\circ\xi\l0\circ\zeta\l0\circ L(iy)\\ 
      &= 
      L^{-1}\circ\zeta\n0^{-1}\circ\xi\n0^{-1}\circ\eta\n0
      \circ\Gamma(L(iy)) = L^{-1}(L(iy)) = iy
    \end{align*}
    so that the induced map indeed interpolates between the given boundary values $c\mapsto c(\lambda(c)\nu_0/\lambda_0)$ on the cardioid and 
    the equipotential maps on the rays.
    \begin{figure}[tbh]
    \centering
    \includegraphics[scale = 0.6]{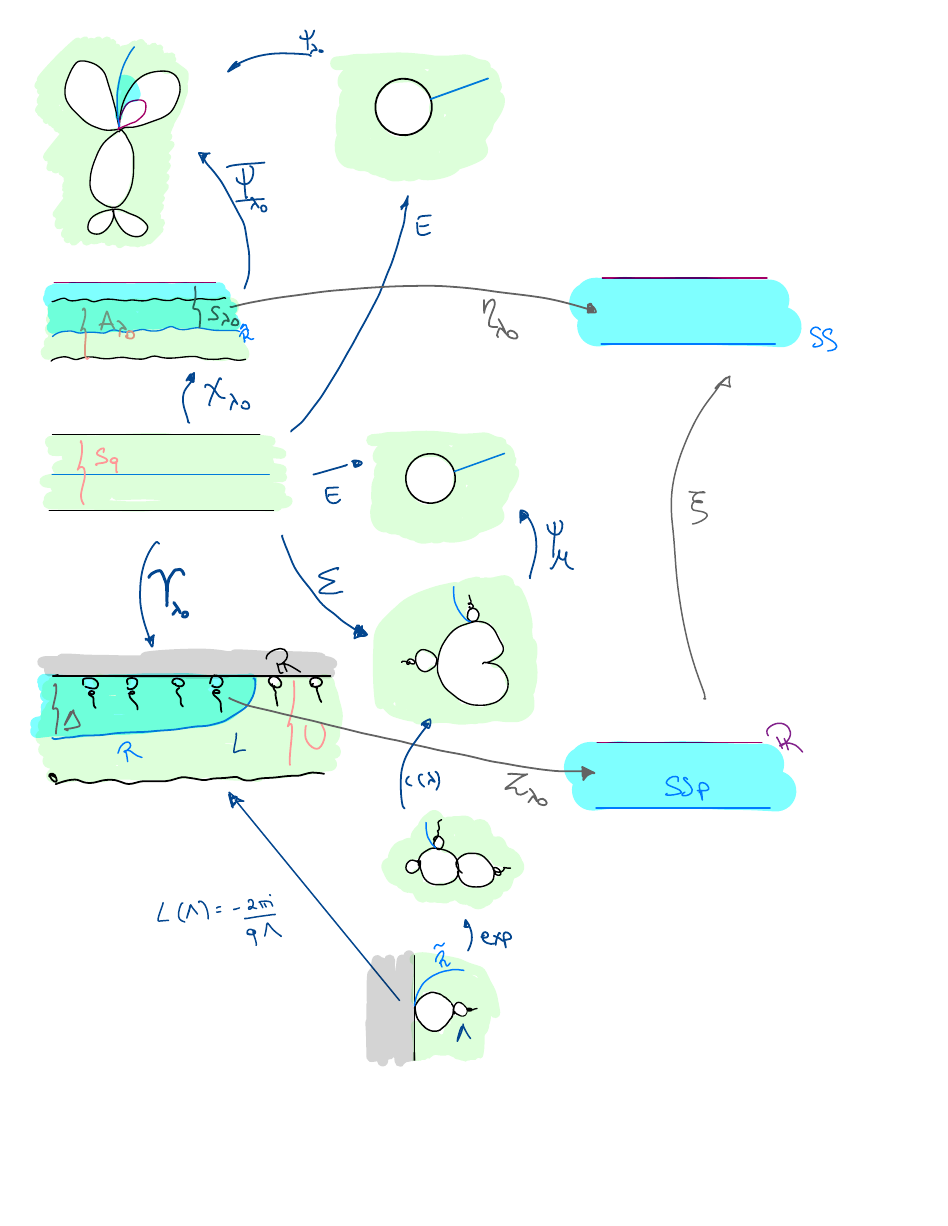}
    \caption{Qc map between dynamical and parameter planes, here $g = \xi^{-1}$.}
    \label{All}
\end{figure}

    In order to make estimates for the boundary maps $\xi$ above, 
    we first produce maximal analytic extensions of the maps defining $\xi$ on the $-\infty$ ends of the boundaries 
    using the theory of covering spaces (see next Section \ref{uncov}). With this we prove Theorem \ref{fim} in Section \ref{longthm}, which implies that the map $\xi_\theta$ admits an asymptotically conformal q.-c.~extension between the minus infinity ends of $SS_\theta^p$ and $SS_\theta$. 
    See Figure \ref{All}.
    \subsubsection{Universal coverings and the functional factorization of $\xi$ on the ray boundary}\label{uncov}
    As above we detail only the case of the upper angle $\theta$ and  $\ilnz = \lambda_0$, and point out differences to the other cases if any, 
    and we will use that $\theta$ determines $\ilnz$. 
    Recall that the repelling Fatou parameter $\Psi_\theta$ was defined in \subsubsecref{dc}. 
    In the following we shall use Schwarz-reflections to obtain maximal complex analytic extensions of the maps involved in defining $\xi$.

    Let $S_q(s) := \{z\;|\;|\Im(z)|< s\pi/(2q\log 2)\}$, $0<s\leq 1$ and $S_q := S_q(1)$
    be the straight strips of height $s\pi/q\log 2$ with core geodesic $\R$. 
    Moreover let $S_q^+ := \{x+iy\in S_q : y>0\}$ 
    and let {\mapfromto{E=E_q}{S_q}{\C\sm\Dbar}} denote the universal covering $E(z) := exp(\e^{zq\log2})$. 
    Then $E_\theta := \e^{i2\pi\theta}E$ maps $\R$ to the straight ray of argument $\theta$ and conjugates $T_1$ 
    to $Q_0^q$. 
    Moreover for $x\in \R$ the point $E_\theta(x)$ is the point on the straight $\theta$-ray of potential $\e^{xq\log 2}$.\\
    For use with Schwarz reflections we also define $S_q^+ := \{x+iy\in S_q : y>0\}$.\\ 
    For $c\in\C$ let $\phi_c$ be the B{\"o}ttcher coordinate for $Q_c$ tangent to the identity at $\infty$ 
    and with maximal round co-domain. 
    The latter equals $\Chat\setminus\Dbar$ if and only if $c\in\M$. By a theorem of Douady and Hubbard the uniformization $\Phi_\M :\Chat\setminus\M \to \Chat\setminus\Dbar$
    tangent to the identity at $\infty$ is given by 
    $\Phi_\M(c) = \phi_c(c)$. 
    Let $\Psi_{\M} = \Phi_\M^{-1}$ denote the inverse.
    Then {\mapfromto{\Sigma = \Sigma_\theta  :=\Psi_\M\circ E_\theta}{S_q}{\C\sm\M}} is a universal covering, 
    which maps $\R$ to the parameter ray $R^\M(\theta)$. 
    Similarly, defining the map {\mapfromto{\Xi = \Xi_\theta}{\C}{\C}}, $\Xi_\theta(L) = c$ 
    by the composition
    \begin{equation}\label{Xi-definition}
    L\mapsto \frac{-2\pi i}{qL} =: \Lambda 
    \mapsto \lambda_0\e^{\Lambda/q} =: \lambda
    \mapsto \frac{\lambda}{2} - \frac{\lambda^2}{4} =: c        
    \end{equation}
    and $U_\theta := \Xi_\theta^{-1}(\C\sm\M)$, the restriction 
    {\mapfromto{\Xi_\theta}{U}{\C\sm\M}} is a universal covering which maps $R_\theta$ to the parameter ray $R^\M(\theta)$. 
    Thus $\Sigma_\theta$ lifts into $\Xi_\theta$ as a biholomorphic map {\mapfromto{\Upsilon_\theta}{S_q}{U_\theta}}, i.e. 
    $\Sigma = \Xi\circ\Upsilon$. 
    For later use we define for $0<s\leq 1$ the restricted subsets $U_\theta(s) := \Upsilon_\theta(S_q(s))$.\\

On the dynamical side we have a similar picture: let   {\mapfromto{\chi_\theta}{S_q}{A_\theta}}
    be the isomorphism such that 
    $\Psi_\theta\circ\chi_\theta = \psi_\cz\circ E_\theta$, 
    where $\psi_{c_0}$ is the inverse of the 
    B{\"o}ttcher-coordinate $\phi_{c_0}$ for $Q_{c_0}$ 
    (recall that $A_\theta = \Psi_\theta^{-1}(\C\sm K_{c_0})$ is 
    a $T_1$ invariant strip on which the restriction 
    {\mapfromto{\Psi_\theta}{A_\theta}{\C\sm K_{c_0}}} of the Fatou parameter $\Psi_\theta$ is a universal covering with 
    $\Psi_\theta(\whR_\theta) = R_\cz(\theta)$), then 
    $\chi_\theta\circ T_1 = T_1\circ\chi_\theta$, 
    since $Q_\cz^q\circ\psi_\cz\circ E_\theta = \psi_\cz\circ E_\theta\circ T_1$ and 
    $Q_\cz^q\circ\Psi_\theta = \Psi_\theta\circ T_1$.\\

    We note that the composition {\mapfromto{\chi_\theta\circ\Upsilon_\theta^{-1}}{U_\theta}{A_\theta}} is a biholomorphic map which extends the 
    equipotential map on the rays.
 
   \begin{teor}\label{parabolic-implosion-thm}
        For any fixed $0<s<1$ 
    $$
    \chi_\theta\circ\Upsilon_\theta^{-1}(L) - L \to 0
    $$ 
    uniformly on $U_\theta(s)$ as $\Re(L)\to-\infty$. 
    \end{teor}
    \begin{proof}
    As in the above discussion we detail the proof for $\ilnz = \lambda_0$ and the upper angle.   
    
    Recall that we defined above the map $\Xi$, with $\Xi(L) = c(\lambda(L))$. 
    
    Let $G_c$ with $G_c(z) = \log|\phi_c(z)|$ on the domain of $\phi_c$ and extended by the functional equation $2G_c(z) = G(q_c(z))$ denote the Green's function for the attracted basin of infinity, $B_c(\infty)$ for $Q_c$. 
    Fix a potential level $\pot>0$, and define the level $\pot$ neighbourhood $\M(\pot)$ of the Mandelbrot set as:
    $$
    \M(\pot) := \{c\in\C | G_c(c)<\pot\}.
    $$
    We have defined $\phi_c$ so that the domain of $\phi_c$ 
    is $\{z| G_c(z)> G_c(0) = G_c(c)/2\}$ and thus the domain 
    of $\psi_c = \phi_c^{-1}$ is $\Chat\setminus\Dbar(r(c))$, 
    where $r(c) = \exp(G_c(0))$. 
    For $c\in\M(\pot)$ the map 
    {\mapfromto{\psi}{\M(\pot)\times\{z : |z| > \exp(\pot/2) \}}{\C}}
    $$
    \psi(c, z) := \psi_c(z)
    $$
    is complex analytic as a function of $(c, z)$ and an isomorphism onto its image for each fixed $c$. 

    Choose $x_0<<0$ such that the Fatou parameter 
    $\Psi = \Psi_\theta$ around $R_{c_0}(\theta)$ 
    defined in \subsubsecref{dc} is univalent on the half plane 
    $\{x+iy| x< x_0\}$, and let $\vartheta^-$ denote the corresponding repelling petal.
    Similarly, let $\vartheta^+$ denote the attracting petal 
    which is mapped univalently onto the right half plane 
    $\{x+iy | x > -1/2-\whResit\}$ by the normalized Fatou-coordinate 
    $\Phi$ (recall that $\Phi$ likewise defined in \subsubsecref{dc} 
    is normalized by $\Phi(\cz) = -\whResit$).\\ 
    Recall that $0< \Re(\e^{zq\log 2}) \leq\e^{xq\log 2}$ for $z=x+iy\in S_q$. 
    Fix $0<s<1$ and choose a potential $\poto>0$ such that 
    $$
    \chi_\theta(\{x+iy\in \overline{S_q(s)} | x \leq \frac{\log(\poto)}{q\log 2}\}) \subset \{x+iy | x< x_0-2\} = \H_{x_0-2}.
    $$
    Then in particular the lifted ray end $\whR_\theta(]0, \poto]) \subset \H_{x_0-2}$. 
    We denote the point $\whR_\theta(\poto) = \chi_\theta(\frac{\log(\poto)}{q\log 2})$ by $Z_1 = Z_{1,\theta}$.

    Choose $x_1 = x_1(s)< 0$ and $b < 0 < t$ such that the half line 
    $(\{x+iy | x \leq x_1, y = b\}) \subset L(H_{p/q})$, 
    where $H_{p/q}$ is the hyperbolic root-component of $\Mpq$. 
    And such that there is a univalent family of Douady-Fatou parameters $\Psi_L = \Psi(L, \cdot)$, 
    for 
    $$L\in PD = PD(s) := \{u+iv | u < x_1(s), b < v < t\}
    $$ 
    defined on vertical strips 
    $\{ x+iy | x(L) < x < x_0\}$ (recall these vertical strips live in the Douady-Fatou parameter plane) and normalized by 
    $\Psi_L(Z_1) = \psi_c(E_\theta(\frac{\log(\poto)}{q\log 2}))$, where $c = \Xi(L)$, 
    i.e. normalized such that $Z_1$ is mapped to the point of potential $\poto$ on the dynamical $\theta$ ray for $Q_c$. 
    Then $\Psi_L \to \Psi$ locally uniformly on $\H_{x_0}$ as 
    $L\in PD, \Re(L) \to -\infty$. 
    Using dynamics to extend the domain of $\Psi_L$ to the left, we can suppose 
    there is some point $Z(L) = Z_\theta(L)$ in the domain of $\Psi_L$, $x(L)+\frac{1}{2}\leq \Re(Z(L))$, such that $\Psi_L(Z(L)) = c$ while $\Psi_L$ remains univalent.  
    Since the set 
    $$
    FD := \{x+iy\in \ov{S_q(s)} |
    \frac{\log(\poto)}{q\log 2} - 2 \leq x \leq \frac{\log(\poto)}{q\log 2}\}
    $$
    is a compact subset of $S_q$ with $FD_0 := \chi_\theta(FD) \subset \H_{x_0-2}$ 
    and the range of $\psi_c$ 
    depends continuously on $c$ we can assume, 
    possibly reducing $x_1$, that $\psi_c\circ E_\theta(FD)$ is contained in the range of $\Psi_L$. 
    Hence we may lift $\psi_c\circ E_\theta$ to $\Psi_L$, i.e. 
    we can define a complex analytical map, injective for each $L$, $\mapfromto{\chi}{PD\times FD}{\C}$ by 
    $$
    \chi(L, z) := \Psi_L^{-1}\circ\psi_c\circ E_\theta
    $$
    where as above $c = \Xi(L)$. 
    Since $\Psi_L$ converge to $\Psi_\theta$ locally uniformly on $\H_{x_0}$ and $\psi_c$ converges uniformly to $\psi_\cz$ as  
    $\Re(L) \to -\infty$, also $\chi_L := \chi(L, \cdot)$ converge to $\chi_\theta = \Psi_\theta^{-1}\circ\psi_\cz\circ E_\theta$ 
    uniformly as $\Re(L) \to -\infty$.
    By construction the map $\chi_L$ commutes with $T_1$, when both sides are defined. 
    This is because both $\psi_c\circ E_\theta$ and $\Psi_L$ conjugates $T_1$ to $f\la$.
    
    Hence for every $L\in PD$ we can extend $\chi_L$ to $\ov{S_q(s)}$ as a homeomorphism holomorphic on the interior and depending complex analytically on $L$. 
    Moreover also the extended $\chi_L$ converges to $\chi_\theta$ 
    uniformly on $\ov{S_q(s)}$ as $\Re(L) \to -\infty$, since it does on a fundamental domain for the action of $T_1$. 

    To complete the proof we recall that {\mapfromto{\Upsilon^{-1}}{U}{S_q}} 
    is the isomorphism which satisfies 
    $$
    \Xi = \Sigma\circ\Upsilon^{-1}
    = \Psi_\M\circ E_\theta\circ\Upsilon^{-1}
    = \psi_{\Xi(L)}\circ E_\theta\circ\Upsilon^{-1}.
    $$
    since $\Phi_\M(c) = \phi_c(c)$.
    And thus that 
    $\psi_{\Xi(L)}\circ E_\theta\circ\Upsilon^{-1}(L) = \Xi(L) = c(L)$
    so that by construction, 
    $$
    \chi(L, \Upsilon^{-1}(L)) = Z(L),
    $$ 
    From here the Theorem follows from the uniform convergence of $\chi_L$ to $\chi_\theta$ 
    on $\ov{S_q(s)}$ as $\Re(L) \to -\infty$ and the following \thmref{Z_of_L_to_L} on holomorphic implosion, which ensures that $Z(L) - L \to 0$ as $L\in PD, \Re(L) \to -\infty$
    (see also \cite{D}, -here in our application the image of any horizontal line in $\ov{S_q(s)}$ plays the role of invariant ray).
    \end{proof}

    We shall extend the use of the term \emph{lifted upper horn map} to any holomorphic map, generically denoted {\mapfromto h W \C}, 
    where $W = T_1(W)$ is a simply connected domain containing an upper half plane, $T_1\circ h = h\circ T_1$ and 
    $h(z) = z + K + \oo(1)$ as $\Im(z)\to + \infty$. 
    Similarly a \emph{lower lifted horn map} $h$ is a holomorphic map such that $-h(-z)$ is an upper lifted horn map. 
    Evidently the composition of lifted upper horn maps is again a lifted upper horn map whose translation part $K$ is the sum of the translation parts of the constituent lifted upper horn maps.
    
    In the $\theta$ case, the $q$ cycle gives rise to lower horn maps, while in the $\theta'$ case the $q$ cycle gives rise to upper horn maps.

\begin{teor}\label{Z_of_L_to_L}
    In the terminology above 
    $$
    Z(L) - L \to 0\qquad\textrm{as}\quad L\in PD, \Re(L) \to -\infty.
    $$
\end{teor}

\begin{proof}
    The idea of the proof is to construct a lifted lower horn map 
    $h_L$ for $L\in PD$ for which $Z(L)$ is a critical value and compare 
    $h_L$ to the lifted lower horn map $h_l$ for $Q_{c_0}$ defined in Equation \eqref{parabolic_lower_horn_map}.

    Note at first that we can holomorphically (but not univalently) extend the Douady-Fatou parameter $\Psi_L$ 
    to the right half plane $\{x+iy\;|\; x(L) < x\}$ using dynamics. 
    We denote by $\wtPsi_L$ the restriction to $\{x+iy\;|\; x_0< x\}$ 
    of the extension of $\Psi_L$.
    Recall from \subsecref{cycle} that for $c = \Xi(L)$ 
    the point $q^0 = q^0_L = Q_c^q(q^0_L)$, 
    is the point which plays the role of $\alpha$-fixed point 
    for the quadratic-like restriction $f_c$ of $Q_c^q$ 
    ($f_\lambda$ in the logistic parametrization) and $\rho = \rho_L$ 
    denotes the multiplier of the $q$-cycle containing $q^0_L$. 
    By construction $\Psi_L(x+iy) \to q^0_L$ as $y\to -\infty$ for any fixed $x > x(L)$. 
    Define lifted lower horn maps $h_L(z) = z + \frac{2\pi i}{\log\rho} + \oo(1) $ by $\Psi_L\circ h_L = \Psi_L$ where both sides are defined 
    (see also \cite{D}). 
    Then $h_L = \Phi_L\circ \wtPsi_L$, where $\Phi_L$ is the inverse of the univalent map $\Psi_L$.

    Since $\Psi_L$ converges locally uniformly to $\Psi_\theta$ 
    it follows from the choice of normalizations (see also \cite{D}) 
    that the normalized lifted lower horn map 
    $$
    -\frac{2\pi i}{\log\rho}+h_L(z) = z + \oo(1)
    $$ 
    converges locally uniformly to the lifted lower horn map $h_l= \Phi\circ\Psi$ for $Q_{c_0}$ (see Equation \ref{parabolic_lower_horn_map} in Section \ref{dc}). 
    That is
    \begin{equation}\label{normalized L-horn-maps-converge}
    -\frac{2\pi i}{\log\rho}+h_L - h_l \to 0\qquad\textrm{locally uniformly}
    \end{equation}
    on the domain $\Psi_\theta^{-1}(B_\cz(\alpha_\cz))$ of $h_l$.
    
    Let $\whPhi_L$ denote the Douady Fatou coordinate, 
    which is the inverse of the univalent map 
    $z\mapsto \Psi_L(z + Z(L)+\whResit)$. 
    Then $\whPhi_L(c) = -\whResit$, where $c = \Xi(L)$ and the co-domain of $\whPhi_L$ contains 
    $\{x+iy| -\frac{1}{2}-\Re(\whResit)<x<x_0-\Re(Z(L)+\whResit)\}$.
    Hence $\whPhi_L$ converges to the normalized attracting parabolic Fatou coordinate $\Phi = \Phi_\cz$ locally uniformly 
    on $\vartheta_0^+$, 
    because $\whPhi_L(c(L)) = -\whResit = \Phi_\cz(\cz)$ and $c(L)\to\cz$ as 
    $\Re(L)\to-\infty$ within $PD$. 

    Denote by $\whh_L = \whPhi_L\circ\Psi_L$ the lower horn map with the same domain as $h_L$, then $\whh_L = -(Z(L)+\whResit) + h_L$
    converges locally uniformly to $h_l$, 
    because $\Psi_L$ converges locally uniformly to $\Psi_\theta$ and $\whPhi_L$ converges locally uniformly to $\Phi_\cz$ as $\Re(L)\to-\infty$. 
    Since both maps 
    $\whh_L = -(Z(L)+\whResit)+h_L$ and $-\frac{2\pi i}{\log\rho}+h_L$ converge locally uniformly to $h_l$ we obtain
    $$
    Z(L)+\whResit-\frac{i2\pi}{\log\rho} \to 0\qquad\textrm{as}\qquad
    \Re(L) \to -\infty.
    $$
    Finally, as Equation\eqref{H_assymptotics} in Section \ref{cycle} can be written as 
    $$
    L - \frac{2\pi i}{\log\rho(L)} \to -\whResit\qquad\textrm{as}\qquad
    \Re(L) \to -\infty.
    $$
    
    we have 
    $$Z(L)- L \to 0\qquad\textrm{as}\qquad
    \Re(L) \to -\infty.
    $$
    
\end{proof}

\subsubsection{The map $\xi^{-1}$ on the boundaries}\label{longthm}
    The aim of this Section is to prove the following estimates:
    \begin{teor}\label{fim}
    For each value of $\theta$ (and $\theta'$) the uniformizing maps $\eta=\eta_\theta$ and $\zeta = \zeta_\theta$ satisfy
        $$
    \zeta(z-n) - \zeta(-n) \to \eta(z-n) - \eta(-n) = \eta(z-n) + n = \eta(z)
        $$
        and as a consequence 
        $$
        \xi(x-ih) - \xi(x) \to -ih
        \qquad\textrm{as}\qquad x\to -\infty
        $$
    on the upper angle side and 
        $$
        \xi(x+ih) - \xi(x) \to ih
        \qquad\textrm{as}\qquad x\to -\infty
        $$
    on the lower angle side
  \end{teor}
  \begin{cor}
      For each value of $\theta$ ($\theta'$) the boundary value map 
      $\xi = \xi_\theta$ has an asymptotically conformal q.c. extension between the minus infinity ends.
  \end{cor}
\begin{proof}(of \thmref{fim})
    To alleviate notation, we omit as above the subscript $\theta$ and we consider only the upper $\theta$ side, leaving the similar lower 
    $\theta'$ side to the reader.\\ 
    Extend the domains of $\eta$ and 
    $\zeta$ by reflections in the ray boundary arcs, using the map $\chi$ in dynamical space and 
    $\Upsilon$ in the $L$ parameter space, 
    and by ordinary reflections in the $\R$ and the
    $]-\infty, 0[$ parts of the boundaries.
    More precisely, let {\mapfromto{\tau}{\C}{\C}} 
    denote the reflection in $\R$, $\tau(x+iy) = x-iy$. 
    Set
     $$W(0) = S\cup\R\cup\tau(S),\qquad\qquad 
     SW(0) = SS\cup\R\cup\tau(SS),$$ 
     (recall that $S=S_\theta$)
     and for $0<s\leq 1$ set $$W(s)= S\cup \chi(S_q(s)) \cup \R \cup \tau(S\cup \chi(S_q(s)))$$
    and define $W=W(1)$. 
    The map $\eta$ extends by Schwarz reflection in the domain boundary arc $\whR$ 
    of $A$ and in the boundary arc $-im+\R$ of $SS$ to a univalent map 
    with $\eta\circ\chi(\R) = -im+\R$. Extending $\eta$ further by reflection in $\R$ 
    in both domain and co-domain we obtain a univalent $\tau$ symmetric map 
    also denoted $\eta$, {\mapfromto {\eta} {W} {SW:=\eta(W)}}.
    For $0\leq s \leq 1$ define 
    $$
    SW(s) := \eta(W(s)).
    $$
    Note that $T_1(W(s)) = W(s)$ and $T_1(SW(s)) = SW(s)$ for all $s$. 

    Then $\ov{W(s)} \subset W(s')$ and $\ov{SW(s)} \subset SW(s')$
    for any $0\leq s < s'\leq 1$. 
    Define similarly 
    $$V(0) = \Delta\,\cup \,\,]-\infty,0[\,\,\cup\,\tau(\Delta),\qquad 
    SV(0) = SS^p\cup\R\cup\tau(SS^p),
    $$ and for $0<s\leq 1$ $V(s), \ov{V(s)}\subset V=V(1)$. 
    Then $\ov{V(s)} \subset V(s')\cup\{0\}$ and $\ov{SV(s)} \subset SV(s')$
    for any $0\leq s < s'\leq 1$. 
    We obtain by Schwarz reflection extended isomorphisms 
    for any $0\leq s < s'\leq 1$
    $${\mapfromto{\eta}{W(s)}{SW(s)}},$$
    $${\mapfromto{\zeta}{V(s)}{SV(s)}}.$$
    Since $\chi$ and $\eta$ commute with $T_1$, we have $W(s) = T_1(W(s))$. 
    As $SS=SS^p$, we have by construction $$SV(0) = SW(0) = SS\cup\R\cup\tau(SS) := SSS.$$ 
    Moreover the biholomorphic map 
    $$
    \mapfromto{\zeta^{-1}\circ\eta}{W(0)}{V(0)}
    $$
    extends for any $0<s\leq 1$ to a bilomorphic map 
    $$
    \mapfromto{\zeta^{-1}\circ\eta}{W(s)}{V(s)}
    $$
    by Schwarz reflections in the boundary arcs using the biholomorphic map 
    {\mapfromto{\Upsilon\circ\chi^{-1}}{A}{U}} and its reflection in $\R$.
    

    
    By \thmref{parabolic-implosion-thm} we have 
    Caratheodory convergence of the sequences of pointed domains 
    $$
    (V(s)+n, 0) \longrightarrow (W(s)+n, 0) = (W(s), 0)
    $$ 
    Since the restrictions {\mapfromto{\eta^{-1}}{SSS}{W(0)}} and 
    {\mapfromto{\zeta^{-1}}{SSS}{V(0)}} are uniformisations of 
    $W(0)$ and $V(0)$ by the same straight strip $SSS$, 
    the Caratheodory convergence above for $s=0$ implies that 
    the Riemann maps from $V(0)+n$ to $SSS$ fixing $0$ with positive derivative converges to the Riemann map from 
    $W(0)+n= W(0)$ to $SSS$ fixing $0$ with positive derivative, hence (recall that $\eta\circ T_1 = T_1\circ\eta$ and $\eta(0) = 0$, 
    so $\eta(-n) = -n$
    \begin{align}
    \zeta(z-n) - \zeta(-n) &\to \eta(z-n) - \eta(-n)\\
    \label{eta-zeta-convergence}
    &= \eta(z-n) + n = \eta(z)
    \end{align}
    as $n\to\infty$ locally uniformly on $W(0) = W(0) + n$. 
    Moreover, for any $0\leq s < s' \leq 1$, \thmref{parabolic-implosion-thm} 
    also implies that there exists $x(s,s') < 0$ such that 
    $$
    \{z\in \ov{V(s)}\,\,\, |\,\,\, \Re(z) \leq x(s,s')\} \subset 
    \{z\in W(s')\,\,\, | \,\,\,\Re(z) \leq x(s,s')\} 
    $$
    and 
    $$
    \{z\in \ov{W(s)}\,\,\, |\,\,\,\Re(z) \leq x(s,s')\} \subset 
    \{z\in V(s')\,\,\, |\,\,\,\Re(z) \leq x(s,s')\}. 
    $$
    In particular, for $s>0$, say $s=1/2$, both $\eta$ and $\zeta$ are defined on 
    $\{z\in W(s) | \Re(z) \leq x(s,s')\}$.
    Since 
    $$\zeta(z-n) - \zeta(-n) \to \eta(z-n) - \eta(-n) = \eta(z), 
    \mbox{ as }n\to\infty$$ locally uniformly on 
    $W(0)$ for $|\Re(z)|\leq 1$, it also holds locally uniformly on $W(s)$ for 
    $|\Re(z)|\leq 1$. 
    It follows that $\zeta$ is uniformly continuous on, say, $W(s/2)$ for $\Re(z) \leq x(s/2,s)$ (this will be useful later). 

    On $\R$ the boundary value map $\xi^{-1}$ is given by 
    $\zeta\circ\eta^{-1}$. 
    For $x< 1$ we find
    \begin{align*}
        \xi^{-1}(x-n)-\zeta(-n) &= \zeta(\eta^{-1}(x-n))-\zeta(-n)\\
        &= \zeta(\eta^{-1}(x)-n) - \zeta(-n)\\ 
        &\to \eta(\eta^{-1}(x)) = x
    \end{align*}
    as $n\to\infty$, with the limit following from \eqref{eta-zeta-convergence}.\\
    
    On $\R-im$ the boundary value map $\xi^{-1}$ is given by 
    $$
    \xi^{-1}=\zeta\circ\Upsilon\circ\chi^{-1}\circ\eta^{-1}
    $$
    and 
    $\Upsilon\circ\chi^{-1}(z)-z\to 0$ $z\in\whR$ as 
    $\Re(z) \to -\infty$.\\ 
    We compute $\xi^{-1}(x-im - n) - \zeta(-n)$, similarly to the computation above
\begin{align*}
    \xi^{-1}(x-im - n) - \zeta(-n) &= 
    \zeta\circ\Upsilon\circ\chi^{-1}\circ\eta^{-1}(x-im - n) - \zeta(-n)\\
    &= \zeta\circ\Upsilon\circ\chi^{-1}(\eta^{-1}(x-im) - n) - \zeta(-n)\\
    &\simeq \zeta(\eta^{-1}(x-im) - n) - \zeta(-n)\\
    &\to \eta(\eta^{-1}(x-im)) = x - im
\end{align*}
    as $n\to\infty$. The uniform continuity of $\zeta$ implies the $\simeq$ above, since $\Upsilon\circ\chi^{-1}(z)-z\to 0$ 
    for $z\in\whR$ as $\Re(z) \to -\infty$.
    Hence 
    $$
    \xi(x-im - n) - \xi(x-n) \to -im.
    $$
    as $n\to\infty$. Since the intervals $[-\frac{1}{2} -n, \frac{1}{2} -n]$ 
    and $[-\frac{1}{2} -(n+1), \frac{1}{2} -(n+1)]$ overlap we can also write the above as
    $$
    \xi(x-im) - \xi(x) \to -im,\quad\textrm{as}\quad x\to-\infty.
    $$
\end{proof}
Note that this ends the proof of Theorem \ref{nbh}.

\end{document}